\documentclass[11pt]{amsart}

\usepackage{amsmath,amssymb,amsthm}
\usepackage{amsmath,amssymb,amsthm,amscd}
\numberwithin{equation}{section}
\usepackage{color}
\usepackage[mathscr]{eucal}
\allowdisplaybreaks[4]

\newtheorem{prop}{Proposition}
\newtheorem{ques}[prop]{Question}
\newtheorem{theo}[prop]{Theorem}
\newtheorem{lemm}[prop]{Lemma}

\newtheorem{rema}[prop]{Remark}

 \DeclareMathOperator{\Tr}{Tr}

 \DeclareMathOperator{\vol}{vol}
\DeclareMathOperator{\diam}{diam}

\def\g{\mathfrak g}

\def\h{{\mathfrak h}}

\def\cF{{\mathcal F}}

\def\cN{{\mathcal N}}
\def\cO{{\mathcal O}}
\def\cP{{\mathcal P}}

\def\cU{{\mathcal U}}

\def\bh{{\mathbf h}}

\def\RR{{\mathbb R}}

\def\and{\quad{\rm and}\quad}

\def\bl{\bigl(}
\def\br{\bigr)}

\let\lra=\longrightarrow

\def\lalp{_{\alpha}}

\def\lep{_{\epsilon}}

\def\mh{\!:\!}

\def\sub{\subset}
\def\sta{^{\ast}}

\def\lab{\ }
\usepackage{amscd}
\def\lab{\label}

\definecolor{light-pink}{rgb}{1,.90,.90}
\definecolor{light-green}{rgb}{.95,1,.95}

\definecolor{yellow}{rgb}{1,1,0}
\definecolor{orange}{rgb}{1,.7,0}
\definecolor{red}{rgb}{1,0.3,0.9}
\definecolor{white}{rgb}{1,1,1}

\definecolor{A}{rgb}{.75,1,.75}

\title{Limiting Behavior of a Class of  Hermitian Yang-Mills Metrics}
\author{Jixiang Fu}
\address{Institute of Mathematics\\ Fudan University \\ Shanghai 200433\\ China}
\email{majxfu@fudan.edu.cn}

\begin{document}

\maketitle
\begin{abstract}
This paper begins to study the limiting behavior of a family of Hermitian Yang-Mills (HYM for brevity) metrics on a
class of rank two slope stable vector bundles over a product of two
elliptic curves with K\"ahler metrics $\omega_\epsilon$ when $\epsilon\to 0$. Here $\omega_\epsilon$
are flat and have areas $\epsilon$ and $\epsilon^{-1}$ on the two
elliptic curves respectively. A family of Hermitian metrics on the vector bundle are explicitly  constructed and with respect to them, the HYM metrics are normalized.
We then compare the family of normalized HYM metrics with the family of constructed Hermitian metrics by doing estimates. We get the higher order estimates as long as the $C^0$-estimate is provided.
We also get the estimate of the lower bound of the $C^0$-norm.
If the desired estimate of the upper bound of the $C^0$-norm can be obtained, then it would be  shown that these two families of metrics
are close to arbitrary order in $\epsilon$ in any $C^k$ norms.
\end{abstract}

\tableofcontents

\section{Introduction}
A Calabi-Yau manifold is a compact K\"ahler manifold with zero first
Chern class. Yau's solution
\cite{Yau78} to the Calabi conjecture provides a unique Ricci-flat
K\"ahler metric in each K\"ahler class of a Calabi-Yau manifold.
Motivated by mirror symmetry and the Strominger-Yau-Zaslow conjecture \cite{SYZ},
Gross and Wilson \cite{Gr} initiated the study of the limiting behavior of
Yau's Ricci flat metrics in a large complex structure limit.
They showed that a family of Ricci flat metrics on a
general K3 surface, which is a hyper-K\"ahler rotation of an
elliptic K3 surface with 24 singular fibers, converge (collapse) to
a metric on the base $S^2$ with singularities on the discriminant
locus of 24 points. Many other investigations of this topic have
appeared in the literature  \cite{Wi,Zh,LYZ0,To,RZ,GTZ}.

In this paper, we will study the Hermitian Yang-Mills (HYM for
brevity) version of the above problem. Let $V$ be a slope stable
holomorphic vector bundle over a compact K\"ahler manifold $X$ with
a K\"ahler metric (form) $\omega$. According to the Donaldson-Uhlenbeck-Yau
theorem \cite{Do,Do2,UY}, $V$ admits a unique irreducible HYM metric $H$
up to a positive multiplicative constant. Suppose that $X$ is a
Calabi-Yau manifold with a family of K\"ahler metrics
$\omega_\epsilon$ approaching a large K\"ahler metric limit, and
assume that $V$ is slope stable with respect to each
$\omega_\epsilon$. Then we obtain a family of HYM metrics
$H_\epsilon$.

\begin{ques} After normalization, what is the limiting behavior of
$H_\epsilon$ when $\omega_\epsilon$ goes to a large K\"ahler metric
limit?
\end{ques}

This natural question will be studied in detail in this paper for a
specific K\"ahler manifold $X$, i.e.,
 the product $B\times T$
of two copies of the complex one-torus  $\mathbb C/\Gamma$, where
$\Gamma=\mathbb Z+i\mathbb Z$. In this case, a family of product
metrics $\omega_\epsilon$, which are flat and  have areas $\epsilon$
and $\epsilon^{-1}$ on  $T$ and $B$ respectively, approaches  a
large K\"ahler metric limit when $\epsilon\to 0$ (cf. \cite{Le}).

The holomorphic vector bundle $V$ over $X$ considered here is
constructed as follows (cf. \cite{Fr,FMW}). Let $T^\ast$ be the dual
of $T$ and let $X^\ast=T^\ast\times B$.  Let  $Y$ be a compact
(complex) curve of $X^\ast$ such that the induced projection
$\varphi:Y\rightarrow B$ is a two-sheet branched cover with $n$ branched points. Denote the other induced map by $q: Y\to T^\ast$.
Denote
\begin{equation*}\label{171221}
\iota=(q,\textup{id}_{T}):Y\times T \lra T^\ast\times T,
\qquad p_2=(\varphi,\textup{id}_T):Y\times T\lra X
\end{equation*}
and denote by $p_1$ the projection map from $Y\times T$ to $Y$.
Let $\cP$ be the Poincar\'e
line bundle on $T^\ast\times T$.
 Then for any
degree zero line bundle $\cF$ over $Y$, we can form a line bundle over $Y$
\begin{equation*}
\cN=K^{1/2}_Y\otimes \varphi\sta K_B^{-1/2}\otimes
\cF
\end{equation*}
and a rank two vector bundle over $ X$ with zero degree
\begin{equation*}
V=p_{2\ast}(\iota^*\cP\otimes p_1^*\cN).
\end{equation*}
By an adiabatic argument (cf. \cite{FMW}), $V$ is $\omega_\epsilon$-slope stable for
small $\epsilon$. Hence  there exists  a family of irreducible
HYM metrics $H_{1,\epsilon}$ on $V$ with respect to
$\omega_\epsilon$. As a consequence of $c_1(V)=0$,
  the associated curvature forms $\Theta(H_{1,\epsilon})$ satisfy
\begin{equation*}\label{06}
\Lambda_{\omega_\epsilon}\Theta(H_{1,\epsilon})=0.
\end{equation*}
The definition of the trace operator $\Lambda_{\omega_\epsilon}$ will be recalled in (\ref{a4}).

The purpose of this paper is to investigate Question 1 for
 $H_{1,\epsilon}$  when $\epsilon\to 0$.
In Section 5, a family of Hermitian metrics $H_{0,\epsilon}$ on $V$
is explicitly constructed  such that the following result holds.

\begin{theo}\label{t0}
For any nonnegative integer $k$ and positive integer $l$, there is a
constant $C$ depending on $k$ and  $l$ such that for any sufficiently small $\epsilon>0$, the associated
curvatures $\Theta(H_{0,\epsilon})$ of  $H_{0,\epsilon}$ satisfy
\begin{equation*}
\parallel\!\Lambda_{\omega_\epsilon}\Theta(H_{0,\epsilon})\!\parallel_{C^k}\leq
C\epsilon ^{l}.
\end{equation*}
\end{theo}

Moreover, the curvatures $\Theta(H_{0,\epsilon})$ of $H_{0,\epsilon}$ satisfy
\begin{equation}\label{1901002}
\textup{Tr}\Lambda_{\omega_\epsilon}\Theta(H_{0,\epsilon})=0.
\end{equation}
Since $H_{1,\epsilon}$ and $H_{0,\epsilon}$ are Hermitian metrics on $V$, there exists a smooth section
$H_\epsilon$ of $\textup{End}(V)$, the endomorphism bundle of $V$,  such that
\begin{equation*}
H_{1,\epsilon}(\cdot\,,\cdot)=H_{0,\epsilon}(H_\epsilon\cdot\,,\cdot).
\end{equation*}
Equation (\ref{1901002}) guarantees that $\det H_\epsilon$ is a constant. We normalize $H_{1,\epsilon}$ so that $\det H_\epsilon=1$.
We then compare $H_{1,\epsilon}$ and $H_{0,\epsilon}$ by comparing $H_\epsilon$ and $\textup{Id}$, the identity section of $\textup{End}(V)$.
We should estimate $\parallel\! H_\epsilon-\textup{Id}\!\parallel_{C^k}$.

For $k\geq 1$, we have the following results.

\begin{theo}\label{1901001}
Fix a positive integer $k$ and an integer $l\geq 3k+\frac {15} 2$. Assume that
there exists a constant $C$ depending on $l$ such that for any sufficiently small $\epsilon>0$,
\begin{equation}\label{1901003}
\parallel\!H_\epsilon-\textup{Id}\!\parallel_{C^0}\leq C\epsilon^l.
\end{equation}
Then there exists another constant, which is still denoted by $C$, depending on $k$ and $l$ such that for any sufficiently small $\epsilon>0$,
\begin{equation*}
\parallel\!H_\epsilon-\textup{Id}\!\parallel_{C^k}\leq  C\epsilon^{l-3k-\frac{15}2}.
\end{equation*}
\end{theo}

\begin{rema}
The above  $\Lambda_{\omega_\epsilon}\Theta(H_{0,\epsilon})$ and
$H_\epsilon$ lie in
$A^0(\textup{End}(V))$, the space of $C^\infty$ sections of
$\text{End}(V)$, where there is no natural $C^k$ norm. We use
$H_{0,\epsilon}$ to define a $C^k$ norm, i.e., for a local
trivialization of $V$, we choose a unitary frame relative to
$H_{0,\epsilon}$ and define a $C^k$ norm on $A^0(\textup{End}(V))$
to be the $C^k$ norm  of the resulting matrix representations.
 The $C^k$ norm of a function is defined as
in \cite[p.53]{GT} which does not depend on $\epsilon$. Hence, if inequality (\ref{1901003}) holds, then the
metrics $H_{1,\epsilon}$ and $H_{0,\epsilon}$ are close
 to  arbitrary order in $\epsilon$ in any $C^k$ norms.
\end{rema}

The $C^0$-estimate (\ref{1901003}) is very hard because in general
the maximum principle for elliptic partial differential systems does not hold.
As $\det H_\epsilon=1$ and $H_\epsilon$ is Hermitian symmetric, we only need to estimate the upper bound of $\textup{Tr} H_\epsilon$.
It is well-known (cf. \cite[P. 24]{Siu}) that from the HYM system
\begin{equation*}
-\bigtriangleup_{\omega_\epsilon}\ln\textup{Tr}H_\epsilon\leq 4\parallel\!\Lambda_{\omega_\epsilon}\Theta(H_{0,\epsilon})\!\parallel_{C^0},
\end{equation*}
which, combined with the inequality  in Theorem \ref{t0}, implies
\begin{equation}\label{1901005}
-\bigtriangleup_{\omega_\epsilon}\ln\textup{Tr}H_\epsilon\leq C\epsilon^l.
\end{equation}
Clearly we can not get the upper bound of $\textup{Tr}H_\epsilon$ from the above  inequality.
However, as the first step of the $C^0$-estimate, we obtain the following lower bound of $\textup{Tr} H_\epsilon$.

\begin{theo}\label{1901004}
For any positive integer $l$ and sufficiently small $\epsilon>0$, there exists a constant $C$ depending on $l$ such that
\begin{equation}\label{1901006}
\inf_{x\in X}\textup{Tr}H_\epsilon(x)\leq 2+C\epsilon^l.
\end{equation}
\end{theo}

\vspace{3mm}

It is believed that the method in this paper can be applied to  other cases such as
the elliptic  $K3$ surface if one knows much more about its large
K\"ahler metric limit.
Since the Poincar\'e line bundle is used in
the construction of the vector bundle, it may have many potential
applications to mirror symmetry (cf. \cite{FMW,Fu0,Fu,Kon,KS,LYZ,Th,TY,Wit}).

\vskip3pt


We give the structure of the paper as follows.  First, we  localize $V$ in Section 2. Then we
use such a localization of $V$ to construct a family of Hermitian metrics $H_{0,\epsilon}$
in Sections 3, 4 and 5.
The key
step is to construct a
family of HYM metrics on $V$ over the product of $T$ and a
neighborhood of a branched point in $B$. In Section 3, we construct
such metrics (\ref{414}) and so derive a PDE (\ref{417}) involving $\epsilon$.
Hence, in Section 4 we consider the boundary value problem (\ref{418}) of PDE (\ref{417}).
This equation has  a unique smooth solution $u_\epsilon$
and  also a singular solution $\frac 1 2\ln r$. Moreover, according
to Gidas-Ni-Nirenberg's theorem in \cite{Gi}, it can be reduced to
an ODE (\ref{420}) on the interval $[0,2r_0]$ which is a singular
perturbed equation with  small parameter  $\epsilon$. We estimate
the $C^k$ norm of $u_\epsilon-\frac 1 2 \ln r$ on $[r_0,2r_0]$. It is
found that they are close to  arbitrary order in $\epsilon$ in any
$C^k$ norms.

In Section 5, we use the Green function of a
degree zero divisor on $B$
 to construct a HYM
metric on $V$, which is singular on $V$ over the product of $T$ and
each branched point. However, this singular metric is essentially
the same as the metrics (\ref{414}) when  the PDE (\ref{418}) takes
the {\sl singular} solution. Hence, this metric can be glued to the
local {\sl smooth}  HYM metrics (\ref{414}). The resulting metrics
can be normalized conformally to  a family of Hermitian metrics
$H_{0,\epsilon}$ so that equation (\ref{1901002}) holds.
Then by the estimates in Section 4, it is easy to prove Theorem \ref{t0}.

In Section 6, we first normalize $H_{1,\epsilon}$ so that $\det H_\epsilon=1$.
Then from inequality (\ref{1901005}) we use the Morse iteration to prove
\begin{equation*}
\sup_{x\in X}\ln \textup{Tr}H_\epsilon(x)\leq (1+C\epsilon^l)\int_X\ln\textup{Tr}H_\epsilon(x)\frac{\omega_\epsilon^2}{2!}.
\end{equation*}
In Section 7, we prove Theorem \ref{1901004}. In fact, the identity  $\det H_\epsilon=1$ implies $\textup{Tr}H_\epsilon\geq 2$.
If $\inf_{x\in X}\textup{Tr}H_\epsilon(x)>2$, then at any point one eigenvalue of $H_\epsilon$ is bigger than one and the other is smaller than one.
Hence the eigenvectors of $H_\epsilon$ form two (complex) subline bundles of $V$. We analyses the relations between the curvatures of the subline bundles and $V$ to obtain inequality (\ref{1901006}).
For the higher order estimates, in Section 8 a new version (\ref{1209065}) of the HYM system is derived.
Then we can reach the goal by using the Gagliardo-Nirenberg inequality (c.f. \cite{Ni}).

\vspace{3mm}

{\bfseries Acknowledgements.} The draft of the first 6 sections was finished in 2002 with the help of Professor Jun Li. The author would like to thank
J. Li for discussions on algebraic geometry and Professors Jiaxing Hong
and Shing-Tung Yau on PDEs.
He would also like to thank Professors Guofang Wang, Qingxue Wang,
Weiping Zhang and Xi Zhang  for useful discussions.
This work is partially
supported  by NSFC grants 11871016, 11421061 and 11025103.

\section{A localization of $V$}
In this section, the basically geometric set-up will be described.
Let $\Gamma=\mathbb Z+i\mathbb Z$ and $\Gamma^\ast$ be the dual of
$\Gamma$. Let $T$ and $B$ be two copies of the complex one-torus
$\mathbb C/\Gamma$ and let $X=B\times T$. Let $T^\ast=\mathbb
C^\ast/\Gamma^\ast$ be the dual of $T$ and $X^\ast=T^\ast\times B$.
Set $z=x_1+ix_2$, $w=y_1+iy_2$, and $w^\ast=y^\ast_1+iy_2^\ast$ as
the complex coordinates of $B$, $T$, and $T^\ast$, respectively. We
endow $X$ with a family of K\"ahler metrics
\begin{equation}\label{403}
{\omega}_\epsilon=\frac i 2\epsilon^{-1}  dz\wedge d\overline z+\frac i 2\epsilon dw\wedge d\overline w.
\end{equation}
Hence, the volume forms $\frac{\omega^2_\epsilon}2$ are independent
of $\epsilon$.

Regarding $\Gamma$ as the transformation group of $\mathbb C$ and $\Gamma^\ast$
as the transformation group of $\mathbb
C^\ast$, $\mathbb C^\ast\times\mathbb C$ becomes the universal cover
of $T^\ast\times T$ with the deck transformation group
$\Gamma^\ast\times\Gamma$, which acts on $\mathbb C^\ast\times \mathbb C$ as
\begin{equation*}
\g_{(\gamma^\ast,\gamma)}(w^\ast,w)=(w^\ast+\gamma^\ast,w+\gamma).
\end{equation*}
After this, we recall the construction of the Poincar\'e line bundle (cf.
\cite{GH}).

Let us start with the trivial line bundle $\tilde
{\mathcal P}$ over $\mathbb C^\ast\times T$ with the standard flat
connection along $\mathbb C^\ast$ and with the connection  form
along $T$ at $\{w^\ast\}\times T$:
\begin{equation*}\label{404}
\theta=-\pi i(w^\ast d\overline{w}+\overline{w^\ast}dw).
\end{equation*}
We can lift the $\Gamma^\ast$ action on $\mathbb C^\ast$ to
$\tilde\cP$. If the constant one global section on $\tilde P$ is
denoted by $\varepsilon_{(w^\ast,w)}$, then
\begin{equation*}\label{405}
\g_{(\gamma^\ast,0)}\sta
\varepsilon_{(w^\ast+\gamma^\ast,w)} =\exp(-\pi
i(\gamma^\ast\overline{w}+\overline {\gamma^\ast}
w))\varepsilon_{(w^\ast,w)}.
\end{equation*}
Thus, $\tilde {\mathcal P}$ can be reduced to a line bundle
$\mathcal P$ over $T^\ast\times T$, which is called the Poincar\'e
line bundle.

The curvature form of $\theta$ is
\begin{equation}\label{f0}
\Theta=-\pi i(dw^\ast\wedge d\overline w+d\overline {w^\ast}\wedge dw),
\end{equation}
which turns out to be a $(1,1)$-form on $T^\ast\times T$. This makes
$\cP$ a holomorphic line bundle with a holomorphic frame
\begin{equation}\label{406}
\tilde{\varepsilon}_{(w^\ast,w)}=\exp(\pi i
w^\ast\overline{w})\varepsilon_{(w^\ast,w)}.
\end{equation}
It transforms under $\Gamma^\ast\times\Gamma$ via
\begin{eqnarray*}\lab{120801}
\begin{aligned}
&\g_{(0,\gamma)}^\ast\tilde{\varepsilon}_{(w^\ast,w+\gamma)}=\exp(\pi
iw^\ast\overline\gamma)\tilde \varepsilon_{(w^\ast,w)},\\
 &\g_{(\gamma^\ast,0)}\sta\tilde{\varepsilon}_{(w^\ast+\gamma^\ast,w)}
=\exp(-\pi i \overline{\gamma^\ast}w)\tilde{\varepsilon}_{(w^\ast,w)}.
\end{aligned}
\end{eqnarray*}
By (\ref{f0}), its first Chern
class $c_1(\mathcal P)$ is represented by (the first Chern form:)
\begin{equation}\lab{g2}
C_1(\mathcal P)=\frac{-\Theta}{2\pi i}=\frac 1 2(dw^\ast\wedge
d\overline w+d\overline {w^\ast}\wedge dw).
\end{equation}

As in Section 1, we take a (complex) curve $Y$ in $X^\ast$ so that
the induced map $\varphi:Y\to B$ is a two-sheet branched cover with $n$ branched points.
Denote   the other induced map by $q:Y\to T^\ast$.
Then as in Section 1, we can use $Y$ and $\mathcal P$  to construct the
rank two vector bundle $V$ over $X$. According to Section 7 in \cite{FMW},
we  have
\begin{eqnarray}\label{g3}
\begin{aligned}
&c_1(V)=p_{2\ast}\bigl(\iota^\ast c_1(\mathcal P)\bigr),\\
&c_2(V)=\frac 1 2 c_1(V)^2-\frac 1 2 p_{2\ast}\bigl(\iota^\ast c_1(\mathcal
P)^2\bigr).
\end{aligned}
\end{eqnarray}
Then as discussion in \cite{FMW}, $c_1(V)=0$, and hence
\begin{equation}\label{10}
\int_X c_2(V)=\deg q
\end{equation}
which can be derived as follows:
\begin{equation*}
\begin{aligned}
\int_Xc_2(V)=&-\frac 12\int_X p_{2\ast}(\iota^\ast c_1(\mathcal P)^2)\qquad\quad \textup{by (\ref{g3})}\\
=&-\frac 12\int_{Y\times T}\iota^\ast c_1(\mathcal P)^2\qquad\quad  \textup{by the definition of $p_{2\ast}$}\\
=&\bigl(\frac i 2\bigr)^2\int_{Y\times T}\iota^\ast(dw^\ast\wedge d\overline{w^\ast}\wedge dw\wedge d\overline w)\qquad\quad  \textup{by (\ref{g2})}\\
=&\frac i 2 \int_Y q^\ast(dw^\ast\wedge d\overline{w^\ast})\qquad\quad \textup{by the definition of $\iota$}\\
=&\deg q\qquad\quad \textup{by the definition of the degree}.
\end{aligned}
\end{equation*}
We can also get the same results in Section 7 by using the curvature forms directly.

Next we should simplify $V$. Let
\begin{equation*}
D_0=\sum_{a=1}^{n}\xi_a
\end{equation*}
be the branched locus on $B$. By the Riemann-Hurwitz formula, the
genus $g(Y)$ of $Y$ is bigger than 1 and $n=2(g(Y)-1)$. Since the
degree of $K_Y$ is $2\bigl(g(Y)-1\bigr)$ and the degree of $K_B$ is
0, we have
\begin{equation*}
\deg\bigl(K_Y^{1/ 2}\otimes\varphi^\ast K_B^{-1/ 2}\bigr)=g(Y)-1=\frac n 2.
\end{equation*}
For simplicity, we assume that $g(Y)$ is odd and hence that $n$ is
divisible by 4. Pick a divisor on $B$:
\begin{equation*}
D_1=\sum_{j=n+1}^{\frac 54 n}\xi_j,
\end{equation*}
which is disjoint from the branched locus $D_0$. Consequently
\begin{equation*}
\deg\bigl(\varphi^\ast\cO_B(D_1)\bigr)=2\deg\bigl(\cO_B(D_1)\bigr)=\frac n 2.
\end{equation*}
Therefore, the line bundle $\mathcal N$ in Section 1 can be taken as
$\varphi^\ast\mathcal O_B(D_1)\otimes \mathcal F'$ for a degree zero
line bundle $\cF'$ over $Y$. Without loss of generality, we can assume
that $\mathcal F'$ is trivial. (Otherwise one can tensor a flat
metric on $\mathcal F'$ with the constructed Hermitian metrics on
$V$ in Section 5.) Thus,
\begin{equation*}
V=p_{2\ast}(\mathcal L)\ \ \ \textup{for}\ \ \ \mathcal
L=\iota\sta\cP\otimes (\varphi\circ p_1)^\ast\mathcal O_B(D_1).
\end{equation*}

For our purposes, we will give a local trivialization of $V$. Denote
by $d_B$ the distance on $B$  induced from the
Euclidean metric on $\mathbb C$. Hence, $d_B$ does not depend on
$\epsilon$. Pick a small $r_0>0$ so that the discs
\begin{equation*}\lab{10.1}
U\lalp=\{z\in B\mid d_B(z,\xi\lalp)<2r_0\}\sub B \quad \textup{for}\
\alpha=1,\cdots,5n/4
\end{equation*}
are disjoint.  For each $\alpha$, let $z_\alpha$ be a coordinate on
$U_\alpha$ so that $z_\alpha(\xi_\alpha)=0$. In the following, for
convenience,  we will denote $\alpha=0,1,\cdots, 5n/4$; $a=1,\cdots
n$; and $j=n+1,\cdots, 5n/4$.

We first give local holomorphic frames of $V$. Denote
$U_0^0=B\setminus D_1$.  We can give a local holomorphic frame
$e_0$ of $\mathcal O_B(D_1)|_{U_0^0}$ and $e_j$ of $\mathcal
O_B(D_1)|_{U_j}$ so that over $U_j\cap U_0^0,$
\begin{equation}\label{120820}
e_j(z_j)=z_j^{-1}\cdot e_0(z)|_{z=z_j}.
\end{equation}
Take $U_{0}=B\setminus(D_0\cup D_1)$. Then $U_0$, $U_a$, and $U_j$
 form  an open cover of $B$ and their pre-images $\mathcal
U_{0}$, ${\mathcal U}_a$, and ${\mathcal U}_j$ in $X$ form an open
cover of $X$. We can localize $V|_{\mathcal U_0}$. Let
$(\tilde\mu_1^0,\tilde\mu_2^0)$ be local holomorphic frames of
$V|_{\cU_0}$ defined by
\begin{eqnarray}\lab{120805}
\begin{aligned}
  \tilde\mu^0_1(w,z)&=p_{2*}(\tilde\varepsilon_{(w^\ast_1(z),w)}\otimes
  e_0(z)),\\
\tilde\mu^0_2(w,z)&=p_{2*}(\tilde\varepsilon_{(w^\ast_2(z),w)}\otimes
e_0(z)).
\end{aligned}
\end{eqnarray}
Here $w_1^\ast(z)$ and $w_2^\ast(z)$ are two local sections of
\begin{equation}\label{171212}
\varphi|_{\varphi^{-1}(U_0)}:\varphi^{-1}(U_0)\subset Y\rightarrow U_0\subset B.
\end{equation}
Since under our assumption, two  sections $w_1^\ast(z_j)$ and
$w_2^\ast(z_j)$ of $\varphi:Y\to B$ restricted to $U_j$ are
well-defined,  we have a holomorphic frame of $V|_{\mathcal U_j}$
\begin{eqnarray}\label{120821}
\begin{aligned}
\tilde\mu_1^j(w,z_j)&=p_{2\ast}(\tilde\varepsilon_{(w_1^\ast(z_j),w)}\otimes
e_j(z_j)),\\
\tilde\mu_2^j(w,z_j)&=p_{2\ast}(\tilde\varepsilon_{(w_2^\ast(z_j),w)}\otimes
e_j(z_j)).
\end{aligned}
\end{eqnarray}
Thus, if  we take locally $w_1^\ast(z)=w_1^\ast(z_j)$ and
$w_2^\ast(z)=w_2^\ast(z_j)$ in (\ref{120805}), then, in view of
(\ref{120820}), combining (\ref{120805}) with (\ref{120821})  gives
the relations over $\mathcal U_0\cap\mathcal U_j$:
\begin{equation}\label{120804}
\tilde\mu_1^j=z_j^{-1}\tilde\mu_1^0,\qquad
\tilde\mu_2^j=z_j^{-1}\tilde\mu_2^0.
\end{equation}

We next look at $\mathcal U_a$. Since $\varphi:Y\to B$ is the
two-sheet branched cover ramified at $\xi_a$, we choose $w^\ast$ so
that over $U_a$ the curve $Y\subset X^\ast$ is given by
$(w^\ast)^2=z_a$.
If we pick $w_1^\ast(z)=\sqrt{z_a}$ and $w^\ast_2(z)=-\sqrt{z_a}$ in
(\ref{120805}) and set
\begin{equation}\label{120803}
\tilde\mu_1^a=\frac 1 {\sqrt{2}}(\tilde\mu_1^0+\tilde\mu_2^0),\qquad
\tilde\mu_2^a=\frac{\sqrt{z_a}}{\sqrt{2}}(\tilde\mu_1^0-\tilde\mu_2^0),
\end{equation}
then the sections $\tilde\mu_1^a$ and $\tilde\mu_2^a$ are
well-defined holomorphic  sections of $V|_{\mathcal U_0\cap \mathcal U_a}$
independent of the choice of single-valued branch of $\sqrt{z_a}$;
also they generate the holomorphic bundle $V|_{\mathcal U_0\cap\mathcal U_a}$. Thus
we can set them to be a holomorphic frame of $V|_{\mathcal U_a}$. In
other words, (\ref{120803}) gives
the transition functions over $\mathcal U_0\cap\mathcal U_a$ between the frames $(\tilde\mu_1^a,\tilde\mu_2^a)$
and $(\tilde\mu_1^0,\tilde\mu_2^0)$.

Similarly, we can also use $\varepsilon_{(w^\ast,w)}$ to define
locally smooth frames $(\hat{\mu}_1^0,\hat\mu_2^0)$ of $V|_{\mathcal
U_0}$, a smooth frame $(\hat\mu_1^j,\hat\mu_2^j)$ of $V|_{\mathcal
U_j}$, and $(\hat\mu_1^a,\hat\mu_2^a)$ of $V|_{\mathcal U_a}$. They
also satisfy the relations:
\begin{gather}
\hat\mu_1^j=z_j^{-1}\hat\mu_1^0,\qquad
\hat\mu_2^j=z_j^{-1}\hat\mu_2^0\qquad \text{over}\ \mathcal
U_j\cap\mathcal U_0;\label{120810}
\\
\hat\mu_1^a=\frac 1 {\sqrt{2}}(\hat\mu_1^0+\hat\mu_2^0),\qquad
\hat\mu_2^a=\frac{\sqrt{z_a}}{\sqrt{2}}(\hat\mu_1^0-\hat\mu_2^0)\qquad \textup{over}\ \mathcal U_a\cap \mathcal U_0.\label{120811}
\end{gather}

Finally, by (\ref{406}),  the local holomorphic frames are
related to the smooth frames as follows:
\begin{equation}\label{120850}
(\tilde\mu_1^\alpha,\tilde\mu_2^\alpha)=
(\hat\mu_1^\alpha,\hat\mu_2^\alpha)A_\alpha,
\end{equation}
where
\begin{equation*}\textstyle
\begin{split}
&A_0=\left(\begin{array}{cc}\exp(\pi i w_1^\ast(z)\overline{w}) &0\\
0&\exp(\pi i w_2^\ast(z)\overline{w})\end{array} \right),\\
&A_j=\left(\begin{array}{cc}\exp(\pi i w_1^\ast(z_j)\overline{w}) & 0\\
0& \exp(\pi i w_2^\ast(z_j)\overline{w})\end{array}\right),\quad \textup{and}\\
&A_a=\left(
\begin{array}{cc}
\cosh(\pi i\sqrt{z_a}\overline{w}) & \sqrt{z_a}\sinh(\pi i
\sqrt{z_a}\overline{w})\\
\frac{1}{\sqrt{z_a}}\sinh(\pi i \sqrt{z_a}\overline{w}) &\cosh(\pi
i\sqrt{z_a}\overline{w})
\end{array}\right).
\end{split}
\end{equation*}

\section{A family of local HYM metrics}
In this section, we first recall  some definitions and notations on
connections in Hermitian vector bundles as in Chapter 1 of
\cite{Ko}. (Hence, our notations here differ from \cite{GH}.)

Let
$E$ be a rank $r$ complex vector bundle over a smooth manifold $M$.
Let $D$ be a connection in $E$. If let $s_U=(s_1,\cdots,s_r)$ be
a local frame of $E$ over an open subset $U\subset M$, then we can
write
\begin{equation*}
Ds_i=\sum s_j\theta_i^j.
\end{equation*}
The matrix valued 1-form $\theta_U=(\theta_i^j)$ is called the
connection form of $D$ with respect to $s_U$. The curvature form
$\Theta_U$ of $D$ with respect to $s_U$ is defined as
\begin{equation}\label{o2}
\Theta_U=d\theta_U+\theta_U\wedge\theta_U.
\end{equation}

Let $s'_U=(s_1',\cdots,s_r')$ be another local frame over $U$ which
is related to $s_U$ by
\begin{equation}\label{a0}
s_U=s_U'A_U.
\end{equation}
Here $A_U:U\to GL(r,\mathbb C)$ is a matrix-valued function on $U$.
Let $\theta'_U$  and $\Theta'_U$ be the connection and
curvature forms of $D$ with respect to  $s'_U$. Then one finds relations:
\begin{equation}\label{a1}
\theta_U=A_U^{-1}\theta' _UA_U+A_U^{-1}dA_U
\end{equation}
and
\begin{equation}\label{a8}
\Theta_U=A_U^{-1}\Theta'_U A_U.
\end{equation}
The first and second Chern classes $c_1(E)$ and $c_2(E)$ of $E$ are represented by curvature forms as follows:
\begin{align}
&C_1(E)=\frac {-1} {2\pi i}\textup{Tr}\Theta_U;\label{171010}\\
&C_2(E)=\frac1 {8\pi^2}\bigl(\textup{Tr}(\Theta_U\wedge\Theta_U)-(\textup{Tr}\Theta_U)^2\bigr).\label{171015}
\end{align}

Let $H$ be  a Hermitian metric on $E$. Set
\begin{equation*}
h_{i\bar j}=H(s_i,s_j)
\end{equation*}
and $H_U=(h_{i\bar j})$ which is a positive definite Hermitian
matrix at every point of $U$. Under a change of frames given by
(\ref{a0}),  the corresponding  Hermitian matrices $H_U$ and $H'_U$
satisfy
\begin{equation}\label{04}
H_U=(A_U)^tH_U'\overline{ A_U}.
\end{equation}
Here $(A_U)^t$ is denoted as the transpose of $A_U$.

Now assume that  $E$ is a holomorphic vector bundle over a complex
manifold $M$. Let $H$ be a Hermitian metric on $E$. The
Hermitian connection $D_H$ (i.e., the Chern connection in some references) associated to $H$ is defined as follows. Let
$\tilde s_U=(\tilde s_1,\cdots,\tilde s_r)$ be a local holomorphic
frame on $U$ and $\tilde H_U$ be the Hermitian matrix for $H$ in $\tilde s_U$. Then the connection form of $D_H$ with
respect to $\tilde s_U$ is
\begin{equation}\label{b1}
\tilde\theta_U=(\partial \tilde H_U\cdot \tilde H_U^{-1})^t.
\end{equation}
By (\ref{o2}), its curvature form is
\begin{equation}\lab{b0}
\tilde\Theta_U=\overline\partial(\partial \tilde H_U\cdot \tilde
H_U^{-1})^t,
\end{equation}
which is a matrix valued $(1,1)$-form. Hence,
by (\ref{a8}) the curvature form $\Theta$ of $D_H$
with respect to any frame $s_U$ is also a matrix valued $(1,1)$-form.

At last, assume that $(M,\omega)$ is a compact K\"ahler manifold with complex dimension $m$. Define
\begin{equation}\label{a4}
\Lambda_\omega \Theta=\frac{m\cdot \Theta\wedge
\omega^{m-1}}{\omega^m}.
\end{equation}
Assume that  $c_1(E)=0$.  A Hermitian metric  $H$ on $E$ is called a
HYM metric with respect to $\omega$  if its associated  curvature form $\Theta$ satisfies
\begin{equation*}
\Lambda_\omega\Theta=0.
\end{equation*}

\vspace{1mm}
In the following, we shall derive a system of HYM connections of $V$
over $\cU_a$  for $1\leq a\leq n$.  Because the $V|_{\cU_a}$'s are
all essentially the same, we shall work out one of them in detail.
For convenience, we shall drop the super(sub)-script $a$.

We endow $V|_{\cU}$ with a class of metrics. For any $\epsilon>0$,
let $u_\epsilon\mh U\to\RR$ be a real function and set
\begin{equation}\label{414}
{\hat h}_\epsilon=\left(\begin{array}{cc}e^{-u_{\epsilon}} & 0\\
0 & e^{ u_\epsilon }\end{array}\right).
\end{equation}
Since $u_\epsilon$ does not depend  on the  variable $w$, ${\hat
h}_\epsilon$ gives a Hermitian metric $ h_\epsilon$ on $V|_{\mathcal
U}$ so that it is the Hermitian matrix for $h_\epsilon$ in
$(\hat\mu_1,\hat\mu_2)$. According to (\ref{04}) and
(\ref{120850}),
\begin{equation}\label{a2}
\tilde{{ h}}_\epsilon=A^t {\hat h}_\epsilon \overline{A}
\end{equation}
 gives the Hermitian matrix for
$h_\epsilon$ in $(\tilde\mu_1,\tilde\mu_2)$,
which depends on $w$.  Hence, the Hermitian connection also
depends on $w$ (see below).

Let $D_{h_\epsilon}$ be
the Hermitian connection on $(V|_{\mathcal U},h_\epsilon)$. Let
$\tilde{\theta}_\epsilon$ and $\hat\theta_\epsilon$ be the
connection forms of $D_{h_\epsilon}$ with respect to
$(\tilde{\mu}_1,\tilde{\mu}_2)$ and $(\hat\mu_1,\hat\mu_2)$.
Then, by (\ref{b1})
\begin{equation}\label{120830}
\tilde{\theta}_\epsilon=(\partial\tilde{{ h}}_\epsilon \cdot
\tilde{{ h}}_\epsilon^{-1})^t,
\end{equation}
and by (\ref{a1}) $\hat\theta_\epsilon$ is related to
$\tilde\theta_\epsilon$ as
\begin{equation*} \label{416}
\hat\theta_\epsilon=A\tilde{\theta}_\epsilon A^{-1}-dA\cdot A^{-1}.
\end{equation*}
Inserting (\ref{a2}) into  (\ref{120830}) and  then inserting the
resulting equation into the above equation, we get
\begin{eqnarray*}
\hat\theta_\epsilon &=&-\overline{\partial}AA^{-1}+(\partial{{\hat
h}}_\epsilon \cdot
{{\hat h}}_\epsilon^{-1})^{t}+({\hat h}_\epsilon\overline{\overline{\partial}AA^{-1}}{\hat h}_\epsilon^{-1})^t\\
&=&-\pi i \left(\begin{array}{cc}0 & z\\
1 & 0 \end{array}\right)d\overline w-\pi i\left(\begin{array}{cc}0 &e^{2u_\epsilon}\\
\overline{z}e^{-2u_\epsilon}& 0 \end{array}\right)d w+\left(\begin{array}{cc}-1 & 0\\
0 & 1 \end{array}\right)\frac{\partial u_\epsilon}{\partial z}dz.
\end{eqnarray*}
Therefore, by (\ref{o2})  the associated  curvature form is
\begin{equation}\label{171225}
\begin{aligned}
\hat\Theta(h_\epsilon) =&\left(\begin{array}{cc}1 & 0\\
0 & -1 \end{array}\right)\frac{\partial^2u_\epsilon}{\partial
z\partial \overline{z}}dz\wedge d\overline{z}+\pi^2(|z|^2e^{-2u_\epsilon}-e^{2u_\epsilon})\left(\begin{array}{cc}1 & 0\\
0 & -1 \end{array}\right)dw\wedge d\overline{w} \\
 &-\pi i\left(\begin{array}{cc}0& 1-2z\frac{\partial u_\epsilon}{\partial z}\\
2\frac{\partial u_\epsilon}{\partial z}& 0\end{array}\right)dz\wedge d\overline w\\
&- \pi i\left(\begin{array}{cc} 0& 2e^{2u_\epsilon}\frac{\partial u_\epsilon}{\partial \overline z}\\
e^{-2u_\epsilon}(1-2\overline z\frac{\partial u_\epsilon}{\partial\overline z})& 0\end{array}\right)d\overline z\wedge dw .
\end{aligned}
\end{equation}
Thus, by definition (\ref{a4}) with $m=2$ and
$\omega=\omega_\epsilon$ in (\ref{403}), we obtain
\begin{equation*}
\frac i 2 \Lambda_{\omega_\epsilon}\hat\Theta(h_\epsilon)=\Bigl(\epsilon\frac{\partial^2u_\epsilon}{\partial
z\partial \overline{z}}+\frac{\pi^2}{\epsilon}(|z|^2e^{-2u_\epsilon}-e^{2u_\epsilon})\Bigr)\left(\begin{array}{cc}1 & 0\\
0 & -1 \end{array}\right).
\end{equation*}
Based on this, we see that $h_\epsilon$ becomes HYM if $u_\epsilon$
satisfies the equation\footnote{Compare this equation with Hitchin's equations, cf.  \cite{MSWW}.}:
\begin{equation}\label{417}
\frac{\partial^2u_\epsilon}{\partial z\partial
\overline{z}}={\pi^2}{\epsilon^{-2}}\bigl(e^{2u_\epsilon}-|z|^2e^{-2u_\epsilon}\bigr).
\end{equation}

\section{Reduction to ODE}
In this section, we shall study the Dirichlet problem
\begin{equation}\label{418}
\left\{\begin{array}{rll}\displaystyle
 \bigtriangleup u
&={4\pi^2}{\epsilon^{-2}}\bigl(e^{2u}-r^2e^{-2u}\bigr) &\text{in} \ B_{2r_0}(0)\\
&\\
u& =\frac{1}{2}\ln(2r_0)  &\text{on}\ \partial B_{2r_0}(0).
\end{array}
\right.
\end{equation}
Here we denote $x=(x_1,x_2)$ as the standard coordinate of
$B_{2r_0}(0)$, $r^2=x_1^2+x_2^2$, and
$\bigtriangleup=\frac{\partial^2}{\partial
{x_1}^2}+\frac{\partial^2}{\partial{ x_2}^2}$. It is easy to see
that $u=\frac 1 2 \ln r$ is a singular solution to (\ref{418}).

The main result of this section is as follows.
\begin{theo}
\label{solution} Equation (\ref{418}) has a unique smooth and
radially symmetric solution $u_\epsilon$ that satisfies the
following estimates.
\vskip2pt
\noindent \textup{(1)} Let
$v_\epsilon(r)=u_\epsilon(r)-\frac{1}{2}\ln r$, $r\in[r_0,2r_0]$,
and let $v_\epsilon^{(k)}(r)$ be the $k$-th derivative of
$v_\epsilon(r)$ in r. Then for any positive integer $l$ and
nonnegative integer $k$ satisfying $l>k$,  there is a constant
$C=C(r_0,l,k)$ such that for any $0<\epsilon<1/8$,
\begin{equation*}
\bigl|\!\bigl| v_\epsilon^{(k)}(r) \bigr|\!\bigr|_{C^0([r_0,2r_0])}
\leq C\epsilon^{l-k}.
\end{equation*}
\vskip2pt
\noindent \textup{(2)} For any $R<2r_0$ and positive integer $k$, there exists a constant $C=C(r_0,R,k)$ such that
for any sufficiently small $\epsilon>0$,
 \begin{equation*}
 \parallel\!
u_\epsilon\!\parallel_{C^{k}(B_R(0))}\leq C\epsilon^{-3k+2}.
\end{equation*}
\end{theo}

\begin{proof}
After substituting $\overline {u}$ for $2u-\ln(2r_0)$, ${x}_1$ for
$\frac{x_1}{2r_0}$, ${x}_2$ for $\frac{x_2}{2r_0}$, $r^2$ for
$\frac{r^2}{4r_0^2}$, and ${\epsilon}$ for
$\frac{\epsilon}{8\pi} r_0^{-\frac{3}{2}}$,  equation
(\ref{418}) becomes
\begin{equation}\label{419}
\left\{
\begin{array}{rll}
\bigtriangleup \overline{u}
&={\epsilon^{-2}}\bigl(e^{\overline{u}}-r^2 e^{-\overline{u}}\bigr) &\text{in}\  B_{1}(0)\\
&\\
\overline {u} &=0  &\text{on}   \ \partial B_{1}(0).
\end{array}
\right.
\end{equation}
Theorem \ref{solution} will follow from Propositions \ref{prop},
\ref{lemma4} and \ref{20180902} below.
\end{proof}

\begin{prop}
\label{prop} Equation (\ref{419}) has a a unique smooth and radially
symmetric solution $\overline{u}_\epsilon$ which satisfies
$\overline{u}_\epsilon<0$ and  $\frac{\partial}{\partial r}
\overline{u}_\epsilon>0$ for $0<r<1$.
\end{prop}
\begin{proof}
Since for each $x=(x_1,x_2)$ the function
${\epsilon^{-2}}\left(e^{\overline{u}}-r^2e^{-\overline{u}}\right)$
is a monotone increasing function of $\overline{u}$, according to
\cite{Sc} the boundary value problem (\ref{419}) has a unique
solution.

To prove that this solution is radially symmetric, we first use the
maximum principle to prove that the solution $\overline u_\epsilon$
to (\ref{419}) is negative. Let $x_0\in\overline{B}_1(0)$ be such
that $\overline{u}_\epsilon(x_0)=
{\sup_{x\in\overline{B}_1(0)}\overline{u}_\epsilon}$. In case
$\overline{u}_\epsilon(x_0)\geq 0$ and $x_0\notin\partial B_1(x_0)$,
we have $e^{2\overline{u}_\epsilon(x_0)}-|x_0|^2>0$. Hence, there
is a neighborhood $\Omega\subset B_1(0)$ of $x_0$ such that $e^{2\overline{u}_\epsilon(x)}-|x|^2>0$ in $\Omega$. Therefore,
\begin{equation*}
\bigtriangleup \overline{u}_\epsilon
={\epsilon^{-2}}\bigl(e^{\overline{u}_\epsilon}-r^2e^{-\overline{u}_\epsilon}\bigr)>0\quad
\textup{for}\ x\in\Omega.
\end{equation*}
The strong maximum principle implies that the maximum of
$\overline{u}_\epsilon$ on $\overline{\Omega}$ can be achieved only
on $\partial \Omega$, contradicting the assumption that $x_0$ is a local maximum
of $\overline u_\epsilon$. This proves that
$\overline{u}_\epsilon<0$ in $B_1(0)$. After this, one can apply
Corollary 1 of \cite[p.227]{Gi} to conclude that
$\overline{u}_\epsilon$ is radially symmetric and
$\frac{\partial}{\partial r}\overline{u}_\epsilon>0$ for all
$0<r<1$.
\end{proof}

By Proposition \ref{prop}, equation (\ref{419}) can be reduced to an ODE:
\begin{equation}\label{420}
\overline{u}''(r)+r^{-1}\overline{u}'(r)=
{\epsilon^{-2}}\bigl(e^{\overline{u}(r)}-r^2e^{-\overline{u}(r)}\bigr).
\end{equation}
Our next goal is to show that when $\epsilon\rightarrow 0$ the solution $\overline{u}_\epsilon(r)$ is close to
$\ln r$ for $r\in[\frac{1}{2},1]$. We
shall set $\overline{v}_\epsilon(r)=\overline{u}_\epsilon(r)
 -\ln r$ for $r\in ]0,1]$ and estimate $\bigl|\!\bigl|\overline{v}^{(k)}_\epsilon(r)\bigr|\!\bigr|_{C^0([\frac{1}{2},1])}
$. Clearly, for any fixed $\epsilon>0$, $\overline v_\epsilon(1)=0$
and $\lim_{r\to 0}\overline v_\epsilon(r)=+\infty$.

\begin{lemm}\label{171250}
\label{lemm1} When $r\in ]0,1[$, $\overline{v}_\epsilon(r)$ satisfies
\begin{equation*}
\overline{v}_\epsilon(r)>0,\quad \overline{v}'_\epsilon(r)<0,\quad
\overline{v}''_\epsilon(r)>0, \and \overline{v}'''_\epsilon(r)<0.
\end{equation*}
\end{lemm}

\begin{proof}
According to (\ref{420}), $\overline v_\epsilon(r)$ satisfies\footnote{This equation is of Painlev\'e III type, c.f. P. 2240 in \cite{MSWW}}.
\begin{equation}\label{421}
\overline{v}''_\epsilon(r)+r^{-1}\overline{v}'_\epsilon(r)=2{\epsilon^{-2}}
r\sinh \overline{v}_\epsilon(r).
\end{equation}
We first use the maximum principle to prove $\overline
v_\epsilon(r)>0$. If it would not be,  let $r_0$ be the first point in
$]0,1[$ such that $\overline v_\epsilon(r_0)=\inf_{r\in
]0,1[}\overline v_\epsilon(r)\leq 0$. Hence, $\overline
v_\epsilon'(r_0)=0$ and $\overline v_\epsilon''(r_0)\geq 0$, which in turn implies $\overline v_\epsilon(r_0)= 0$ by (\ref{421}).
Thus by the uniqueness theorem of solutions of an ODE, one can assume
that there exists  $r_1\in ]r_0,1[$ such that $\overline v_\epsilon
(r_1)=\sup_{r\in ]r_0,1[}\overline v_\epsilon(r)>0$, which implies
$\overline v_\epsilon'(r_1)=0$ and $\overline v_\epsilon''(r_1)\leq
0$. This contradicts  (\ref{421}). Hence, $\overline
v_\epsilon(r)>0$ for all $r\in ]0,1[$.

Now applying  \cite[Theorem 3]{Gi} to equation (\ref{421}), one
gets $\overline v_\epsilon'(r)<0$ for $r\in [\frac 1 2,1[$. We claim
that this inequality holds for all $r\in ]0,1[$. Otherwise,  there
would exist
 $r_2\in ]0,\frac 1 2[$ such that $\overline v_\epsilon'(r_2)=0$
and $\overline v_\epsilon'(r)<0$ for any $r>r_2$. Hence, $\overline
v''_\epsilon(r_2)\leq 0$ and (\ref{421}) implies
$\sinh\overline v_\epsilon(r_2)\leq 0$ or $\overline
v_\epsilon(r_2)\leq 0$. This is a contradiction.

The inequality for the second derivative follows directly from
(\ref{421}). Differentiating (\ref{421}) with respect to $r$ and
using (\ref{421}) again, we get
\begin{equation}\label{a5}
\overline{v}'''_\epsilon(r)=2\bigl({r^{-2}}+{\epsilon^{-2}}r\cosh
\overline{v}_\epsilon(r)\bigr)\overline{v}'_\epsilon(r).
\end{equation}
Hence,  $\overline{v}'''_\epsilon(r)<0$ follows.
\end{proof}

For $t\in ]0,1]$,   set
\begin{equation}\label{z}
M_i(t) =\left\{\begin{array}{ll}{\sup}_{r\in[t,1]}\ \bigl|
\overline{v}^{(i)}_\epsilon(r)\bigr| & \textup{for}\ i=0,1,2\\
\\
{\sup}_{r\in[t,1]}\ |\sinh\overline{v}_\epsilon(r)| & \textup{for}\
i=3. \end{array}\right.
\end{equation}
Then,  by Lemma \ref{lemm1}, $M_i(t)$ is strictly decreasing in
$t\in ]0,1[$ and $M_0(t)<M_3(t)$. We first show that
\begin{equation}\label{lemma2}
M_3(1/4)\leq 2^8\epsilon^2.
\end{equation}

Indeed, rewriting (\ref{420}) as
\begin{equation*}
\label{422} (r\overline{u}'_\epsilon(r))'=2{\epsilon^{-2}}r^2\sinh
\overline{v}_\epsilon(r)
\end{equation*}
and integrating over $[0,1]$, we have
\begin{equation*}
\label{423}
 \overline{u}'_{\epsilon-}(1)=\int^1_0
(r\overline{u}'_\epsilon(r))'dr=\int^1_0{2}{\epsilon^{-2}}r^2\sinh
\overline{v}_\epsilon(r)dr.
\end{equation*}
On the other hand, the first item in Lemma \ref{lemm1}  implies
$\overline u_\epsilon(r)>\ln r$, and hence,
\begin{equation*}
\label{424} \overline{u}'_{\epsilon-}(1)=\lim_{r\rightarrow
1-0}\frac{\overline{u}_{\epsilon}(r)-\overline{u}_{\epsilon}(1)}{r-1}\leq
\lim_{r\rightarrow 1-0}\frac{\ln r-\ln 1}{r-1}=1.
\end{equation*}
Thus,
\begin{equation}\label{c0}
\int^1_0r^2\sinh \overline{v}_\epsilon(r)dr\leq
\epsilon^2/2.
\end{equation}
Since $\sinh \overline{v}_{\epsilon}(r)$ is strictly decreasing,
\begin{equation*}
(1/8)^2\sinh\overline v_\epsilon(1/4)<r^2\sinh\overline
v_\epsilon(r)\ \ \ \textup{for}\ r\in [1/8,1/4].
\end{equation*}
Integrating over $[1/8,1/4]$ and using (\ref{c0}), we obtain
\begin{equation*}
(1/8)^3\sinh\overline v_\epsilon(1/4)<\epsilon^2/2.
\end{equation*}
This proves (\ref{lemma2}).

We need more estimates on $M_i(t)$.

\begin{lemm}
\label{lemma3} For any $t$, $t'\in[1/4,1/2]$
and for any $0<\epsilon<1/8$, \\
\textup{(1)} $M_2(t)=\frac{2t}{\epsilon^2}M_3(t)+\frac{1}{t}M_1(t)$;\\
\textup{(2)}  $M_1(t)<\frac{2}{\epsilon}M_3(t)$; and\\
\textup{(3)} $M_3(t')<\frac{2\epsilon^2}{t'-t}M_1(t)$ for $t'>t$.
\end{lemm}

\begin{proof}
Formula (1) follows directly from (\ref{421}) and Lemma \ref{lemm1}.
We now prove (2). For $1/4\leq t\leq 1/2$ and $0<\epsilon<1/8$, the
Taylor expansion of $ \overline{v}_\epsilon(r)$ at $r=t$ is given by
\begin{equation*}
\overline{v}_\epsilon(t+\epsilon)=
\overline{v}_\epsilon(t)+\overline{v}'_\epsilon(t)\epsilon
+\overline{v}''_\epsilon(t+\eta\epsilon){\epsilon^2}/2,\ \
0\leq\eta\leq1.
\end{equation*}
Then, using Lemma \ref{lemm1}, we have
\begin{align*}
 0>\overline{v}'_\epsilon(t)\epsilon=\overline{v}_\epsilon(t+\epsilon)-\overline{v}_\epsilon(t)-
\overline{v}''_\epsilon(t+\eta\epsilon){\epsilon^2}/{2}
>-\overline{v}_\epsilon(t)-
\overline{v}''_\epsilon(t){\epsilon^2}/{2}.
\end{align*}
Hence,
\begin{equation*}
M_1(t)<{\epsilon^{-1}}M_0(t)+({\epsilon}/{2})
M_2(t)<{\epsilon^{-1}}M_3(t)+({\epsilon}/{2})M_2(t).
\end{equation*}
Substituting (1) into the above inequality, we obtain
\begin{equation*}
M_1(t)<{\epsilon^{-1}}M_3(t)+{t}{\epsilon^{-1}}M_3
(t)+\epsilon(2t)^{-1} M_1(t).
\end{equation*}
Consequently, as $t\in [1/4,1/2]$ and $\epsilon\in[0,1/8]$,
\begin{equation*}
M_1(t)<\frac{1+t}{\epsilon(1-\frac{\epsilon}{2t})}
M_3(t)\leq\frac{2}{\epsilon}M_3(t).
\end{equation*}
This proves (2).

For (3), one can rewrite (\ref{421}) as
\begin{equation*}
(r\overline{v}'_\epsilon(r))'={2}{\epsilon^{-2}}r^2\sinh
\overline{v}_\epsilon(r).
\end{equation*}
Integrating  over $[t,1]$ and using
Lemma \ref{lemm1}, we get
\begin{equation}\label{425}
 {2}{\epsilon^{-2}}\int^1_t r^2\sinh
\overline{v}_\epsilon(r)dr=\overline{v}'_{\epsilon
-}(1)-t\overline{v}'_\epsilon(t)\leq t| \overline{v}'_\epsilon(t)|=t
M_1(t).
\end{equation}
On the other hand, as in the proof of inequality (\ref{lemma2}), for $t'>t$ we have
\begin{equation*}\label{426}
\begin{aligned}
&{2}{\epsilon^{-2}}\int^1_t r^2\sinh \overline{v}_\epsilon(r)dr
>{2}{\epsilon^{-2}}\int^{t'}_t r^2\sinh
\overline{v}_\epsilon(r)dr \\
&\geq{2}{\epsilon^{-2}}t^2(t'-t)\sinh \overline{v}_\epsilon(t')
={2}{\epsilon^{-2}}t^2(t'-t)M_3(t').
\end{aligned}
\end{equation*}
Combined with (\ref{425}), as $ t\in [\frac 1 4, \frac 1 2]$, we obtain (3).
\end{proof}

We are now  ready to prove estimates (1) in Theorem \ref{solution}.

\begin{prop}
\label{lemma4} For any  positive integer $l$ and nonnegative integer
$k$ satisfying $l>k$, there exists a constant $C=C(l,k)$ such that
for any $0<\epsilon<1/8$,
\begin{equation*}
\bigl|\!\bigl|\overline v_\epsilon^{(k)}(r)\bigr|\!\bigr|_{C^0([\frac 1 2,1])}\leq C\epsilon^{l-k}.
\end{equation*}
\end{prop}

\begin{proof}
According to  definition (\ref{z}), $\bigl|\!\bigl|\overline
v_\epsilon^{(k)}(r)\bigr|\!\bigr|_{C^0([\frac 1 2,1])}=M_k(\frac 1 2)$
for $k=0,1,2$. We first examine  the case where $k=0$. Combining (2)
and (3) in Lemma \ref{lemma3}, we have
\begin{equation*}
M_3(t')\leq\frac{2^2\epsilon}{t'-t} M_3(t)\quad \text{for}\ \
1/4\leq t<t'\leq 1/2.
\end{equation*}
Based on this inequality, one can use the iterated method to get:
\begin{equation*}
M_3(1/2)\leq M_3\Bigl(\frac{1}{2}\cdot\frac{l-1}{l}\Bigr)
\leq (2^3)^{l-2}l(l-1)^2\cdots 3^2\cdot 2\epsilon^{l-2}
M_3(1/4).
\end{equation*}
Hence by (\ref{lemma2})
\begin{equation*}
M_3(1/2)\leq 2^{3l+1}{(l!)^2}l^{-1}\epsilon^l.
\end{equation*}
Thus,
\begin{equation*}
 M_0(1/2)\leq M_3(1/2)\leq
M_3\Bigl(\frac{1}{2}\cdot \frac{l-1}{l}\Bigr)\leq
2^{3l+1}(l!)^2l^{-1}\epsilon^l.
\end{equation*}
This proves the case where $k=0$.

The case where $k=1$ follows from (2) in Lemma \ref{lemma3}:
\begin{equation*}
 M_1(1/2)< M_1\Bigl(\frac{1}{2}\cdot
\frac{l-1}{l}\Bigr)\leq\frac{2}{\epsilon} M_3\Bigl(\frac{1}{2}\cdot
\frac{l-1}{l}\Bigr)\leq 2^{3l+2}{(l!)^2}{l^{-1}}\epsilon^{l-1}.
\end{equation*}
Now the case where $k=2$ follows from  (1) in Lemma \ref{lemma3} and the
above discussions.

For the case where $k\geq 3$, taking the $(k-3)$-times of derivatives
to both sides of (\ref{a5}) and  using the inductive method, one
gets  the estimates in the proposition.
\end{proof}

In the remainder   of this section, we will prove estimates (2) in Theorem
\ref{solution}. For brevity, in this time  set
\begin{equation*}
F(\overline u_\epsilon,r^2)= {\epsilon^{-2}}(e^{\overline u
_\epsilon}-r^2e^{-\overline u_\epsilon}).
\end{equation*}
Denote  the derivatives of $F$ in the first and
 second variables by $F_1$ and $F_2$, respectively. We also use the notations $F_{11}$,
$F_{12}$, $F_{22}$, and so on. Then we have the following formulas:
\begin{equation}\label{171272}
F_1= {\epsilon^{-2}}(e^{\overline u_\epsilon}+r^2e^{-\overline
u_\epsilon}),\qquad  F_2=-{\epsilon^{-2}}e^{-\overline
u_\epsilon};
\end{equation}
\begin{equation}\label{171273}
F_{11}=F,\qquad F_{12}=-F_2,\qquad F_{22}=0.
\end{equation}

\begin{lemm}\label{1005}
For any $r\in ]0,1[$,\\
\textup{(1)} $0<F<\epsilon^{-2}$, \ $0<F_1<2 \epsilon^{-2}$, \
$-\epsilon^{-2}<rF_2<0$; and\\
\textup{ (2)} $\parallel\!\overline
u_\epsilon\!\parallel_{C^0}\leq \epsilon^{-1}$,\
$\parallel\!\overline u_\epsilon'\!\parallel_{C^0}\leq
\epsilon^{-1}$,\ $\parallel\!\overline
u''_\epsilon\!\parallel_{C^0}\leq \epsilon^{-2}$.
\end{lemm}

\begin{proof}
By Proposition \ref{prop} and the first two inequalities in Lemma
\ref{lemm1}, we have
\begin{equation}\label{171271}
\ln r <\overline u_\epsilon(r)<0\quad\and\quad
0<\overline
u'_\epsilon(r)<r^{-1}.
\end{equation}
Hence the first two inequalities in (1) are valid
and   the derivative of
$rF_2$ in $r$
satisfies
\begin{equation*}
(rF_2)'=-{\epsilon^{-2}}(1-r\overline
u_\epsilon'(r))e^{-\overline u_\epsilon}<0,
\end{equation*}
which implies $rF_2$ is strictly decreasing in $r$ and hence the third one in (1)
follows.

As to the inequalities in (2), we first rewrite equation (\ref{420}) as $(r\overline u'_\epsilon(r))'=rF$.
Then by the first inequality in (1) we have
\begin{equation*}
0<(r\overline u'_\epsilon(r))'< {r}{\epsilon^{-2}}.
\end{equation*}
Integrating over $[0,r]$ gives
\begin{equation}\label{1001}
0<\overline u_\epsilon'(r)<\frac r {2\epsilon^2}.
\end{equation}
Hence when $r\in ]0,\epsilon]$, $0<\overline u'_\epsilon(r)<(2\epsilon)^{-1}$; while by the second inequality in (\ref{171271}),
when $r\in ]\epsilon,1[$, $0<\overline u_\epsilon'(r)<r^{-1} <\epsilon^{-1}$.  Consequently,  the second
inequality in (2) holds and the first one can be derived as follows:
\begin{equation*}
0<-\overline u_\epsilon(0)=\overline u_\epsilon(1)-\overline
u_\epsilon(0)=\int_0^1\overline u_\epsilon'(r)dr\leq \epsilon^{-1}.
\end{equation*}
Finally we rewrite equation (\ref{420}) as $\overline u''_\epsilon=F-r^{-1}\overline u'_\epsilon$. Then the first inequality in (1) and (\ref{1001}) imply $-\epsilon^{-2}/2\leq \overline u''_\epsilon(r)\leq \epsilon^{-2}$ which implies the third one in (2).
\end{proof}

For any $R<1$ and nonnegative integer $k$, denote $R_k=R+\frac{1-R}{k+1}$ and for simplicity denote $B_{R_k}(0)$ by $B_k$.
By the method in \cite[p.273-275]{jost}, we have the following inequalities.

\begin{lemm}\label{20180708}
For any $R<1$ and integer $k\geq 2$, there exists a constant $C$ depending on $R$ and $k$ such that
\begin{equation}\label{20180705}
\int_{B_k}\mid\!\bigtriangledown^k \overline u_\epsilon\!\mid^2\leq 3\int_{B_{k-1}}\mid\!\bigtriangledown^{k-2}\bigtriangleup\overline u_\epsilon\!\mid^2
+C\int_{B_{k-1}}\mid\!\bigtriangledown^{k-1}\overline u_\epsilon\!\mid^2.
\end{equation}
\end{lemm}

\begin{proof}
Let $\chi_k(r)$ be a cut-off function as in \cite{jost}  with
\begin{equation*}
\begin{aligned}
&0\leq \chi_k\leq 1,\\
&\chi_k (r)=1\ \ \textup{for}\ r\in (0,R_k),\\
&\chi_k(r)=0\ \ \textup{for}\ r\in (R_{k-1},1),\ \ \textup{and}\\
&|\!\bigtriangledown \chi_k|\leq \frac{2k(k+1)}{1-R}.
\end{aligned}
\end{equation*}
Since $\chi_k$ has a compact support in $B_{k-1}$, for a smooth function $f$ defined on $B_1(0)$,
by using the integration by parts and Stokes' theorem  we get
\begin{equation*}
\begin{aligned}
\int_{B_{k-1}}\chi_k^2\mid\!\bigtriangledown^2 f\!\mid^2=&\int_{B_{k-1}}\chi_k^2(\bigtriangleup f)^2+2\int_{B_{k-1}}\chi_k\bigtriangleup f\sum_{j=1}^2\frac{\partial\chi_k}{\partial x_j}\frac{\partial f}{\partial x_j}\\
&-2\int_{B_{k-1}}\chi_k\sum_{i,j=1}^2\frac{\partial\chi_k}{\partial x_i}\frac{\partial f}{\partial x_j}\frac{\partial^2 f}{\partial x_i\partial x_j}.
\end{aligned}
\end{equation*}
By the interpolation inequality we have
\begin{equation*}
\int_{B_{k-1}}\chi_k^2\mid\!\bigtriangledown^2 f\!\mid^2\leq \frac 3 2\int_{B_{k-1}}\chi_k^2(\bigtriangleup f)^2+\frac 12\int_{B_{k-1}}\chi_k^2\mid\!\bigtriangledown^2 f\!\mid^2+\frac 1 2C\int_{B_{k-1}}\mid\!\bigtriangledown f\!\mid^2
\end{equation*}
or
\begin{equation}\label{171275}
\int_{B_{k-1}}\chi_k^2\mid\!\bigtriangledown^2 f\!\mid^2\leq 3\int_{B_{k-1}}\chi_k^2(\bigtriangleup f)^2
+C\int_{B_{k-1}}\mid\!\bigtriangledown f\!\mid^2,
\end{equation}
where $C=8\parallel\!\bigtriangledown\chi_k\!\parallel_{C^0}$ depends on $R$ and $k$.

When $k\geq 2$, we write
\begin{equation*}
\begin{aligned}
\int_{B_{k-1}}\chi_k^2\mid\!\bigtriangledown^k \overline u_\epsilon\!\mid^2=&\int_{B_{k-1}}\chi_k^2\sum_{i_1,\cdots, i_k=1}^2\Bigl(\frac{\partial^k \overline u_\epsilon}
{\partial x_{i_1}\cdots\partial x_{i_k}} \Bigr)^2\\
=&\sum_{i_3,\cdots,i_k=1}^2\int_{B_{k-1}}\chi_k^2\,\Bigl|\bigtriangledown^2\Bigl(\frac{\partial^{k-2}\overline u_\epsilon}{\partial x_{i_3}\cdots\partial x_{i_k}}\Bigr)\Big|^2.
\end{aligned}
\end{equation*}
Then using inequality (\ref{171275}) for $f=\frac{\partial^{k-2}\overline u_\epsilon}{\partial x_{i_3}\cdots\partial x_{i_k}}$ yields
\begin{equation*}
\int_{B_{k-1}}\chi_k^2\mid\!\bigtriangledown^k \overline u_\epsilon\!\mid^2\leq 3\int_{B_{k-1}}\chi_k^2\mid\!\bigtriangledown^{k-2}\bigtriangleup\overline u_\epsilon\!\mid^2
+C\int_{B_{k-1}}\mid\!\bigtriangledown^{k-1}\overline u_\epsilon\!\mid^2,
\end{equation*}
which implies inequality (\ref{20180705}).
\end{proof}


\begin{prop}\label{20180902}
For any $R<1$ and positive integer $k$ there exists a constant $C$ depending on $R$ and $k$ such that
for sufficiently small $\epsilon>0$,
\begin{equation}\label{20180707}
\parallel\!\overline u_\epsilon\!\parallel_{C^k(B_{k+2})}\leq C\epsilon^{-3k+2}.
\end{equation}
\end{prop}

\begin{proof}
By the second inequality in (2) of Lemma \ref{1005}, we have
\begin{equation*}
\parallel\!\bigtriangledown \overline u_\epsilon\!\parallel_{C^0(B_1)}\leq \parallel\!\overline u_\epsilon'(r)\!\parallel_{C^0(B_1)}\parallel\!\bigtriangledown r\!\parallel_{C^0(B_1)}\leq C\epsilon^{-1}.
\end{equation*}
Combined with the first inequality in (2) of Lemma \ref{1005}, we find that inequality (\ref{20180707})  for $k=1$ holds.

We use the inductive method to prove the proposition. Assume that (\ref{20180707}) holds for any $k\leq m$. We will use the Sobolev inequality to prove that it  also holds for $k=m+1$.
To this end, we will use the inequality in Lemma \ref{20180708} to prove
\begin{equation}\label{20180801}
\parallel\!\bigtriangledown^{m+l}\overline u_\epsilon\!\parallel_{L^2(B_{m+l})}\leq C\epsilon^{-3(m+1)+5-l} \quad \textup{for\ \ $l=1,\,2,\,3$}.
\end{equation}
Here the constant $C$ depends only on $R$ and $m$ and will be used in the generic sense. Assume that the positive $\epsilon$ is small enough.
By (\ref{171273}) we write
\begin{equation*}
\begin{aligned}
\bigtriangledown^{m+l-2}\bigtriangleup\overline u_\epsilon=&\sum (FC_1(i_1\cdots i_j)+F_1C_2(i_1\cdots i_j))
\mathcal F(i_1\cdots i_j)\\
& +\sum F_2C_3(i_1\cdots i_j)\bigtriangledown\!r^2\ \mathcal G_1(i_1\cdots i_j)\\
&+\sum F_2C_4(i_1\cdots i_j)\bigtriangledown^2\!r^2\ \mathcal G_2(i_1\cdots i_j),
\end{aligned}
\end{equation*}
where $C_1,\,C_2,\,C_3$ and $C_4$ are constants only depending on $m$, $\mathcal F(i_1\cdots i_j)$, $\mathcal G_1(i_1\cdots i_j)$ and $\mathcal G_2(i_1\cdots i_j)$ are used to denote
\begin{equation*}
\bigtriangledown^{i_1}\overline u_\epsilon\cdots\bigtriangledown^{i_j}\overline u_\epsilon, \quad \textup{where\ \
$i_1\geq \cdots\geq  i_j\geq 1$}
\end{equation*}
and $i_1+\cdots+ i_j$, respectively,  equals to $m+l-2,\,m+l-3$ and $m+l-4$.
Hence by the inequalities in (1) of Lemma \ref{1005}, we have
\begin{equation}\label{20181001}
\parallel\!\bigtriangledown^{m+l-2}\bigtriangleup\overline u_\epsilon\!\parallel_{L^2}\leq C\epsilon^{-2}\sum \parallel\!\mathcal  F\!\parallel_{L^2}+C\epsilon^{-2}\sum \parallel\!\mathcal G_1\!\parallel_{L^2}+C\sum\parallel\!F_2\mathcal G_2\!\parallel_{L^2}
\end{equation}
where the domain of integration is $B_{m+l-1}$ which has been  omitted.

Now we prove inequalities (\ref{20180801}).
The proof of the cases where $l=1$ and $l=2$ are easier than of the case where $l=3$. Hence we prove the case where $l=3$ and omit the other two cases.
 We first deal with a term $\mathcal F=\mathcal F(i_1\cdots i_j)$ where $i_1+\cdots i_j=m-1$. If $j=1$, then $\mathcal F=\mathcal F(m+1)=\bigtriangledown^{m+1}\overline u_\epsilon$.
Hence by (\ref{20180801}) for $l=1$, we have
\begin{equation}\label{20180806}
\parallel\!\mathcal F\!\parallel_{L^2}=\parallel\!\bigtriangledown^{m+1}\overline u_\epsilon \!\parallel_{L^2}\leq C\epsilon^{-3(m+1)+4}.
\end{equation}If $j\geq 2$, then $i_1,\cdots, i_j$
are less than or equal to $m$. Hence by the inductive assumption, we have
\begin{equation}\label{20180807}
\parallel\!\mathcal F\!\parallel_{L^2}\leq C\parallel\!\bigtriangledown^{i_1}\overline u_\epsilon\!\parallel_{C^0}\cdots
\parallel\!\bigtriangledown^{i_j}\overline u_\epsilon\!\parallel_{C^0}\leq C\epsilon^{-3(m+1)+2j}\leq C\epsilon^{-3(m+1)+4}.
\end{equation}
As the same reason, for a term $\mathcal G_1=\mathcal G_1(i_1\cdots i_j)$, since $i_1+\cdots+i_j=m$,  we have
\begin{equation}\label{20180808}
\parallel\!\mathcal G_1\!\parallel_{L^2}\leq C\epsilon^{-3m+2j}\leq C\epsilon^{-3(m+1)+5};
\end{equation}
and for a term $\mathcal G_2=\mathcal G_2(i_1\cdots i_j)$, since $i_1+\cdots+i_j=m-1$, we have
\begin{equation}\label{20180805}
\parallel\!F_2\mathcal G_2\!\parallel_{L^2}\leq C\parallel\bigtriangledown^{i_1}\overline u_\epsilon\!\parallel_{C^0}\cdots\parallel\!
\bigtriangledown^{i_j}\overline u_\epsilon\!\parallel_{C^0}\parallel\!F_2\!\parallel_{L^2}\leq C\epsilon^{-3(m-1)+2j}\parallel\!F_2\!\parallel_{L^2}.
\end{equation}
However, by the second equation in (\ref{171272}),
\begin{equation*}
\parallel\!F_2\!\parallel_{L^2}^2\leq \int_{B_0}\epsilon^{-4}e^{-2\overline u_\epsilon}dx_1dx_2=2\pi\epsilon^{-4}\int_0^1e^{-2\overline u_\epsilon}rdr.
\end{equation*}
Using the integration by parts we have
\begin{equation*}
\parallel\!F_2\!\parallel_{L^2}^2\leq \pi\epsilon^{-4}e^{-2\overline u_\epsilon}r^2|_{0}^1+2\pi\epsilon^{-4}\int_0^1e^{-2\overline u_\epsilon}r^2\overline u_\epsilon'(r)dr\leq C\epsilon^{-5},
\end{equation*}
where the last inequality follows by  the third inequality in (1) and the second one in (2) of Lemma \ref{1005}.
Combined with (\ref{20180805}), we have
\begin{equation}\label{20180809}
\parallel\!F_2\mathcal G_2\!\parallel_{L^2}\leq C\epsilon^{-3(m+1)+\frac {11}{2}}.
\end{equation}
Therefore  inserting (\ref{20180806})--(\ref{20180808}) and (\ref{20180809}) into (\ref{20181001}), we find
\begin{equation*}
\parallel\!\bigtriangledown^{m+1}\bigtriangleup\overline u_\epsilon\!\parallel_{L^2}\leq C\epsilon^{-3(m+1)+2}.
\end{equation*}
Then inserting the above inequality and (\ref{20180801}) for $l=2$ into the inequality in  Lemma \ref{20180708} for $k=m+3$, we obtain  inequality (\ref{20180801}) for $l=3$.

Now by the Sobolev inequality \cite[P.158]{GT}, inequalities (\ref{20180801}) and the inductive assumptions imply
\begin{equation*}
\parallel\!\overline u_\epsilon\!\parallel_{C^{m+1}(B_{m+3})}
\leq C\sum_{k=0}^{m+3}\parallel\!\bigtriangledown^{k}\overline
u_\epsilon\!\parallel_{L^2(B_{m+3})}\leq C \epsilon^{-3(m+1)+2}.
\end{equation*}
This proves the proposition for $k=m+1$.
\end{proof}

The following lemma will be used in the last section.
\begin{lemm}\label{20180901}
For any integer $p\geq 2$,
\begin{equation*}
\int_{B_0}e^{-p\overline u_\epsilon}dx_1dx_2\leq \pi p!\epsilon^{-p+1}.
\end{equation*}
\end{lemm}
\begin{proof}
As in the above proof, using the integration by parts, we have
\begin{equation*}
\int_{B_0}e^{-p\overline u_\epsilon}dx_1dx_2=2\pi\int_0^1e^{-p\overline u_\epsilon}rdr=\pi p\int_0^1r^2e^{-p\overline u_\epsilon}\overline u_\epsilon' dr.
\end{equation*}
Since $0<r\overline u_\epsilon<1$ and $|\overline u_\epsilon'(r)|<\epsilon^{-1}$, we get
\begin{equation*}
\int_{B_0}e^{-p\overline u_\epsilon}dx_1dx_2\leq \pi p\epsilon^{-1}\int_0^1re^{-(p-1)\overline u_\epsilon}dr.
\end{equation*}
Then the iterated method gives the conclusion of the lemma.
\end{proof}

\section{Construction of a family of Hermitian metrics}
In this section, if $H$ is a Hermitian metric on $V$, we will denote
 the associated  Hermitian connection by $D_H$,  and the curvature forms of $D_H$ with respect to
  $(\hat\mu_1^\alpha,\hat\mu_2^\alpha)$ and
$(\tilde\mu_1^\alpha,\tilde\mu_2^\alpha)$ by $\hat\Theta(H)$ and
$\tilde\Theta(H)$ which are $2\times 2$ matrix valued
2-forms on $\mathcal U_\alpha$.

Following the convention in  Section 2, $\xi_a$ is a branched point
on $B$ and $\xi_j$ is  a point in the support of $D_1$. Let
\begin{equation*}
D=\sum_{a=1}^{n} \xi_a-4\sum_{j=n+1}^{5n/4}\xi_{j}
\end{equation*}
be a new divisor  of degree zero on $B$. Let $G$ be the Green
function of $D$ (cf. \cite[p.339-340]{La}). Its local
expansion near $\xi\lalp$ for $1\leq \alpha\leq 5n/4$ takes the form:
\begin{equation}\label{427}
G(z\lalp)=-c\lalp\log |z_\alpha|+2g_\alpha(z\lalp)
\end{equation}
for the constant $c_\alpha=1$ (resp. $-4$) for $\alpha=a$ (resp. $j$)
and  some harmonic function $g\lalp$. We fix $r_0>0$ small enough so
that $G|_{U_\alpha}$ has the above local expansion.

We now construct a Hermitian metric on $V$ by using the Green function
$G$ and the HYM metrics $h_\epsilon^a$, which are denoted as
$h_\epsilon$ in Section 3. We define $h_0$ to be the metric on
$V|_{\mathcal U_0}$ given by a Hermitian matrix valued function in
$(\hat\mu_1^0,\hat\mu_2^0)$:
\begin{equation*}
\hat h_0=e^{\frac 1 2G}I,
\end{equation*}
where $I$ is the $2\times 2$ identity matrix. In this way, the
ambiguity of choosing $(\hat\mu_1^0,\hat\mu_2^0)$ in Section 2  is
irrelevant.

By (\ref{04}) and the notation in
(\ref{120850}), the Hermitian matrix of $h_0$ in
$(\tilde\mu_1^0,\tilde\mu_2^0)$ is
\begin{equation*}
\tilde{ h}_0=(A_0)^t{\hat h}_0\overline{A}_0.
\end{equation*}
Since $G$ is harmonic, direct calculation as in Section 3 gives
\begin{eqnarray}\label{g1}
\begin{aligned}
\label{432} \tilde\Theta(h_0)=\hat\Theta(h_0)= &-\pi
i\begin{pmatrix}\frac{\partial w_1^\ast(z)}{\partial z}
&0\\0&\frac{\partial w_2^\ast(z)}{\partial z}
\end{pmatrix}dz\wedge d\overline{w}-\pi i\begin{pmatrix}\frac{\partial
\overline{w^\ast_1(z)}}{\partial\overline z} &0\\0&\frac{\partial
\overline{w^\ast_2(z)}}{\partial \overline{z}}
\end{pmatrix} d\overline{z}\wedge d w.
\end{aligned}
\end{eqnarray}
Hence, $h_0$ is a HYM metric on $V|_{\cU_{0}}$. For $n+1\leq j\leq
5n/4$, because of (\ref{120810}),  $h_0|_{\cU_j\cap\mathcal U_0}$ in the frame
$(\hat\mu^j_1,\hat\mu^j_2)$ is given by
the matrix valued function
\begin{equation*}
\label{429} \hat h_j=e^{g_j}I.
\end{equation*}
In this way, $h_0$ extends to a smooth metric on $V|_{\cU_j}$.
However,  because of (\ref{120811}),  $h_0|_{\mathcal U_a\cap\mathcal U_0}$ in
the frame $(\hat\mu^a_1,\hat\mu^a_2)$ has the form
\begin{equation*}
\label{430}
 \hat h_a=e^{g_a(z_a)}
\begin{pmatrix}|z_a|^{-\frac{1}{2}}&0\\0&|z_a|^{\frac{1}{2}}\end{pmatrix}.
\end{equation*}
Clearly, $h_0$ can not extend to $V|_{T^\ast\times \{\xi_a\}}$. As
stated  in Section 3, we found a new HYM metric $h_\epsilon^a$ of
$V|_{\cU_a}$ which  in $(\hat\mu^a_1,\hat\mu^a_2)$
has the form
\begin{equation*}
\label{431} \hat h_\epsilon^a=
\begin{pmatrix}e^{-u_\epsilon} &0\\0&e^{u_\epsilon}
\end{pmatrix},
\end{equation*}
where  $u_\epsilon$ is the solution to  equation (\ref{418}). Let
$h_{a,\epsilon}=e^{g_a} h^a_\epsilon$. Then $h_{a,\epsilon}$ is also
a HYM metric on $V|_{\cU_a}$.

We then interpolate  the two metrics $h_0$ and $h_{a,\epsilon}$ over
$\cU_a$. Let
\begin{equation*}
\rho:]0,(2r_0)^2[\rightarrow [0,1]
\end{equation*}
be a fixed $C^\infty$ cut-off function with $\rho(r^2)=1$ for
$r<r_0$, $\rho(r^2)=0$ for $r\geq \frac 43 r_0$. We  define on
$V|_{\mathcal U_a}$
\begin{equation*}
\bh_\epsilon|_{\cU_a}=(1-\rho(|z_a|^2))h_0+\rho(|z_a|^2)
h_{a,\epsilon}.
\end{equation*}
This is a smooth Hermitian metric on $V|_{\cU_a}$ that coincides
with $h_0$ for $|z_a|\geq\frac 43 r_0$ and coincides with
$h_{a,\epsilon}$ for $|z_a|\leq r_0$. After working this out for all
branched points, we obtain a global Hermitian metric $\bh_\epsilon$
that is $h_0$ on $V|_{ X-\cup_1^n \cU_a(\frac 43 r_0)}$ and
$h_{a,\epsilon}$ on $V|_{\cU_{a}(r_0)}$. Here, we denote by
$\cU_{a}(r)$ the pre-image in $ X$ of $U_a(r)$,  which is the disc
in $B$ with  center $\xi_a$ and radius $r$.   From now on we denote
$U_0=B-\bigl(D_1\cup(\cup_{a=1}^nU_a(\frac 3 2 r_0)\bigr)\bigr)$ and
$\mathcal U_0=T^2\times U_0$. Then $\mathcal U_0$, $\mathcal U_i$
and $\mathcal U_a$ still form a cover of $X$. We take the
corresponding trivialization of $V$ for this cover.

Hence, over $\mathcal U_0$ and $\mathcal U_j$,
$\hat\Theta(\bh_\epsilon)=\hat\Theta(h_0)$. Over $\cU_a$, direct
calculation as in Section 3 gives
\begin{equation}\label{433}
\begin{aligned}
\hat\Theta(\bh_\epsilon)=&-\begin{pmatrix}\frac{\partial^2 \phi_1}{\partial z_a\partial
\overline{z_a}} &0\\0&\frac{\partial^2\phi_2}{\partial z_a\partial
\overline {z_a}}
\end{pmatrix}dz_a\wedge d\overline{z_a}+\pi^2\bl r^2\kappa^{-4}-\kappa^4\br\begin{pmatrix}1&0\\0&-1
\end{pmatrix}dw\wedge d\overline{w}\\
&
-\pi i\begin{pmatrix}0&1-z_a\frac{\partial(4\ln\kappa)}{\partial
z_a}
\\ \frac{\partial(4\ln\kappa)}{\partial z_a}&0
\end{pmatrix}dz_a\wedge d\overline{w}\\
&-\pi
i\begin{pmatrix}0&\kappa^4\frac{\partial(4\ln\kappa)}{\partial
\overline{z_a}}
\\ \kappa^{-4}(1-\overline{z_a}\frac{\partial(4\ln\kappa)}{\partial \overline{z_a}})&0
\end{pmatrix}d\overline{z_a}\wedge dw,
\end{aligned}
\end{equation}
where $r=|z_a|$,
\begin{equation}\label{0000}
\phi_1=\ln\bl(1-\rho)r^{-\frac{1}{2}}+\rho
e^{-u_\epsilon}\br,\quad \phi_2=\ln\bl(1-\rho)r^{\frac{1}{2}}+\rho
e^{u_\epsilon} \br
\end{equation}
and
\begin{equation}\label{1003}
\kappa=e^{\frac 1 4 (\phi_2-\phi_1)}.
\end{equation}
Notice that when restricted on $\mathcal U_a(r_0)$, $\phi_1=-u_\epsilon$, $\phi_2=u_\epsilon$, $\kappa^4=e^{2u_\epsilon}$. Hence in this case $\hat\Theta(\mathbf h_\epsilon)$ in (\ref{433}) is indeed equal to $\hat\Theta(h_\epsilon)$ in (\ref{171225}).
Also notice  that near the boundary of $\cU_a$, $\phi_1$ and
$\phi_2$ reduces to $-\frac{1}{2}\ln r$ and $\frac{1}{2}\ln r$,
respectively, and hence $\phi_1+\phi_2$ vanishes. Thus, $\phi_1+\phi_2$ can be viewed as a function of $X$ by defining it to be zero
on $X-\cup_1^n \mathcal U_a$.  This convention will be used in the
following normalization.

The metric $\bh\lep$ should be modified conformally. From (\ref{433})
we have
\begin{equation*}
\label{437}
\Tr\bigl(\frac i 2\Lambda_{\omega_\epsilon}\hat\Theta(\bh_\epsilon)\bigr)=-\epsilon\frac{\partial^2(\phi_1+\phi_2)}{
\partial z_a \partial\overline{z_a}}.
\end{equation*}
To make it vanish, we normalize $\bh_\epsilon$ conformally by
the factor $e^{-\frac 1 2(\phi_1+\phi_2)}$:
\begin{equation*}
\label{439} H_{0,\epsilon}=e^{-\frac 1
2(\phi_1+\phi_2)}\cdot \bh_\epsilon.
\end{equation*}
Hence
\begin{equation}\label{171011}
\hat\Theta(H_{0,\epsilon})=\frac 1 2\partial\overline\partial(\phi_1+\phi_2) I+\hat\Theta(\mathbf h_\epsilon).
\end{equation}
Consequently
\begin{equation}\label{440}
\Tr(\Lambda_{\omega_\epsilon}\hat\Theta(H_{0,\epsilon}))=0.
\end{equation}
Moreover, by the construction, over $\mathcal U_0$, $\mathcal U_j$ and $\mathcal U_a(r_0)$,
$\Lambda_{\omega_\epsilon}{\hat\Theta}(H_{0,\epsilon})=0$, and over $\mathcal U_a-\mathcal U_a(r_0)$,
\begin{equation}\label{444}
\frac i 2 \Lambda_{\omega_\epsilon}{\hat\Theta}(H_{0,\epsilon})=\psi
\left(\begin{array}{cc} 1&0\\
0&-1\end{array}\right)
\end{equation}
for the function
\begin{equation}\label{02}
\psi=\frac{1}{\epsilon}\pi^2\bl r^2\kappa^{-4}-\kappa^4\br
+\frac{\epsilon}{2}\frac{\partial^2(4\ln\kappa)}{\partial
z_a\partial\overline{z_a}}.
\end{equation}
Clearly, by definition (\ref{1003}) and equation (\ref{417}), $\psi$ is zero near the boundary of $\cU_a-\cU_a(r_0)$.
Hence $\psi$ can be defined on whole $X$ by zero extension. After this, (\ref{444}) holds on whole $X$.

Definitions (\ref{0000}) and (\ref{1003}) give
\begin{equation*}
\kappa^{-4}=e^{\phi_1-\phi_2}=r^{-1}\phi
\end{equation*}
for
\begin{equation}\label{171018}
\phi=\frac{1+\rho(e^{-(u_\epsilon-\frac 1 2 \ln
r)}-1)}{1+\rho(e^{u_\epsilon-\frac 1 2 \ln
r}-1)}=1+O\bigl(u_\epsilon-\frac 1 2 \ln r\bigr).
\end{equation}
 Since $\ln r$ is
harmonic, we get
\begin{equation*}
\psi=\frac 1 \epsilon \pi^2 r (\phi-\phi^{-1})-\frac \epsilon 2
\frac{\partial^2\ln \phi}{\partial z_a\partial\overline{z_a}}.
\end{equation*}
Then by estimate (1) in Theorem \ref{solution},  the
function $\psi$ satisfies that, for any positive integer $l$ and any
nonnegative integer $k$ with $ l>k$, there is a constant
$C=C(r_0,l,k)$ such that  for any $0<\epsilon<1/8$,
\begin{equation*}
\parallel\! \psi\!\parallel_{C^k([r_0,2r_0])}\leq
C\epsilon^{l-k-1}.
\end{equation*}
So
\begin{equation}\label{445}
\parallel\!\psi\!\parallel_{C^k(X)}\leq C\epsilon^{l-k-1}.
\end{equation}
Therefore, we immediately have
\begin{equation}\label{800}
\parallel\!\Lambda_{\omega_\epsilon}\hat\Theta(H_{0,\epsilon})\!\parallel_{C^k}
\leq C\epsilon^{l-k-1}.
\end{equation}
Since by (\ref{a8}),
$\tilde\Theta(H_{0,\epsilon})=A_\alpha^{-1}\hat\Theta(H_{0,\epsilon})A_\alpha$,
and $A_\alpha$ does not depend on $\epsilon$, we also have
\begin{equation}\label{120960}
\parallel\!\Lambda_{\omega_\epsilon}\tilde\Theta(H_{0,\epsilon})\!\parallel_{C^k}
\leq C\epsilon^{l-k-1}.
\end{equation}

Finally, by the construction, $(\hat\mu_1^\alpha,\hat\mu_2^\alpha)$
is orthogonal for $H_{0,\epsilon}$. It can be normalized to a
unitary frame $(\check \mu_1^\alpha,\check\mu_2^\alpha)$:
\begin{equation}\label{120961}
(\hat\mu_1^\alpha,\hat\mu_2^\alpha)=(\check\mu_1^\alpha,\check\mu_2^\alpha)N_\alpha,
\end{equation}
where
\begin{equation}\label{801}
N_0=e^{\frac 1 4 G}I,\quad N_j=e^{\frac 1 2 g_j}I, \quad\textup{and}\quad
N_a=e^{\frac 12 g_a}\begin{pmatrix}\kappa^{-1} &0\\
0& \kappa \end{pmatrix}.
\end{equation}
Combining (\ref{120961}) with (\ref{120850}) yields
\begin{equation}\label{802}
(\tilde\mu_1^\alpha,\tilde\mu_2^\alpha)=(\check\mu_1^\alpha,\check\mu_2^\alpha)B_\alpha\qquad\textup{for}\ \
B_\alpha=N_\alpha A_\alpha.
\end{equation}
If $\check\Theta(H_{0,\epsilon})$ denotes the curvature form of
$D_{H_{0,\epsilon}}$ with respect to
$(\check\mu_1^\alpha,\check\mu_2^\alpha)$, then (\ref{a8}) gives
\begin{equation}\label{a7}
\check\Theta(H_{0,\epsilon})=N_\alpha\hat\Theta(H_{0,\epsilon})N_\alpha^{-1}.
\end{equation}
Hence, as $N_\alpha$ and
$\Lambda_{\omega_\epsilon}\hat\Theta(H_{0,\epsilon})$ are  diagonal
matrices, we have
\begin{equation}\label{900}
\Lambda_{\omega_\epsilon}\check\Theta(H_{0,\epsilon})=\Lambda_{\omega_\epsilon}\hat\Theta(H_{0,\epsilon}).
\end{equation}
Thus, by (\ref{800}), we get the following proposition, which implies Theorem \ref{t0} if we replace $l-k-1$ by $l$.

\begin{prop}\label{1209064}
For any positive integer $l$ and nonnegative positive integer $k$
satisfying $l>k$, there is a constant $C=C(r_0,l,k)$ such that for
any $0<\epsilon<1/8$,
\begin{equation*}
\label{443} \parallel\!
\Lambda_{\omega_\epsilon}{\check\Theta}(H_{0,\epsilon})\!\parallel_{C^k}<C\epsilon^{l-k-1}.
\end{equation*}
\end{prop}

\section{Preparation for estimates}

Since  $V$ is stable with respect to the
K\"ahler metric $\omega_\epsilon$, it admits a HYM metric
$H_{1,\epsilon}$ which is unique up to a scale.
As $H_{1,\epsilon}$ and $H_{0,\epsilon}$ are Hermitian metrics on
$V$, there exists an element $H_\epsilon\in A^0(\textup{End}(V))$ such that
\begin{equation*}
H_{1,\epsilon}(\cdot\,,\cdot)=H_{0,\epsilon}(H_\epsilon\cdot\,,\cdot).
\end{equation*}

We will use the following notations.

\renewcommand\arraystretch{1.25}

\[
\begin{tabular}{|c|c|c|c|c|c|}
\hline
&  $H_{1,\epsilon}$ & $\Theta(H_{1,\epsilon})$ & $H_{0,\epsilon}$ & $\Theta(H_{0,\epsilon})$ & $H_\epsilon$\\
\hline
\textup{smooth frames} $(\hat\mu^\alpha_1,\hat\mu^\alpha_2)$ & $\hat H_{1,\epsilon}^\alpha$ & $\hat \Theta_{1,\epsilon}^\alpha$ & $\hat H_{0,\epsilon}^\alpha$ & $\hat\Theta_{0,\epsilon}^\alpha$ & $\hat H_\epsilon^\alpha$\\
\hline
\textup{holomorphic frames} $(\tilde \mu^\alpha_1,\tilde\mu^\alpha_2)$ & $\tilde H_{1,\epsilon}^\alpha$ & $\tilde \Theta_{1,\epsilon}^\alpha$ & $\tilde H_{0,\epsilon}^\alpha$ & $\tilde\Theta_{0,\epsilon}^\alpha$ & $\tilde H_\epsilon^\alpha$\\
\hline
\textup{unitary frames} $(\check\mu^\alpha_1,\check\mu^\alpha_2)$ & $\check H_{1,\epsilon}^\alpha$ & $\check \Theta_{1,\epsilon}^\alpha$ & $\check H_{0,\epsilon}^\alpha$ & $\check\Theta_{0,\epsilon}^\alpha$ & $\check H_\epsilon^\alpha$\\
\hline
\end{tabular}
\]
\renewcommand\arraystretch{0.8}

\noindent Here, for example, $\hat H^\alpha_\epsilon$, $\tilde H^\alpha_\epsilon$ and $\check H^\alpha_\epsilon$ are the resulting matrix representations of $H_\epsilon$ in frames $(\hat\mu^\alpha_1,\hat\mu^\alpha_2)$, $(\tilde \mu^\alpha_1,\tilde\mu^\alpha_2)$, and $(\check\mu^\alpha_1,\check\mu^\alpha_2)$, respectively.
In this section, we often drop
the superscript and subscript $\alpha$ when working with a single frame. Hence the notation $\check \Theta_{0,\epsilon}$ denote $\check\Theta(H_{0,\epsilon})$, etc.

Clearly we have the relations
\begin{equation}\label{803}
\check H_{1,\epsilon}=(\check H_\epsilon)^t\quad  \and \quad \tilde
H_{1,\epsilon}=(\tilde H_\epsilon)^t\cdot \tilde H_{0,\epsilon}\ .
\end{equation}
Since $H_{1,\epsilon}$ is the HYM metric, by
(\ref{b0}) and the second identity in (\ref{803}),  direct
computation as in \cite[p.S264]{UY} yields
\begin{eqnarray}\label{05}
\begin{aligned}
0=&\Lambda_{\omega_\epsilon} \tilde\Theta_{1,\epsilon}=\Lambda_{\omega_\epsilon}
\overline\partial(\partial
\tilde{H}_{1,\epsilon}\cdot(\tilde{H}_{1,\epsilon})^{-1})^t\\
=&\Lambda_{\omega_\epsilon}\overline\partial(\tilde{H}^{-1}_\epsilon\cdot \partial
\tilde{H}_\epsilon)+\Lambda_{\omega_\epsilon}\tilde{H}^{-1}_\epsilon\cdot
\tilde\Theta_{0,\epsilon}\cdot {\tilde H}_\epsilon\\
&-\Lambda_{\omega_\epsilon} \tilde{H}^{-1}_\epsilon\cdot\overline\partial \tilde{
H}_\epsilon\cdot \tilde{H}^{-1}_\epsilon\wedge (\partial
\tilde{H}_{0,\epsilon}\cdot
(\tilde{H}_{0,\epsilon})^{-1})^t\cdot \tilde{H}_\epsilon\\
&-\Lambda_{\omega_\epsilon} \tilde{H}^{-1}_\epsilon\cdot (\partial \tilde{
H}_{0,\epsilon}\cdot (\tilde{H}_{0,\epsilon})^{-1})^t\wedge
\overline\partial \tilde{H}_\epsilon.
\end{aligned}
\end{eqnarray}
Taking the trace of the above system and combining with
$\Tr(\Lambda_{\omega_\epsilon}\tilde\Theta_{0,\epsilon})=0$, which is equivalent to (\ref{440}) by (\ref{a8}), we have
\begin{equation*}
\bigtriangleup_\epsilon \ln\det \tilde{H}\lep=0.
\end{equation*}
Here $\bigtriangleup_\epsilon$ is defined as
\begin{equation}\label{o1}
\bigtriangleup_\epsilon= \epsilon\Bigl(\frac{\partial^2}{\partial
x_1^2}+\frac{\partial^2}{\partial x_2^2}\Bigr)+\frac
1\epsilon\Bigl(\frac{\partial^2}{\partial
y_1^2}+\frac{\partial^2}{\partial y_2^2}\Bigr).
\end{equation}
(Hence, our notation here differs from \cite{Si}.) Thus, $\det \tilde{H}\lep$ is a constant.
We normalize $H_{1,\epsilon}$ so that
\begin{equation*}
\det \tilde{H}\lep=1.
\end{equation*}

On the other hand, from (\ref{05}) we also have
\begin{equation}\label{180101}
\begin{aligned}
\textup{Tr}(\Lambda_{\omega_\epsilon}\tilde\Theta_{0,\epsilon}\cdot\tilde H_\epsilon)=&\Lambda_{\omega_\epsilon}\partial\overline\partial \textup{Tr}\tilde H_\epsilon+\Lambda_{\omega_\epsilon}\textup{Tr}
(\overline \partial \tilde H_\epsilon
\cdot \tilde H_\epsilon^{-1}\wedge\partial \tilde H_\epsilon)\\
+&\Lambda_{\omega_\epsilon}\textup{Tr}((\partial \tilde H_{0,\epsilon}\cdot(\tilde H_{0,\epsilon})^{-1})^t\wedge\overline \partial\tilde H_\epsilon)\\
+&
\Lambda_{\omega_\epsilon}\textup{Tr}(\overline\partial \tilde H_\epsilon\cdot \tilde H_\epsilon^{-1}\wedge (\partial \tilde H_{0,\epsilon}\cdot (\tilde H_{0,\epsilon})^{-1})^t\cdot\tilde H_\epsilon).
\end{aligned}
\end{equation}
Then we have the following inequality (i.e., the inequality (1.9.2) in \cite[p.24]{Siu})
\begin{equation}\label{180102}
\bigtriangleup_\epsilon\ln\textup{Tr}\tilde H_\epsilon\geq -4\mid\!\Lambda_{\omega_\epsilon}\check\Theta_{0,\epsilon}\!\mid.
\end{equation}
Combining it with Proposition \ref{1209064} yields
\begin{equation}\label{20180101}
\bigtriangleup_\epsilon\ln\textup{Tr}\tilde H_\epsilon\geq  -C\epsilon^{l-1}.
\end{equation}

Next we will give the estimate of the Sobolev constant of $X$ with the metric $\omega_\epsilon$.
For a smooth function $f$ on $X$, we use the metric
$\omega_\epsilon$ to define $|df|_\epsilon$ as
\begin{equation}\label{20180109}
|df|_\epsilon^2=\epsilon\Bigl(\Bigl|\frac{\partial f}{\partial
x_1}\Bigr|^2+\Bigl|\frac{\partial f}{\partial x_2}\Bigr|^2\Bigr)+\epsilon^{-1}\Bigl(\Bigr|\frac{\partial
f}{\partial y_1}\Bigr|^2+\Bigl|\frac{\partial f}{\partial y_2}\Bigr|^2\Bigr).
\end{equation}
Note that the $L^p$ norm $\parallel\!f\!\parallel_p$ of $f$ with respect to the volume form
$\frac{\omega_\epsilon^2}2$ is independent of $\epsilon$.

\begin{lemm}\label{sobineq}
 There is a function $I(\epsilon)$ in $\epsilon$
with $I(\epsilon)\geq C\epsilon^{10}$ where $C$ is a constant such
that for any smooth function $f$ on ${X}$,
\begin{equation*}
\parallel\! |df|_\epsilon\!\parallel_2^2\geq
I(\epsilon)(\parallel\! f\!\parallel^2_4-\parallel\!
f\!\parallel^2_2).
\end{equation*}
\end{lemm}
\begin{proof} We shall follow the proof in \cite{Gr}.
Since $X$ has volume one and dimension four,
following the notation of \cite[Lemma 2]{Pli}, for any arbitrary
function $f$ on $ X$, one has
 \begin{equation*}
\parallel\! |df|_\epsilon\!\parallel^2_2\geq D(4)C_2(\parallel\!
f\!\parallel^2_4-\parallel\! f\!\parallel^2_2).
\end{equation*}
Here $D(4)$ is an absolute constant, $C_2=D(4)C^{\frac{1}{2}}_0$ with
$2C_1\geq C_0\geq C_1$ where $C_1$ is
given by the isoperimetric inequality
\begin{equation*}
C_1(\inf\{\vol(M_1),\vol(M_2)\})^3\leq \vol(N)^4
\end{equation*}
with $N$ running  through all codimension one submanifolds dividing
$ X$  into two components $M_1$ and $M_2$. As  ${X}$ is flat and
$\diam({X})=\sqrt{2}\epsilon^{-1}$, \cite[Thm 13]{Cr} implies
 \begin{equation*}
 C_1\geq
 C_4\Bigl(\int_0^{\textup{diam}({X})}r^3dr\Bigr)^{-5}=C_5\epsilon^{20}
 \end{equation*}
for some constants $C_4$ and $C_5$ independent of $\epsilon$. Hence,
$C_0\geq C_5\epsilon^{20}$. In this way, $I(\epsilon)$ can be taken
as
\begin{equation*}
 I(\epsilon)=\inf\{D(4)^2,1\}C_0^{\frac 1 2}\geq C\epsilon^{10}.
 \end{equation*}
\end{proof}

After this, we can use the Morse iteration to prove the following inequality.

\begin{lemm}\label{20180903}
There exists a constant $C$ depending only on $r_0$ and $l$ such that
\begin{equation*}
\sup_{x\in X}\ln\textup{Tr}H_\epsilon (x)\leq (1+C\epsilon^{\frac{l-11}{2}})\int_X\ln\textup{Tr}H_\epsilon(x)\frac{\omega_\epsilon^2}{2!}.
\end{equation*}
\end{lemm}

\begin{proof}
Let $t_\epsilon(x)=2\ln \textup{Tr}H_\epsilon(x)$ and for simplicity we drop the subscript $\epsilon$ of $t_\epsilon$. We rewrite  (\ref{20180101}) as
\begin{equation}-\label{447}
\bigtriangleup_\epsilon t\leq C\epsilon^{l-1},
\end{equation}
where $C$ is a constant depending only on $l$ and $r_0$ which will be used in the generic sense in the following.
Hence we have
\begin{equation*}\label{448}
-\int_{{X}}t^{2p-1}\bigtriangleup_\epsilon t\leq
C\epsilon^{l-1}\int_{{X}}t^{2p-1} \ \ \text{for} \ \ p\geq 1.
\end{equation*}
Here we have omitted the volume form $\frac{\omega_\epsilon^2}{2}$.
(Note that it is independent of $\epsilon$.) Since
 \begin{equation*}
 -\int_{{X}}t^{2p-1}\bigtriangleup_\epsilon t=(2p-1)p^{-2}
\int_{{X}}|d t^{p}|^2_\epsilon,
\end{equation*}
the above inequality implies
\begin{equation*}\label{449}
\int_{{X}}|d t^{p}|_\epsilon^2\leq
 C\epsilon^{l-1}p\int_{{X}}t^{2p-1}.
\end{equation*}
Combined with Lemma \ref{sobineq}, we have
\begin{equation}\lab{20180102}
\parallel \!t\!\parallel^{2p}_{4p}=\parallel\!t^p\!\parallel_4^2\leq \parallel\! t^p\!\parallel_2^2+I(\epsilon)^{-1}\parallel\! |dt^p|_\epsilon\!\parallel_2^2\leq\parallel\! t\!\parallel^{2p}_{2p}+C\epsilon^{l-11}p\int_X t^{2p-1}.
\end{equation}
On the other hand, by H\"older's inequality we have
\begin{equation*}
\int_X t^{2p-1}\leq \Bigl(\int_X t^{2p}\Bigr)^{\frac{2p-1}{2p}}=\parallel\!t\!\parallel_{2p}^{2p-1}\leq \parallel\!t\!\parallel_{2p}^{2p}
\end{equation*}
as $\int_X 1=1$ and $t\geq 2\ln 2>1$.
Hence combined with (\ref{20180102}), we get
\begin{equation*}
\parallel\! t\!\parallel_{4p}^{2p}\leq (1+C\epsilon^{l-11}p)\parallel\! t\!\parallel_{2p}^{2p}.
\end{equation*}
If we set $p=2^m$, then
\begin{equation*}
\label{451}
\parallel\! t\!\parallel^{2}_{2^{m+2}}\leq
(1+C\epsilon^{l-11}2^m)^{\frac{1}{2^m}}\parallel
\!t\!\parallel^{2}_{2^{m+1}}.
\end{equation*}
Iterating this inequality, we obtain
\begin{equation*}
\parallel \!t\!\parallel^{2}_{\infty}\leq\prod_{m=0}^{\infty}
(1+C\epsilon^{l-11}2^m)^{\frac{1}{2^m}}\parallel\!
t\!\parallel^{2}_{2}.
\end{equation*}

It is easy to see that there exist constants $C'$  and $C''$ such that for any sufficiently small $\epsilon>0$,
\begin{equation}\label{20180702}
 \prod_{m=0}^{\infty}
(1+C\epsilon^{l-11}2^m)^{\frac{1}{2^m}}\leq e^{C'\epsilon^{\frac{l-11}{2}}}\leq 1+C''\epsilon^{\frac{l-11}{2}}.
\end{equation}
We denote $C''$ still by $C$.  Hence,
\begin{equation*}\label{453}
\parallel\!t\!\parallel_\infty^2\leq
(1+C\epsilon^{\frac{l-11}{2}})\parallel\!t\!\parallel_2^2
\end{equation*}
which implies
\begin{equation*}
\parallel\! t\!\parallel_{\infty}\leq (1+C\epsilon^{\frac{l-11}{2}})\parallel\!t\!\parallel_1.
\end{equation*}
We finish the proof of the lemma.
\end{proof}

\section{The estimate of the lower bound of the $C^0$-norm}
In this section, we prove the following proposition which implies Theorem \ref{1901004} if $l-6$ is replaced by $l$.

\begin{prop}\label{171097}
For any integer $l>6$, there exists a constant $C$ depending on $l$ and $r_0$ such that for any sufficiently small $\epsilon>0$,
\begin{equation*}
\inf_{x\in X}\tau(x)\leq 2+C\epsilon^{l-6}.
\end{equation*}
\end{prop}

\begin{proof}
We will still drop the superscript and subscript $\alpha$  if it is clear from the context.
We will denote $\tau=\textup{Tr}\check H_\epsilon$. Since $\det \check H_\epsilon=1$, $\tau\geq 2$.

If $\inf_{x\in X}\tau(x)>2$, then at any point $x$ in $X$, one eigenvalue of $\check H_\epsilon(x)$ is $\lambda(x)> 1$ and the other one is $\lambda^{-1}(x)<1$.
Consequently,
the eigenvectors of $\check H_\epsilon(x)$ associated with $\lambda(x)$ form a complex subline bundle $L$ of $V$.
As $c_1(V)=0$, $V$ can be decomposed as the direct sum of $L$ and $L^{-1}$.

We will give a localization of $L$. Fix a $\mathcal U_\alpha$ and denote the $(i,j)$-th entry of  $\check H_\epsilon$ by $h_{i\bar j}$.
As $\lambda+\lambda^{-1}=h_{1\bar 1}+h_{2\bar 2}$,  $h_{1\bar 1}$ or $h_{2\bar 2}$ is less than $\lambda$. Hence, if let
\begin{equation*}
\mathcal U_\alpha'=\{x\in\mathcal  U_\alpha\mid h_{2\bar 2}(x)<\lambda(x)\}\quad\and\quad \mathcal U_\alpha''=\{x\in \mathcal U_\alpha\mid h_{1\bar 1}(x)<\lambda(x)\},
\end{equation*}
then $\mathcal U_\alpha$ is the union of $\mathcal U_\alpha'$ and $\mathcal U_\alpha''$.
If $\mathcal U_\alpha'$ is not empty,  then we define on it a function
\begin{equation*}
\iota'=\bigl((\lambda-h_{2\bar 2})(\lambda-\lambda^{-1})\bigr)^{\frac 1 2}.
\end{equation*}
One can check that  $\iota'^{-1} (\lambda-h_{2\bar 2},h_{2\bar 1})^t$ and $\iota'^{-1}(-h_{1\bar 2},\lambda-h_{2\bar 2})^t$ are two unitary eigenvectors of $\check H_\epsilon$
with eigenvalues $\lambda$ and $\lambda^{-1}$ respectively.  So
\begin{equation*}
\iota'^{-1}((\lambda-h_{2\bar 2})\check \mu_1^\alpha+h_{2\bar 1}\check \mu_2^\alpha)
 \end{equation*}
is a unitary frame of $L|_{\mathcal U_\alpha'}$ with respect to $H_{0,\epsilon}|_L$. Similarly, if $\mathcal U_\alpha''$ is not empty, then we
 also define on it a function
\begin{equation*}
\iota''=\bigl((\lambda-h_{1\bar 1})(\lambda-\lambda^{-1})\bigr)^{\frac 1 2}.
\end{equation*}
Hence
\begin{equation*}
\iota''^{-1}(h_{1\bar 2}\check\mu_1^\alpha+(\lambda-h_{1\bar 1})\check\mu_2^\alpha)
 \end{equation*}
is a  unitary frame of $L|_{\mathcal U_\alpha''}$ with respect to  $H_{0,\epsilon}|_{L}$.
In this way, a localization of $L$ is given.
Since all discussions on $\mathcal U_\alpha'$ and $\mathcal U_\alpha''$ are parallel, we will
concentrate on $\mathcal U_\alpha'$. Hence we will also drop the superscript $'$.

Denote
\begin{equation}\label{171230}
S=\frac 1 {\iota}\left(\begin{array}{cc} \lambda-h_{2\bar 2}& -h_{1\bar 2}\\h_{2\bar 1}& \lambda-h_{2\bar 2}\end{array}\right).
\end{equation}
Then from the above discussions we know that
\begin{equation}\label{171231}
(\mathring\mu_1,\mathring\mu_2)=(\check \mu_1,\check\mu_2)S
\end{equation}
is a unitary frame of $V|_{\mathcal U}$ with respect to $H_{0,\epsilon}$.
Hence, the Hermitian matrix $\mathring H_{0,\epsilon}$ for $H_{0,\epsilon}$ in $(\mathring \mu_1,\mathring\mu_2)$ is the identity matrix.
Denote by $\mathring H_{1,\epsilon}$ the Hermitian matrix for $H_{1,\epsilon}$ in $(\mathring\mu_1,\mathring\mu_2)$.
Since $\check H_{1,\epsilon}=\check H^t$,
by (\ref{04}) and (\ref{171230}) we have
\begin{equation*}
\mathring H_{1,\epsilon}=S^t \check H_{1,\epsilon}\overline S=\Lambda,
\end{equation*}
where $\Lambda$ is a 2-by-2 diagonal matrix whose diagonal entries are $\lambda$ and $\lambda^{-1}$.
Let $\mathring \Theta_{1,\epsilon}$ and $\mathring \Theta_{0,\epsilon}$
be the curvature forms of the Hermitian connections of $H_{1,\epsilon}$ and $H_{0,\epsilon}$ in the frame $(\mathring\mu_1,\mathring\mu_2)$ respectively.
We will drop the subscript $\epsilon$ of  $\mathring\Theta_{1,\epsilon}$ and $\mathring\Theta_{0,\epsilon}$ etc.

De

note $S^{-1}B$ by $T$.  Then combining (\ref{802}) with (\ref{171231}) yields
\begin{equation}\label{170970}
(\tilde \mu_1,\tilde\mu_2)=(\mathring\mu_1,\mathring\mu_2)T.
\end{equation}
Hence by (\ref{04}) again, we have
\begin{align}
&\tilde H_{1,\epsilon}=T^t\mathring H_{1,\epsilon}\overline T=T^t \Lambda \overline T, \label{170930}\\
&\tilde H_{0,\epsilon}=T^t \mathring H_{0,\epsilon}\overline T=T^t\overline  T.\label{170931}
\end{align}
For convenience, denote
\begin{equation*}
\overline\partial T T^{-1}=\left(\begin{array}{cc}\overline s_{11}& \overline s_{12}\\ \overline s_{21}&\overline s_{22}\end{array}\right),
\end{equation*}
where $\overline s_{ij}$ is a $(0,1)$-form on $\mathcal U$.

\begin{lemm}\label{170971}
Let $\mathring\Theta_{1,11}$ and $\mathring\Theta_{0,11}$ be the $(1,1)$-th entries  of $\mathring\Theta_{1}$ and $\mathring\Theta_{0}$ respectively. Then
\begin{align}
&\mathring\Theta_{1,11}=-\partial \overline s_{11}+\overline\partial s_{11}+\lambda^2 s_{12}\wedge\overline s_{12}-\lambda^{-2}s_{21}\wedge\overline s_{21}-\partial\overline\partial\ln \lambda,\label{171260}\\
&\mathring\Theta_{0,11}=-\partial \overline s_{11}+\overline\partial s_{11}+s_{12}\wedge\overline s_{12}-s_{21}\wedge\overline s_{21}.\label{171261}
\end{align}
\end{lemm}

\begin{proof}
Formula (\ref{a8}) combined with (\ref{170970}) yields
\begin{align*}
&\mathring\Theta_{1}=T\tilde\Theta_1 T^{-1}=T\overline\partial(\partial\tilde H_{1,\epsilon} \tilde H_{1,\epsilon}^{-1})^t T^{-1},\\
&\mathring \Theta_0=T \tilde \Theta_0 T^{-1}=T \overline\partial(\partial\tilde H_{0,\epsilon}\tilde H_{0,\epsilon}^{-1})^t T^{-1}.
\end{align*}
Then we use (\ref{170930}) and (\ref{170931}) to expand  the curvature forms $\tilde \Theta_1$ and $\tilde \Theta_0$ respectively. By standard and tedious calculation, we get the conclusions.
\end{proof}

We should understand the term $-\partial\overline s_{11}+\overline\partial s_{11}$ appeared  in (\ref{171260}) and (\ref{171261}).

\begin{lemm}\label{170975}
The first Chern class $c_1(L)$ of $L$ is represented by (the first Chern form of $L$)
\begin{equation}\label{171262}
C_1(L)=\frac{-1}{2\pi i}d(s_{11}-\overline s_{11}).
\end{equation}
\end{lemm}
\begin{proof}
Since $\mathring\mu_1$ and $\mathring\mu_2$ are unitary frames of $(L, H_{0,\epsilon}|_L)$ and $(L^{-1},H_{0,\epsilon}|_{L^{-1}})$ respectively, there exist  real functions $\theta_{\beta\alpha}^1$ and $\theta_{\beta\alpha}^2$ on $\mathcal U_\alpha\cap\mathcal U_\beta$  such that
\begin{equation*}
\mathring\mu_1^\beta=e^{i\theta_{\beta\alpha}^1}\mathring\mu_1^\alpha\quad\and\quad
\mathring\mu_2^\beta=e^{i\theta_{\beta\alpha}^2}\mathring\mu_2^\alpha.
\end{equation*}
Here $\mathcal U_\alpha$ should be replaced by $\mathcal U_\alpha'$ or $\mathcal U_\alpha''$ if necessary. Write
\begin{equation*}
(\tilde \mu_1^\alpha,\tilde\mu_2^\alpha)=(\tilde\mu_1^\beta,\tilde\mu_2^\beta)D_{\alpha\beta}.
\end{equation*}
Then we find
\begin{equation*}
T_\alpha=\left(\begin{array}{cc}e^{i\theta_{\beta\alpha}^1}&0\\
0&e^{i\theta_{\beta\alpha}^2}\end{array}\right) T_{\beta} D_{\alpha\beta}.
\end{equation*}
Since $D_{\alpha\beta}$ are holomorphic, by direct calculation we have
\begin{equation}\label{171240}
\overline\partial T_\alpha T_\alpha^{-1}=i\left(\begin{array}{cc}\overline\partial\theta_{\beta\alpha}^1& 0\\ 0&\overline\partial\theta^2_{\beta\alpha}\end{array}\right)
+\left(\begin{array}{cc}
\overline s_{11}^{\beta}& e^{i(\theta_{\beta\alpha}^1-\theta^2_{\beta\alpha})}\overline s_{12}^\beta\\ e^{i(\theta^2_{\beta\alpha}-\theta^1_{\beta\alpha})}\overline s_{21}^\beta& \overline s_{22}^\beta\end{array}\right).
\end{equation}
Hence
\begin{equation*}
\overline s_{11}^\alpha=\overline s_{11}^\beta+i\overline\partial\theta_{\beta\alpha}^1.
\end{equation*}
So
\begin{equation*}
s_{11}^\beta-\overline s_{11}^\beta=s_{11}^\alpha-\overline s_{11}^\alpha+id\theta_{\beta\alpha}^1.
\end{equation*}
This implies that  $s_{11}^\alpha-\overline s_{11}^\alpha$ is the  connection 1-form of a connection on $L$ with respect to the frame $\mathring\mu_1^\alpha$ (cf. \cite[p.4]{Ko}). Its curvature form is $d(s_{11}^\alpha-\overline s_{11}^\alpha)$.
Thus, by (\ref{171010}) we finish the proof.
\end{proof}

Hence $-\partial\overline s_{11}+\overline \partial s_{11}$ is the $(1,1)$-part of the first Chern form of $L$ which is globally defined on $X$.
From (\ref{171240}) we also have
\begin{equation*}
\overline s_{12}^\alpha=e^{i(\theta_{\beta\alpha}^1-\theta^2_{\beta\alpha})}\overline s_{12}^\beta,
\end{equation*}
which says that  $s_{12}^\alpha\wedge \overline s_{12}^\alpha$ is a globally defined $(1,1)$-form on $X$.
Certainly $s_{21}^\alpha\wedge\overline s_{21}^\alpha$ is also globally defined.

Since $ H_{1,\epsilon}$ is the HYM metric, $\Lambda_{\omega_\epsilon}\mathring\Theta_1=0$. So by (\ref{171260}) we have
\begin{equation}\label{171241}
\frac i 2\Lambda_{\omega_\epsilon}(-\partial \overline s_{11}+\overline\partial s_{11}+\lambda^2 s_{12}\wedge\overline s_{12}-\lambda^{-2}s_{21}\wedge\overline s_{21}-\partial\overline\partial\ln \lambda)=0.
\end{equation}
On a K\"ahler manifold with a K\"ahler metric $\omega$, $\frac i 2\Lambda_\omega (s\wedge\overline s)=|s|_{\omega}^2$ for a $(1,0)$-form $s$.
Hence we can rewrite (\ref{171241}) as
\begin{equation}\label{1709-1}
\frac i 2 \Lambda_{\omega_\epsilon}(-\partial\overline s_{11}+\overline\partial s_{11})+\lambda^2|s_{12}|^2_\epsilon-\lambda^{-2}|s_{21}|^2_\epsilon-\frac 1 4 \bigtriangleup_\epsilon \ln \lambda=0.
\end{equation}
Here for simplicity, we have denoted $|s_{12}|^2_{\omega_\epsilon}$
by $|s_{12}|^2_\epsilon$ and $|s_{21}|^2_{\omega_\epsilon}$ by $|s_{21}|^2_\epsilon$.

On the other hand, we can use the explicit expression (\ref{444}) of the curvature form $\hat\Theta_0$ in Section 5 to calculate $\Lambda_{\omega_\epsilon}\mathring\Theta_{0,11}$.

\begin{lemm}\label{170973} If $\psi$ defined in (\ref{02}) is extended by zero to whole  $X$, then
\begin{equation*}
\frac i 2 \Lambda_{\omega_\epsilon}\mathring\Theta_{0,11}=\frac{h_{1\bar 1}-h_{2\bar 2}}{\lambda-\lambda^{-1}}\psi.
\end{equation*}
\end{lemm}

\begin{proof}
Since $T=S^{-1}B$ and $B=NA$, by (\ref{a8}) we have
\begin{equation}\label{171030}
\mathring \Theta_0=S^{-1}N A\tilde\Theta_0A^{-1} N^{-1} S=S^{-1} N \hat\Theta_0 N^{-1} S.
\end{equation}
Hence, by (\ref{444}) we get
\begin{equation*}
\frac i 2 \Lambda_{\omega_\epsilon} \mathring\Theta_0= S^{-1} N\psi \left(\begin{array}{cc}1&0\\ 0&-1\end{array}\right)N^{-1} S.
\end{equation*}
So the conclusion follows by direct calculation.
\end{proof}

Combining Lemma \ref{170973} with (\ref{171261}) yields
\begin{equation}\label{170976}
\frac i 2\Lambda_{\omega_\epsilon}(-\partial \overline s_{11}+\overline\partial s_{11})+|s_{12}|^2_\epsilon-|s_{21}|^2_\epsilon=\frac{h_{1\bar 1}-h_{2\bar 2}}{\lambda-\lambda^{-1}}\psi.
\end{equation}
Combined with (\ref{1709-1}), we have
\begin{equation*}
(\lambda^2-1)|s_{12}|^2_\epsilon+(1-\lambda^{-2})|s_{21}|^2_\epsilon-\frac 1 4 \bigtriangleup_\epsilon\ln \lambda=\frac{h_{2\bar 2}-h_{1\bar 1}}{\lambda-\lambda^{-1}}\psi,
\end{equation*}
which is the same as the formula in \cite[p.244]{Do2}.
Note that since $\det (h_{i\bar j})=1$, we have
\begin{equation}\label{171040}
(\lambda-\lambda^{-1})^2=(h_{2\bar 2}-h_{1\bar 1})^2+4|h_{1\bar 2}|^2.
\end{equation}
Hence if we denote $\lambda_0$ to be the minimum of the function $\lambda(x)$, then
\begin{equation}\label{170978}
(\lambda^2_0-1)\int_X|s_{12}|^2_\epsilon+(1-\lambda_0^{-2})\int_X|s_{21}|^2_\epsilon\leq \parallel\!\psi\!\parallel_{C^0(X)}(\leq C\epsilon^{l-1}).
\end{equation}
From this inequality we see that if we can prove
$\int_X|s_{12}|^2_\epsilon$ or $\int_X|s_{21}|^2_\epsilon$ is not too small, e.g.,  not less than $\epsilon^{3}$,
 then we finish the proof of the lemma.  To this end, we will use $C_1(L)$.

Since $C_1(L)$ is real, we can  write
\begin{equation}\label{170974}
\begin{aligned}
C_1(L)=&\frac i 2 a_1dz\wedge d\overline z+\frac i 2 a_2dw\wedge d\overline w+\frac i 2 a_3dz\wedge d\overline w+\frac i 2 \overline a_3dw\wedge d\overline z\\
&+\frac i 2a_4dz\wedge dw-\frac i 2\overline a_4d\overline z\wedge d\overline w+d(\theta+\overline\theta),
\end{aligned}
\end{equation}
where $a_1$ and $a_2$ are real numbers, and $a_3$ and $a_4$ are complex numbers.
Since $c_1(L)\in H^2(T,\mathbb Z)$, and $dx\wedge dy$, $dx\wedge du$, etc., form an integral basis of $H^2(T,\mathbb Z)$, by direct calculation
we conclude that $a_1,\,a_2\in\mathbb Z$ and $2a_3,\, 2a_4\in \mathbb Z[i]$.
Note that $L$ depends on $\epsilon$. Hence $a_1,\,a_2,\,a_3,\,a_4$ and  $\theta$ also depend on $\epsilon$.

Combining the $(1,1)$-components of right hand sides in (\ref{170974}) and (\ref{171262}) yields
\begin{equation}\label{171130}
\begin{aligned}
&-\partial\overline s_{11}+\overline\partial s_{11}=\pi(a_1 dz\wedge d\overline z+a_2dw\wedge d\overline w\\
&\qquad \qquad +a_3dz\wedge d\overline w+\overline a_3dw\wedge
d\overline z)-2\pi i(\partial\overline\theta+\overline\partial\theta).
\end{aligned}
\end{equation}
Consequently,
\begin{equation}\label{171080}
\frac i 2\Lambda_{\omega_\epsilon}(-\partial\overline s_{11}+\overline\partial s_{11})=\pi(a_1\epsilon+a_2\epsilon^{-1})+\pi \Lambda_{\omega_\epsilon}(\partial\overline\theta+\overline\partial\theta).
\end{equation}
Combined with (\ref{170976}), we have
\begin{equation}\label{171252}
\pi(a_1\epsilon+a_2\epsilon^{-1})+\pi \Lambda_{\omega_\epsilon}(\partial\overline\theta+\overline\partial\theta)
+|s_{12}|^2_\epsilon-|s_{21}|^2_\epsilon=\frac{h_{1\bar 1}-h_{2\bar 2}}{\lambda-\lambda^{-1}}\psi.
\end{equation}
Integrating over $X$, by Stokes' theorem and (\ref{171040}) we get
\begin{equation}\label{170963}
-\parallel\!\psi\!\parallel_{C^0(X)}\leq \pi(a_1\epsilon+a_2\epsilon^{-1})+\int_X(|s_{12}|^2_\epsilon-|s_{21}|^2_\epsilon)\leq \parallel\!\psi\!\parallel_{C^0(X)}.
\end{equation}

\begin{lemm}\label{171095}
If $|a_1\epsilon+a_2\epsilon^{-1}|>\epsilon^{3}$, then for any integer $l\geq 5$, there exists a constant $C$ depending
 on $l$ and $r_0$ such that for any sufficiently small $\epsilon>0$,
\begin{equation*}
\lambda_0\leq 1+C\epsilon^{l-4}.
\end{equation*}
\end{lemm}
\begin{proof}
If $a_1\epsilon+a_2\epsilon^{-1}<-\epsilon^3$, then by (\ref{170963}),
\begin{equation*}
\int_X |s_{12}|^2_\epsilon\geq \pi \epsilon^3-\parallel\!\psi\!\parallel_{C^0(X)}.
\end{equation*}
Hence, by (\ref{170978}) we have
\begin{equation*}
(\lambda_0^2-1)(\pi\epsilon^3-\parallel\!\psi\!\parallel_{C^0(X)})\leq \parallel\!\psi\!\parallel_{C^0(X)}.
\end{equation*}
According to (\ref{445}), for any integer $l>1$, there exists a constant $C_1$ depending on $r_0$ and $l$ such that $\parallel\!\psi\!\parallel_{C^0(X)}\leq C_1\epsilon^{l-1}$.
Hence if $l\geq 5$ and $\epsilon>0$ is small enough, then $\pi\epsilon^3>2C_1 \epsilon^{l-1}$.
So $\lambda_0^2-1\leq\frac2\pi C_1\epsilon^{l-4}<C_1\epsilon^{l-4}$. Thus $\lambda_0<1+\frac{C_1}{2}\epsilon^{l-4}=1+C\epsilon^{l-4}$.

On the other hand, if $a_1\epsilon+a_2\epsilon^{-1}>\epsilon^3$, then by (\ref{170963}) again we have
\begin{equation*}
\int_X|s_{21}|^2_\epsilon\geq \pi \epsilon^3-\parallel\!\psi\!\parallel_{C^0(X)}.
\end{equation*}
Hence we can also get the conclusion by the similar arguments as the first case.
\end{proof}

In the remainder part of this section, we will prove that if $|a_1\epsilon+a_2\epsilon^{-1}|\leq\epsilon^{3}$, then $\lambda_0\leq 1+C\epsilon^{l-6}$.
The strategy is to estimate $\int_X\bigl(\mathring\Theta_{0,11}-\frac 1 2 \textup{Tr}\hat\Theta_0\bigr)^2$ by two methods separately
in Lemmas \ref{170980} and \ref{170981}.
In this way we can get a positive lower bound  $C\epsilon^5$ of  $\int_X(|s_{12}|^2_\epsilon+|s_{21}|^2_\epsilon)$.
Then as the proof of the above lemma, we can get the desired estimate of $\lambda_0$.

\begin{lemm}\label{170980}
For any integer $l>1$, there exists a  constant $C$ depending on $l$ and $r_0$ such that for any sufficient small $\epsilon>0$,
\begin{equation*}
\frac{1}{4\pi^2}\int_X\bigl(\mathring\Theta_{0,11}-\frac 1 2 \textup{Tr}\hat\Theta_0\bigr)^2\leq \int_Xc_2(V)-C(r_0)
\epsilon^{4}+C\epsilon^{2l-2},
\end{equation*}
where the constant $C(r_0)$ is positive and only depends on $r_0$.
\end{lemm}

\begin{proof}
On $V|_{\mathcal U_0}$ or $V|_{\mathcal U_j}$, $H_{0,\epsilon}=\mathbf h_\epsilon=h_0$ and hence $\hat\Theta_0=\hat\Theta(h_0)$.
By (\ref{g1}) we have
\begin{equation}\label{171016}
\textup{Tr}\hat\Theta_0=-\pi i(dw_1^\ast(z)+dw_2^\ast(z))\wedge d\overline w-\pi i
(d\overline{w_1^\ast(z)}+d\overline{w_2^\ast(z)})\wedge dw
\end{equation}
and
\begin{equation}\label{171013}
\begin{aligned}
\textup{Tr}(\hat\Theta_0\wedge\hat\Theta_0)=-2\pi^2\bigl(dw_1^\ast(z) \wedge d\overline{w_1^\ast(z)}+dw_2^\ast(z) \wedge d\overline{w_2^\ast(z)}\bigr)\wedge dw\wedge d\overline w.
\end{aligned}
\end{equation}
On $V|_{\mathcal U_a}$, combining (\ref{171011}) with (\ref{433}) yields
$\textup{Tr}\hat\Theta_0=0$
and
\begin{equation}\label{171017}
\begin{aligned}
&\textup{Tr}(\hat\Theta_0\wedge\hat\Theta_0)=8\pi^2(\kappa^4-|z|^2\kappa^{-4})\frac{\partial^2(4\ln \kappa)}{\partial z\partial\overline z}\frac{\omega_\epsilon^2}{2}\\
&\quad\quad +8\pi^2\Bigl(\kappa^{-4}\Bigl|1-z\frac{\partial(4\ln\kappa)}{\partial z}\Bigr|^2+\kappa^4\Bigl|\frac{\partial(4\ln\kappa)}{\partial z}\Bigr|^2\Bigr)\frac{\omega_\epsilon^2}{2}.
\end{aligned}
\end{equation}

Before proceeding to the next step of this proof, we first use the curvature form $\hat\Theta_0$ to revisit $c_1(V)$ and $c_2(V)$.
Recall that on $U_a$, we have picked $w_1^\ast=\sqrt{z_a}$ and $w_2^\ast=-\sqrt{z_a}$. If we take $w_2^\ast(z)=-w_1^\ast(z)$ as two local sections of (\ref{171212}), then by (\ref{171016})  $\textup{Tr}\hat\Theta_0$ also vanishes on $\mathcal U_0$ and $\mathcal U_a$. Hence $c_1(V)=0$. Thus by (\ref{171015}), $c_2(V)$ can be represented by
\begin{equation*}
C_2(V)=\frac{1}{8\pi^2}\textup{Tr}(\hat\Theta_0\wedge\hat\Theta_0).
\end{equation*}
We define on $X$ a real $(2,2)$-current $\Omega$ by
\begin{equation}\label{171246}
\Omega|_{\mathcal U_\alpha}=\textup{Tr}(\hat\Theta_0\wedge\hat\Theta_0)|_{\mathcal U_\alpha}\qquad\textup{for\ \  $\alpha=0$ or $j$},
\end{equation}
and
\begin{equation*}
\Omega|_{\mathcal U_a}=\frac{4\pi^2}{|z_a|}\frac{\omega_\epsilon^2}{2},
\end{equation*}
which is equal to $\textup{Tr}(\hat\Theta_0\wedge\hat\Theta_0)$ in (\ref{171013}) when $w_1^\ast(z)=\sqrt{z_a}$ and $w_2^\ast(z)=-\sqrt{z_a}$. So $\Omega$ is indeed well-defined. Now
\begin{equation}\label{171245}
\int_{\mathcal U_a}\Omega=\int_{U_a}\frac{4\pi^2}{|z_a|}\frac i 2 dz_a\wedge dz_a=8\pi^3\int_0^{2r_0}dr=16\pi^3r_0.
\end{equation}
On the other hand, $\textup{Tr}(\hat\Theta_0\wedge\Theta_0)$ in (\ref{171017}) can be written as
\begin{equation*}
-2\pi^2\partial\overline\partial\sigma\wedge dw\wedge d\overline w\qquad \textup{for\ \  $\sigma=\kappa^4+r^2\kappa^{-4}$}.
\end{equation*}
Near the boundary of $\mathcal U_a$, $\kappa^4=r$ and hence $\sigma=2r$. So
\begin{equation*}
\begin{aligned}
&\int_{\mathcal U_a}\textup{Tr}(\hat\Theta_0\wedge\hat\Theta_0)=8\pi^2\int_{U_a}\frac{\partial^2\sigma}{\partial z_a\partial\overline{z_a}}\frac i 2 dz_a\wedge d\overline{z_a}\\
=&2\pi^2\int_{U_a}\bigl(\sigma''(r)+\frac 1 r \sigma'(r)\bigr)rdrd\theta=8\pi^3r_0\sigma'(2r_0\bigr)=16\pi^3r_0.
\end{aligned}
\end{equation*}
Compared with (\ref{171245}), by definition (\ref{171246}) we know that
\begin{equation*}
\int_X\Omega=\int_X\textup{Tr}(\hat\Theta_0\wedge\hat\Theta_0).
\end{equation*}
As $\dim_{\mathbb R}X=4$, $c_2(V)$ can also be represented by $\frac 1 {8\pi^2}\Omega$ which is clearly equal to
\begin{equation*}
-\frac 1 4 p_{2\ast}(dw^\ast\wedge d\overline{w^\ast}\wedge dw\wedge d\overline w).
\end{equation*}
Thus we can get formula (\ref{10}).

Now we proceed to prove the lemma.
In order to see why the constant $C$ in the lemma is independent  of $\textup{Tr}\hat\Theta_0$, we will not particularly assume that $w_2^\ast(z)=-w_1^\ast(z)$ for
two local sections of (\ref{171212}). Thus in general,
$\textup{Tr}\hat\Theta_0$ does not vanish and is determined by a fixed localization of $\varphi$  in Section 2.

 On $V|_{\mathcal U_0}$ or $V|_{\mathcal U_j}$,
combining (\ref{171030}) with (\ref{g1}) yields
\begin{equation}\label{170985}
\begin{aligned}
\mathring\Theta_{0,11}=&-\pi i\frac{\lambda-h_{2\bar 2}}{\lambda-\lambda^{-1}}(dw_1^\ast\wedge d\overline w+d\overline{ w_1^\ast}\wedge dw)\\
&-\pi i\frac{\lambda-h_{1\bar 1}}{\lambda-\lambda^{-1}}(dw_2^\ast\wedge d\overline w+d\overline{ w_2^\ast}\wedge dw).
\end{aligned}
\end{equation}
Since
\begin{equation*}
\lambda-h_{2\bar 2}=\frac{h_{1\bar 1}-h_{2\bar 2}}{2}+\frac{\lambda-\lambda^{-1}}{2} \ \and\  \lambda-h_{1\bar 1}=\frac{h_{2\bar 2}-h_{1\bar 1}}{2}+\frac{\lambda-\lambda^{-1}}{2},
\end{equation*}
we can rewrite (\ref{170985}) as
\begin{equation*}
\begin{aligned}
\mathring\Theta_{0,11}=&\frac{\pi i}{2}\frac{h_{2\bar 2}-h_{1\bar 1}}{\lambda-\lambda^{-1}}\bigl((dw_1^\ast-dw^\ast_2)\wedge d\overline w+(d\overline{ w_1^\ast}-d\overline{w_2^\ast})\wedge d w\bigr)\\
&-\frac{\pi i}{2}\bigl((dw_1^\ast+dw_2^\ast)\wedge d\overline w+(d\overline{ w_1^\ast}+d\overline{w_2^\ast})\wedge dw\bigr),
\end{aligned}
\end{equation*}
where the second term by (\ref{171016}) is $\frac 12 \textup{Tr}\hat \Theta_0$. Hence,
\begin{equation}\label{171001}
\begin{aligned}
&\frac{1}{4\pi^2}\bigl(\mathring\Theta_{0,11}-\frac 1 2 \textup{Tr}\hat\Theta_0\bigr)^2\\
=&-\frac{1}{8}\frac{(h_{2\bar 2}-h_{1\bar 1})^2}{(\lambda- \lambda^{-1})^2}\bigr((dw_1^\ast\wedge d\overline{ w_1^\ast}+dw^\ast_2\wedge d\overline {w_2^\ast})\wedge dw\wedge d\overline w\\
&\qquad\qquad\qquad \qquad-(dw_1^\ast \wedge d\overline {w_2^\ast}+dw_2^\ast\wedge d\overline {w_1^\ast})\wedge dw\wedge d\overline w\bigr).
\end{aligned}
\end{equation}
Note that
\begin{equation}
\begin{aligned}\label{171242}
&-\frac 1 4 dw_1^\ast\wedge d\overline {w_1^\ast}\wedge dw\wedge d\overline w=\Bigl|\frac{\partial w_1^\ast}{\partial z}\Bigr|^2\frac{\omega_\epsilon^2}{2}\geq 0,\\
& -\frac 1 4 dw_2^\ast\wedge d\overline {w_2^\ast}\wedge dw\wedge d\overline w=\Bigl|\frac{\partial w_2^\ast}{\partial z}\Bigr|^2\frac{\omega_\epsilon^2}{2}\geq 0,
\end{aligned}
\end{equation}
and that by the triangle inequality and Cauchy's inequality,
 \begin{equation}\label{171243}
 \begin{aligned}
 &\frac 1 4 (dw_1^\ast\wedge d\overline {w_2^\ast}+dw_2^\ast\wedge d\overline {w_1^\ast})\wedge dw\wedge d\overline w\\
 =&-\Bigl(
\frac{\partial w_1^\ast}{\partial z}\overline{\frac{\partial w_2^\ast}{\partial z}}+\frac{\partial w_2^\ast}{\partial z}\overline{\frac{\partial w_1^\ast}{\partial z}}\Bigr)\frac{\omega_\epsilon^2}{2}
\leq \Bigl(\Bigl|\frac{\partial w_1^\ast}{\partial z}\Bigr|^2+\Bigl|\frac{\partial w_2^\ast}{\partial z}\Bigr|^2\Bigr)\frac{\omega_\epsilon^2}{2}.
\end{aligned}
\end{equation}
Inserting (\ref{171242}) and (\ref{171243}) into (\ref{171001}), we get
\begin{equation*}
\frac{1}{4\pi^2}(\mathring\Theta_{0,11}-\frac 1 2 \textup{Tr}\hat\Theta_0)^2\leq \frac{(h_{2\bar 2}-h_{1\bar 1})^2}{(\lambda- \lambda^{-1})^2}\Bigl(\Bigl|\frac{\partial w_1^\ast}{\partial z}\Bigr|^2+\Bigl|\frac{\partial w_2^\ast}{\partial z}\Bigr|^2\Bigr)\frac{\omega_\epsilon^2}{2}.
\end{equation*}
Therefore on $V|_{\mathcal U_0}$ or $V|_{\mathcal U_j}$, by (\ref{171040}), (\ref{171242}) and (\ref{171013}) we obtain
\begin{equation}\label{171041}
\begin{aligned}
&\frac{1}{4\pi^2}(\mathring\Theta_{0,11}-\frac 1 2 \textup{Tr}\hat\Theta_0)^2
\leq\Bigl(\Bigl|\frac{\partial w_1^\ast}{\partial z}\Bigr|^2+\Bigl|\frac{\partial w_2^\ast}{\partial z}\Bigr|^2\Bigr)\frac{\omega_\epsilon^2}{2}\\
=&-\frac 1 4 (dw_1^\ast \wedge d\overline{ w_1^\ast}+dw^\ast_2\wedge d \overline{ w_2^\ast})\wedge dw\wedge d\overline w
=\frac{1}{8\pi^2}\textup{Tr}(\hat\Theta_0\wedge\hat\Theta_0).
\end{aligned}
\end{equation}

Next we should concentrate on $V|_{\mathcal U_a}$. Combining (\ref{171030}) with (\ref{171011}) and (\ref{433}) yields
\begin{equation*}
\begin{aligned}
\mathring \Theta_{0,11}=&-\frac 12\frac{h_{2\bar 2}-h_{1\bar 1}}{\lambda-\lambda^{-1}}\frac{\partial^2(4\ln \kappa)}{\partial z\partial\overline z}dz\wedge d\overline z+\frac{h_{2\bar 2}-h_{1\bar 1}}{\lambda-\lambda^{-1}}\pi^2(\kappa^4-r^2\kappa^{-4})dw\wedge d\overline w\\
&-\frac{\pi i}{\lambda-\lambda^{-1}}\Bigl(h_{2\bar 1}\kappa^{-2}
\Bigl(1-z\frac{\partial(4\ln\kappa)}{\partial z}\Bigr)+h_{1\bar 2}\kappa^2\frac{\partial(4\ln\kappa)}{\partial z}\Bigr)dz\wedge d\overline w\\
&-\frac{\pi i}{\lambda-\lambda^{-1}}\Bigl(h_{2\bar 1}\kappa^2\frac{\partial(4\ln\kappa)}{\partial z}+h_{1\bar 2}\kappa^{-2}\Bigl(1-\overline z\frac{\partial
(4\ln\kappa)}{\partial\overline z}\Bigr)\Bigr)d\overline z\wedge dw.
\end{aligned}
\end{equation*}
Consequently, since $\textup{Tr}\hat \Theta_{0}|_{\mathcal U_a}=0$, by Cauchy's inequality we obtain
\begin{equation*}\label{171002}
\begin{aligned}
\frac{1}{4\pi^2}\bigl(\mathring\Theta_{0,11}-\frac 1 2\textup{Tr}\hat\Theta_0\bigr)^2
\leq & \frac{(h_{1\bar 1}-h_{2\bar 2})^2}{(\lambda-\lambda^{-1})^2}\bigl(\kappa^4
-r^2\kappa^{-4}\bigr)\frac{\partial^2(4\ln\kappa)}{\partial z\partial \overline z}\frac{\omega_\epsilon^2}{2}\\
+&\frac{4|h_{1\bar 2}|^2}{(\lambda-\lambda^{-1})^2}\Bigl(\kappa^{-4}\Bigl|1-z\frac{\partial(4\ln\kappa)}{\partial z}\Bigr|^2+\kappa^4\Bigl|\frac{\partial(4\ln\kappa)}{\partial z}\Bigr|^2\Bigr)\frac{\omega_\epsilon^2}{2}.
\end{aligned}
\end{equation*}
Then by (\ref{171040}) and (\ref{171017}), it can be rewrite as
\begin{equation}\label{171002}
\frac{1}{4\pi^2}\bigl(\mathring\Theta_{0,11}-\frac 1 2\textup{Tr}\hat\Theta_0\bigr)^2
\leq \frac{1}{8\pi^2}\textup{Tr}(\hat\Theta_0\wedge\hat\Theta_0)-Q\frac{\omega_\epsilon^2}{2},
\end{equation}
where $Q$ is a function on $\mathcal U_a$ defined as
\begin{equation*}
\begin{aligned}
Q=&\frac{4|h_{1\bar 2}|^2}{(\lambda-\lambda^{-1})^2}(\kappa^4-r^2\kappa^{-4})\frac{\partial^2(4\ln\kappa)}{\partial z\partial\overline z}\\
+&\frac{(h_{1\bar 1}-h_{2\bar 2})^2}{(\lambda- \lambda^{-1})^2}\Bigl(\kappa^{-4}\Bigl|1-z\frac{\partial(4\ln\kappa)}{\partial z}\Bigr|^2+\kappa^4\Bigl|\frac{\partial(4\ln\kappa)}{\partial z}\Bigr|^2\Bigr),
\end{aligned}
\end{equation*}
where the second term is nonnegative. Hence we only need to consider the first term.

On $\mathcal U_a-\mathcal U_a(r_0)$, $4\ln\kappa=\phi_2-\phi_1=\ln(\frac r\phi)$. Here $\phi$ is defined by (\ref{171018}). So the first term of $Q$
can be written as
\begin{equation*}
-\frac{4|h_{1\bar 2}|^2r}{(\lambda-\lambda^{-1})^2}\bigl(\frac 1\phi-\phi\bigr)\frac{\partial^2\ln\phi}{\partial z\partial\overline z}.
\end{equation*}
It is bounded by $-C_1\epsilon^{2l-2}$ for any integer $l\geq 2$ and sufficiently small $\epsilon>0$. This is because  $\phi=1+O(u_\epsilon-\frac 1 2\ln r)$ and according to estimate (1) in Theorem \ref{solution} there exists a constant $C_2$ depending on $l$ and $r_0$ such that
\begin{equation*}
\parallel\!u_\epsilon-1/ 2 \ln r\!\parallel_{C^k([r_0,2r_0])}\leq C_2\epsilon^{l-k}.
\end{equation*}
Hence on $\mathcal U_a-\mathcal U_a(r_0)$,
\begin{equation*}\label{171003}
Q\geq -C_1\epsilon^{2l-2}.
\end{equation*}
Combining (\ref{171041}) and (\ref{171002}) and using the above inequality, we obtain
\begin{equation}\label{171007}
\frac{1}{4\pi^2}\int_X\bigl(\mathring\Theta_{0,11}-\frac 1 2\textup{Tr}\hat\Theta_0\bigr)^2\leq \int_Xc_2(V)-\sum_{a=1}^{n}\int_{\mathcal U_a(r_0)}Q+C\epsilon^{2l-2},
\end{equation}
where the constant $C=nC_1$ only depends on $l$ and $r_0$ as $nr_0^2<1$.

Finally we will estimate $\int_{\mathcal U_a(r_0)} Q$. We recall on $\mathcal U_a(r_0)$:
\begin{equation*}
\kappa^4=e^{2u_\epsilon};
\quad \frac{\partial(4\ln\kappa)}{\partial z}=u_\epsilon'\cdot\frac{\overline z}{r};\quad \frac{\partial^2(4\ln\kappa)}{\partial z\partial\overline z}=\frac 1 2(u''_\epsilon(r)+\frac 1 r u'_\epsilon(r)).
\end{equation*}
Hence, by (\ref{417}) we have
\begin{equation*}
\begin{aligned}
Q=&\frac{4|h_{1\bar 2}|^2}{(\lambda-\lambda^{-1})^2}\frac{2\pi^2}{\epsilon^2}\bigl(e^{2u_\epsilon}-r^2e^{-2u_\epsilon}\bigr)^2+\frac{(h_{1\bar 1}-h_{2\bar 2})^2}{(\lambda-\lambda^{-1})^2}\bigl(e^{-2u_\epsilon}(1-ru_\epsilon')^2+e^{2u_\epsilon}(u'_\epsilon)^2\bigr).
\end{aligned}
\end{equation*}
For convenience, define
\begin{align*}
&f_1= \frac{2\pi^2}{\epsilon^2}\bigl(e^{2u_\epsilon}-r^2e^{-2u_\epsilon}\bigr)^2,\\
&f_2=e^{2u_\epsilon}(u'_{\epsilon})^2+ e^{-2u_\epsilon}(1-ru_\epsilon')^2.
\end{align*}
They only depend on $r$. We then define on $\mathcal U_a(r_0)$ a function
\begin{equation*}
\tilde f(r)=\inf\{f_1(r), f_2(r)\}.
\end{equation*}
In view of (\ref{171040}), we have
\begin{equation*}
\int_{\mathcal U_a(r_0)}Q\geq \int_{\mathcal U_a(r_0)}\tilde f(r)\frac{\omega_\epsilon^2}{2}=2\pi \int_0^{r_0}\tilde f(r) rdr.
\end{equation*}

As in Section 3, we define $v_\epsilon(r)=u_\epsilon(r)-\frac 1 2\ln r$ on $]0, 2r_0]$.
Then by Lemma \ref{171250}, $v_\epsilon(r)>0$ and $v'_\epsilon(r)<0$. Hence
\begin{equation*}
\Bigl(\frac{f_1}{r^2}\Bigr)'=\frac{8\pi^2}{\epsilon^2}(e^{2v_\epsilon}-e^{-2v_\epsilon})(e^{2v_\epsilon}+e^{-2v_\epsilon})v_\epsilon'<0.
\end{equation*}
So $\frac{f_1}{r^2}$ is a decreasing function. On the other hand, since $0\leq u'_\epsilon(r)\leq \frac 1 {2r}$ and $u_\epsilon(2r_0)=\frac 1 2 \ln (2r_0)$, we have
\begin{equation*}
(1-ru'_\epsilon(r))^2>\frac 1 4\quad \textup{ and}\quad e^{-2u_\epsilon(r)}>e^{-2u_\epsilon(2r_0)}=\frac 1 {2r_0}.
\end{equation*}
 Hence,
\begin{equation*}
f_2\geq  e^{-2u_\epsilon}(1-ru_\epsilon')^2\geq \frac 1 {8r_0}.
\end{equation*}
So
\begin{equation*}
\frac{f_2}{r^2}\geq \frac{1}{8r_0^3}.
\end{equation*}

Since $\frac{1}{8r_0^3}$ is the constant function  and $\frac{f_1}{r^2}$ is the  decreasing function,
there exists a $r_1$ in $[0,r_0]$ such that
\begin{equation*}
\begin{aligned}
&\inf\Bigl\{\frac{f_1}{r^2},\frac{f_2}{r^2}\Bigr\}\geq
\frac{1}{8r_0^3}\quad \textup{if}\ r\in[0,r_1]\quad \textup{and}\\
&\inf\Bigl\{\frac{f_1}{r^2},\frac{f_2}{r^2}\Bigr\}\geq \frac{f_1}{r^2} \quad
 \textup{if}\ r\in [r_1,r_0].
\end{aligned}
\end{equation*}
So
\begin{equation*}
\begin{aligned}
\tilde f\geq \frac{r^2}{8r_0^3}\quad \textup{if}\ r\in[0,r_1]\quad \textup{and} \quad \tilde f\geq f_1\quad \textup{if}\ r\in[r_1, r_0].
\end{aligned}
\end{equation*}
Consequently,
\begin{equation*}
\begin{aligned}
\int_{\mathcal U_a(r_0)}Q\geq & \frac{\pi}{4r_0^3}\int_0^{r_1}r^3dr+\frac{4\pi^3}{\epsilon^2}\int_{r_1}^{r_0}(e^{2u_\epsilon}-r^2e^{-2u_\epsilon})^2rdr\\
=&\frac{\pi r_1^4}{16 r_0^3}+\frac{4\pi^3}{\epsilon^2}\int_{r_1}^{r_0}(e^{2u_\epsilon}-r^2e^{-2u_\epsilon})^2rdr.
\end{aligned}
\end{equation*}
In the following we assume that $r_1<r_0$. Otherwise we have $\int_{\mathcal U_a(r_0)}Q\geq \frac{\pi r_0}{16}$ and we are done.
Using the Schwarz inequality to the second term yields
\begin{equation}\label{171005}
\begin{aligned}
\int_{\mathcal U_a(r_0)}Q
\geq& \frac{\pi r_1^4}{16 r_0^3}+\frac{4\pi^3}{\epsilon^2}\frac{(\int_{r_1}^{r_0}(e^{2u_\epsilon}-r^2e^{-2u_\epsilon})rdr)^2}{\int_{r_1}^{r_0}rdr}\\
\geq & \frac{\pi r_1^4}{16 r_0^3}+\frac{8\pi^3}{r_0^2\epsilon^2}\bigl(\int_{r_1}^{r_0}(e^{2u_\epsilon}-r^2e^{-2u_\epsilon})rdr\bigr)^2.
\end{aligned}
\end{equation}
By equation (\ref{417}), we have
\begin{equation*}
\frac{\pi^2}{\epsilon^2}(e^{2u_\epsilon}-r^2e^{-2u_\epsilon})=\frac 1 4 (u''_\epsilon+\frac 1 r u'_\epsilon)=\frac 1 {4r}(ru'_\epsilon)'.
\end{equation*}
Hence
\begin{equation*}
\frac{\pi^2}{\epsilon^2}\int_{r_1}^{r_0}(e^{2u_\epsilon}-r^2e^{-2u_\epsilon})rdr=\frac 14 \int_{r_1}^{r_0}(ru'_\epsilon)'dr=\frac 1 4 (r_0u_\epsilon'(r_0)-r_1u_\epsilon'(r_1)).
\end{equation*}
Putting it into (\ref{171005}), we obtain
\begin{equation*}
\int_{\mathcal U_a(r_0)}Q\geq \frac{\pi r_1^4}{16r_0^3}+\frac{\epsilon^2}{2\pi r_0^2}(r_0u'_\epsilon(r_0)-r_1u_\epsilon'(r_1))^2.
\end{equation*}

We recall inequality (\ref{1001}): When $r\in [0,r_0]$, $0\leq u'_\epsilon(r)<\frac{r}{4\epsilon^2}$.
Hence if $r_1<\epsilon(r_0u'_\epsilon(r_0))^{\frac 12}$, then
\begin{equation*}
0\leq r_1u'_\epsilon(r_1)<\frac{r_1^2}{4\epsilon^2}<\frac{r_0u'_\epsilon(r_0)}{4}.
\end{equation*}
Thus
\begin{equation*}
\int_{\mathcal U(r_0)}Q\geq\frac{\epsilon^2}{2\pi r_0^2}(r_0u'_\epsilon(r_0)-r_1u_\epsilon'(r_1))^2\geq  \frac{9\epsilon^2}{32\pi r_0^2}(r_0u_\epsilon'(r_0))^2\geq C_1(r_0)\epsilon^2,
\end{equation*}
since according to  estimate (1) of Theorem \ref{solution}, $|u'_\epsilon(r_0)-\frac 1 {2r_0}|\leq C_2(r_0)\epsilon$ for any sufficiently small $\epsilon>0$. On the other hand, if $r_1\geq \epsilon(r_0u_\epsilon'(r_0))^{\frac  1 2}$, then
\begin{equation*}
\int_{\mathcal U_a(r_0)}Q\geq\frac{\pi r_1^4}{16 r_0^3}\geq  \frac {\pi(r_0u_\epsilon'(r_0))^2}{16r_0^3}\epsilon^4\geq C_3(r_0)\epsilon^4.
\end{equation*}
In summary, we have proved that for sufficient small $\epsilon>0$, there exists a positive constant $C_4(r_0)$ depending on $r_0$ such that
\begin{equation*}
\int_{\mathcal U_a(r_0)}Q\geq C_4(r_0)\epsilon^4.
\end{equation*}
Combined with (\ref{171007}), since the constant $nC_4(r_0)$ can be written as a constant $C(r_0)$,  we finish the proof of the lemma.
\end{proof}

\begin{lemm}\label{170981}
For any integer $l>1$, there exists a constant $C$ depending on $r_0$ and $l$ such that
for sufficiently small $\epsilon>0$,
\begin{equation*}
\begin{aligned}
&\frac{1}{4\pi^2}\int_X\bigl(\mathring\Theta_{0,11}-\frac 12 \textup{Tr}\hat\Theta_0\bigr)^2
\geq -\epsilon^{-1}(1+|a_1|+|a_2|+|a_3|)\int_X(|s_{12}|^2_\epsilon+|s_{21}|^2_\epsilon)\\
&\qquad \qquad \qquad \qquad -2a_1a_2+2|a_3|^2-2(a_1\epsilon+a_2\epsilon^{-1})^2-C\epsilon^{2l-2}.
\end{aligned}
\end{equation*}
\end{lemm}
\begin{proof}
By (\ref{171261}), $\mathring\Theta_{0,11}-\frac 1 2 \textup{Tr}\hat\Theta_0$ is decomposed as the sum of the following three expressions:
\begin{equation*}
\begin{aligned}
&I_1=-\partial\overline s_{11}+\overline\partial s_{11}+2\pi i(\partial\overline\theta+\overline\partial\theta);\\
&I_2=s_{12}\wedge \overline s_{12}-s_{21}\wedge
\overline s_{21}-2\pi i (\partial\overline\theta+\overline\partial\theta);\quad\and\\
&I_3=-\frac 1 2\textup{Tr}\hat\Theta_0.
\end{aligned}
\end{equation*}
Then we write
\begin{equation}\label{171205}
\frac{1}{4\pi^2}\int_X\bigl(\mathring\Theta_{0,11}-\frac 12\textup{Tr}\hat\Theta_0\bigr)^2=\sum_{i,j=1}^3 I_{ij}\quad \textup{for}\ \ I_{ij}=\frac 1{4\pi^2}\int_X I_i\wedge I_j.
\end{equation}

By (\ref{171130}) we have
\begin{equation}\label{171150}
I_1=\pi(a_1dz\wedge d\overline z+a_2dw\wedge d\overline w+a_3dz\wedge d\overline w+\overline a_3dw\wedge d\overline z).
\end{equation}
Hence
\begin{equation}\label{171210}
I_{11}=-2a_1a_2+2|a_3|^2.
\end{equation}
Since $c_1(V)=0$, $\frac 1 2 \textup{Tr}\hat \Theta_0$ is a $\partial\overline\partial$-exact form. By Stokes' theorem we have
\begin{equation}\label{171203}
2I_{13}=0,\qquad I_{33}=0.
\end{equation}

By Stokes' theorem, we also have
\begin{equation}\label{171151}
2I_{12}=\frac 1 {2\pi^2} \int_X I_1\wedge (s_{12}\wedge \overline s_{12}-s_{21}\wedge\overline s_{21}).
\end{equation}
Let
\begin{equation*}
s_{12}=b_1dz+b_2dw\quad\and\quad s_{21}=b_3dz+b_4dw,
\end{equation*}
where $b_i$ for $i=1,\,2,\,3,\,4$ are complex functions locally defined on $X$.
Then
\begin{equation}\label{171009}
\begin{aligned}
&s_{12}\wedge \overline s_{12}=|b_1|^2dz\wedge d\overline z+|b_2|^2 dw\wedge d \overline w+b_1\overline b_2dz\wedge d\overline w+\overline b_1b_2dw\wedge d\overline z,\\
&s_{21}\wedge\overline s_{21}=|b_3|^2dz\wedge d\overline z+|b_4|^2dw\wedge d\overline w+b_3\overline b_4dz\wedge d\overline w+\overline b_3b_4dw\wedge d\overline z.
\end{aligned}
\end{equation}
Hence
\begin{equation}\label{171251}
|s_{12}|_\epsilon^2=|b_1|^2\epsilon+|b_2|^2\epsilon^{-1}\and |s_{21}|^2_\epsilon=|b_3|^2\epsilon+|b_4|^2\epsilon^{-1}.
\end{equation}
Putting (\ref{171150}) and (\ref{171009}) into (\ref{171151}) yields
\begin{equation*}
\begin{aligned}
2I_{12}=&-\frac{2}{\pi}a_1\int_X(|b_2|^2-|b_4|^2)-\frac {2}{\pi}a_2\int_X(|b_1|^2-|b_3|^2)\\
&+\frac{2}{\pi}a_3\int_X(\overline b_1b_2-\overline b_3b_4)+\frac 2 \pi\overline a_3\int_X(b_1\overline b_2-b_3\overline b_4).
\end{aligned}
\end{equation*}
Using the triangle inequality and Cauchy's inequality to the third and fourth terms, we have
\begin{equation*}
\begin{aligned}
2I_{12}\geq& -\frac 2 \pi |a_1|\epsilon\int_X (|b_2|+|b_4|)\epsilon^{-1}-\frac 2 \pi |a_2|\epsilon^{-1}\int_X(|b_1|^2+|b_3|^2)\epsilon\\
&-\frac 2 \pi |a_3|\int_X(|b_1|\epsilon+|b_2|\epsilon^{-1}+|b_3|^2\epsilon+|b_4|^2\epsilon^{-1}).
\end{aligned}
\end{equation*}
Hence, by (\ref{171251}) we can easily get
\begin{equation}\label{171022}
\begin{aligned}
2I_{12}\geq -\epsilon^{-1}(|a_1|+|a_2|+|a_3|)\int_X(|s_{12}|^2_\epsilon+|s_{21}|^2_\epsilon).
\end{aligned}
\end{equation}

Now we handle $I_{22}$. If we let $2\pi\overline\theta=t_1d\overline z+t_2d\overline w$, then
\begin{equation*}
\begin{aligned}
&2\pi i(\partial\overline\theta+\overline\partial\theta)=-2\textup{Im}\frac {\partial t_1}{\partial z}dz\wedge d\overline z-2\textup{Im}
\frac{\partial t_2}{\partial w}dw\wedge d\overline w\\
&\qquad +i\Bigl(\frac{\partial t_2}{\partial z}-\overline{\frac{\partial t_1}{\partial w}}\Bigr)dz\wedge d\overline w+i\Bigl(\frac{\partial t_1}{\partial w}-\overline{\frac{\partial t_2}{\partial z}}\Bigr)dw\wedge d\overline z.
\end{aligned}
\end{equation*}
Combined with (\ref{171009}), we have
\begin{equation*}
\begin{aligned}
I_{22}=&-\frac{2}{\pi^2}\int_X\bigl(|b_1|^2-|b_3|^2+2\textup{Im}\frac{\partial t_1}{\partial z}\bigr)\bigl(|b_3|^2-|b_4|^2+2\textup{Im}\frac{\partial t_2}{\partial w}\bigr)\\
&+\frac{2}{\pi^2}\int_X\Bigl|b_1\overline b_2-b_3\overline b_4-i\Bigl(\frac{\partial t_2}{\partial z}-\overline{\frac{\partial t_1}{\partial w}}\Bigr)\Bigr|^2,
\end{aligned}
\end{equation*}
where the second term is nonnegative, and the first term can be written as the sum of the following three terms:
\begin{equation*}
\begin{aligned}
&I_{22}^1=-\frac{1}{\pi^2}\int_X\Bigl(\epsilon\bigl(|b_1|^2-|b_3|^2+2\textup{Im}\frac{\partial t_1}{\partial z}\bigr)+\epsilon^{-1}
 \bigl(|b_2|^2-|b_4|^2+2\textup{Im}\frac{\partial t_2}{\partial w}\bigr)\Bigr)^2;\\
&I_{22}^2=\frac{1}{\pi^2}\int_X\epsilon^2\bigl(|b_1|^2-|b_3|^2+2\textup{Im}\frac{\partial t_1}{\partial z}\bigr)^2;\\
&I_{22}^3=\frac{1}{\pi^2}\int_X\epsilon^{-2}\bigl(|b_2|^2-|b_4|^2+2\textup{Im}
\frac{\partial t_2}{\partial w}\bigr)^2.
\end{aligned}
\end{equation*}
Here $I_{22}^2$ and $I_{22}^3$ are also nonnegative. Hence we only need to deal with $I_{22}^1$.
We observe that its integrand is exactly
\begin{equation*}
-\frac 1 {\pi^2}\bigl(|s_{12}|^2_\epsilon-|s_{21}|^2_\epsilon+\pi \Lambda_{\omega_\epsilon}(\partial\overline\theta+\overline\partial\theta)\bigr)^2,
\end{equation*}
which by (\ref{171252}) is
\begin{equation*}
\begin{aligned}
-\frac 1{\pi^2}\Bigl(-\pi (a_1\epsilon+a_2\epsilon^{-1})+\frac{h_{1\bar 1}-h_{2\bar 2}}{\lambda-\lambda^{-1}}\psi\Bigr)^2,
\end{aligned}
\end{equation*}
which,  by Cauchy's inequality, (\ref{171040}) and (\ref{445}),  is bigger than
\begin{equation*}
-2(a_1\epsilon+a_2\epsilon^{-1})^2-\frac{2}{\pi^2}\parallel\!\psi\!\parallel^2_{C^0}
\geq -2(a_1\epsilon+a_2\epsilon^{-1})^2-C(r_0,l)\epsilon^{2l-2}.
\end{equation*}
Thus we obtain
\begin{equation}\label{171023}
I_{22}\geq I_{22}^1\geq -2(a_1\epsilon+a_2\epsilon^{-1})^2-C(r_0,l)\epsilon^{2l-2}.
\end{equation}

Finally, we deal with $I_{23}$. Since $\textup{Tr}\hat\Theta_0$ is a $\partial\overline\partial$-exact form, by Stokes' theorem,
\begin{equation*}\label{171201}
2I_{23}=-\frac 1 {4\pi^2}\int_X(s_{12}\wedge \overline s_{12}-s_{21}\wedge\overline s_{21})\wedge \textup{Tr}\hat\Theta_0
\end{equation*}
By (\ref{171016}) we have
\begin{equation*}
\textup{Tr}\hat\Theta_0=-\pi i\frac{\partial (w_1^\ast+w_2^\ast)}{\partial z}dz\wedge d\overline w-\pi i
\overline{\frac{\partial (w_1^\ast+w_2^\ast)}{\partial z}}d\overline z\wedge d w.
\end{equation*}
Combined with (\ref{171009}), direct calculation yields
\begin{equation*}
-\frac{1}{4\pi^2}(s_{12}\wedge \overline s_{12}-s_{21}\wedge \overline s_{21})\wedge\textup{Tr}\hat\Theta_0=\frac{2}{\pi} \textup{Im}\Bigl((b_1\overline b_2+b_3\overline b_4)\overline {\frac{\partial (w_1^\ast+w_2^\ast)}{\partial z}}
\Bigr)\frac{\omega^2_\epsilon}{2},
\end{equation*}
which, by the triangle inequality and Cauchy's inequality, is bigger than
\begin{equation*}
-\frac 1 \pi\Bigl|\frac{\partial(w_1^\ast+w_2^\ast)}{\partial z}\Bigr|(|s_{12}|_\epsilon^2+|s_{21}|_\epsilon^2)\frac{\omega_\epsilon^2}{2}.
\end{equation*}
In Section 2, we picked on $U_a$ $w_2^\ast=-w_1^\ast=-\sqrt{z}$, and hence, $|\frac{\partial(w_1^\ast+w_2^\ast)}{\partial z}|=0$.
So we can take a localization of $\varphi$ in Section 2 such that
\begin{equation*}
\Bigl|\frac{\partial(w_1^\ast+w_2^\ast)}{\partial z}\Bigr|_{C^0}
\end{equation*}
is bounded by a constant. We fix such a localization and denote such a constant by $C(\varphi)$. Thus,
for sufficiently small $\epsilon>0$
\begin{equation}\label{171204}
2I_{23}\geq -C(\varphi)\int_X(|s_{12}|^2_\epsilon +|s_{21}|^2_\epsilon)\geq -\epsilon^{-1} \int_X(|s_{12}|^2_\epsilon +|s_{21}|^2_\epsilon).
\end{equation}

Now combining (\ref{171205}) with (\ref{171210}), (\ref{171203}), (\ref{171022}),  (\ref{171023}) and (\ref{171204}), at last we finish the proof.
\end{proof}

Combining Lemmas \ref{170980} and \ref{170981}, we arrive at
\begin{equation}\label{171081}
\begin{aligned}
&\epsilon^{-1}(|a_1|+|a_2|+|a_3|+1)\int_X(|s_{12}|^2_\epsilon+|s_{21}|^2_\epsilon)
\geq -2a_1a_2+2|a_3|^2\\
&-\int_Xc_2(V)+C(r_0)\epsilon^4-2(a_1\epsilon+a_2\epsilon^{-1})^2-C(r_0,l)\epsilon^{2l-2}.
\end{aligned}
\end{equation}

\begin{lemm}\label{171096}
For any integer $l\geq4$, if $|a_1\epsilon+a_2\epsilon^{-1}|\leq\epsilon^3$, then there exists a constant $C$ depending on $l$, $r_0$ and $\deg q$
such that for any sufficiently small $\epsilon>0$,
\begin{equation*}
\lambda_0\leq 1+C\epsilon^{l-6}.
\end{equation*}
\end{lemm}
\begin{proof}
If $l\geq 4$ and $|a_1\epsilon+a_2\epsilon^{-1}|\leq \epsilon^3$, then there exists a positive constant $C_1(r_0,l)$ such that for sufficient small $\epsilon>0$,
\begin{equation*}
C(r_0)\epsilon^4-2(a_1\epsilon+a_2\epsilon^{-1})^2-C(r_0,l)\epsilon^{2l-2}\geq C_1(r_0,l)\epsilon^4.
\end{equation*}
Hence by (\ref{171081}), we have
\begin{equation}\label{171091}
\int_X(|s_{12}|^2_\epsilon+|s_{21}|^2_\epsilon)\geq \frac{\epsilon (-2a_1a_2+2|a_3|^2-\int_Xc_2(V))+C_1(r_0,l)\epsilon^5}{|a_1|+|a_2|+|a_3|+1}.
\end{equation}

If $|a_1\epsilon+a_2\epsilon^{-1}|\leq \epsilon^3$, then $a_1=a_2=0$ or $a_1a_2<0$, since
 $a_1$ and $a_2$ are integers and $\epsilon>0$ is sufficiently small. Hence
\begin{equation}\label{171092}
-2a_1a_2+2|a_3|^2\geq 2(|a_1|+|a_2|+|a_3|).
\end{equation}
 On the other hand, since $V=L\oplus L^{-1}$, $c_2(V)=-c^2_1(L)$. Then by (\ref{170974}) we have
\begin{equation}\label{171090}
\int_Xc_2(V)=-2a_1a_2+2|a_3|^2-2|a_4|^2\leq -2a_1a_2+2|a_3|^2.
\end{equation}

Our goal is to prove
\begin{equation}\label{20180601}
\int_X(|s_{12}|^2_\epsilon+|s_{21}|^2_\epsilon)\geq C\epsilon^{-5}
\end{equation}
by considering the following two cases:
\begin{equation}\label{20180602}
-2a_1a_2+2|a_3|^2\geq 2 \int_Xc_2(V),
\end{equation}
or
\begin{equation}\label{20180603}
-2a_1a_2+2|a_3|^2<2\int_X c_2(V).
\end{equation}
If assumption (\ref{20180602}) holds, then by (\ref{171092}),
\begin{equation*}\label{171093}
-2a_1a_2+2|a_3|^2-\int_Xc_2(V)\geq -a_1a_2+|a_3|^2\geq |a_1|+|a_2|+|a_3|.
\end{equation*}
Combined with (\ref{171091}), since $|a_1|+|a_2|+|a_3|\geq 1$, we get
\begin{equation*}\label{171253}
\int_X(|s_{12}|^2_\epsilon+|s_{21}|^2_\epsilon)\geq \frac{ \epsilon (|a_1|+|a_2|+|a_3|)}{|a_1|+a_2|+
|a_3|+1}\geq \frac \epsilon 2.
\end{equation*}
If assumption (\ref{20180603}) holds,
 then combined it with (\ref{171092}), we have   $|a_1|+|a_2|+|a_3|<\int_Xc_2(V)$. Hence from (\ref{171091}), we get
\begin{equation*}\label{171254}
\int_X| (s_{12}|^2_\epsilon+|s_{21}|^2_\epsilon)\geq \frac{C_1(r_0,l)\epsilon^5}{\int_Xc_2(V)+1}\geq C\epsilon^{-5},
\end{equation*}
where by (\ref{10}) the positive constant $C$ depends on $l$, $r_0$ and $\deg q$. Hence inequality (\ref{20180601}) holds.

Now combining  inequality (\ref{20180601}) with (\ref{170978}),  we can finish the  proof of this lemma by the similar arguments in the proof of Lemma \ref{171095}. We omit it here.
\end{proof}

We are ready to give the proof of the lemma. If $\inf_{x\in X}\tau(x)>2$, then combining Lemmas \ref{171095} and \ref{171096} yields
\begin{equation*}
\inf_{x\in X}\tau(x)=\lambda_0+\frac 1 {\lambda_0}\leq \lambda_0+1\leq 2+C\epsilon^{l-6}.
\end{equation*}
\end{proof}

\section{The higher order estimates}
In this section, we prove Theorem \ref{1901001}.
\begin{proof}
We start from the formula
\begin{equation}\label{804}
\tilde H_{1,\epsilon}=B^t(\check{H}_\epsilon)^t\overline{B}
\end{equation}
which can be proved  by (\ref{802}) as follows:
\begin{eqnarray*}
\begin{aligned}
(\tilde H_{1,\epsilon})_{i\bar j}=&H_{1,\epsilon}(\tilde\mu_i,\tilde\mu_j)
=H_{0,\epsilon}(H_{\epsilon}(\tilde\mu_i),\tilde\mu_j)
=H_{0,\epsilon}(H_\epsilon(b_{ki}\check \mu_k),b_{lj}\check\mu_l)\\
=&b_{ki}\overline {b}_{lj}H_{0,\epsilon}((\check H_\epsilon)_{mk}\check\mu_m,\check\mu_l)
=b_{ki}\overline b_{lj}(\check H_\epsilon)_{lk}.
\end{aligned}
\end{eqnarray*}
We will use the notations $\tilde\Theta_1$ and $\tilde\Theta_0$ etc. as in the above section.
Since $H_{1,\epsilon}$ is the HYM metric, we use the formula (\ref{b0}) to $\tilde H_{1,\epsilon}$ and  (\ref{804}) to get
\begin{equation*}\label{13}
0=\Lambda_{\omega_\epsilon}\tilde\Theta_1=
\Lambda_{\omega_\epsilon}\bigl(\overline\partial\bigl(\partial(B^t(\check
H_\epsilon)^t\overline B\bigr)\bigl(B^t(\check
H_\epsilon)^t\overline B)\bigr)^{-1}\bigr)^t,
\end{equation*}
which is equivalent to
\begin{equation}\label{15}
\check H_\epsilon B
\Lambda_{\omega_\epsilon}\bigl(\overline\partial\bigl(\partial(B^t(\check
H_\epsilon)^t\overline B\bigr)\bigl(B^t(\check
H_\epsilon)^t\overline B)\bigr)^{-1}\bigr)^t B^{-1}=0.
\end{equation}
On the other hand, (\ref{802}) implies $\tilde
H_{0,\epsilon}=B^t\overline B$. Hence formula (\ref{b0}) also gives
\begin{equation*}\label{14}
\Lambda_{\omega_\epsilon}\tilde\Theta_0=\Lambda_{\omega_\epsilon}\overline\partial(\partial(B^t\overline
B)(B^t\overline B)^{-1})^t
\end{equation*}
or
\begin{equation}\label{16}
B \Lambda_{\omega_\epsilon}\overline\partial(\partial(B^t\overline
B)(B^t\overline B)^{-1})^t  B^{-1}\check H_\epsilon=B
\Lambda_{\omega_\epsilon}\tilde\Theta_0 B^{-1}\check H_\epsilon.
\end{equation}
Combining (\ref{a8}) with (\ref{802}) and using (\ref{900}), we have
\begin{equation*}
\Lambda_{\omega_\epsilon} B\tilde\Theta_0B^{-1}
=\Lambda_{\omega_\epsilon}\check\Theta_0=\Lambda_{\omega_\epsilon}
\hat\Theta_0.
\end{equation*}
Now subtracting (\ref{16}) from (\ref{15}), expanding the left hand side of the
derived equation, and properly adjusting some terms,  we arrive at the
system:
\begin{eqnarray}\label{1209065}
\begin{aligned}
0=&\ \ i\Lambda_{\omega_\epsilon}\overline\partial\partial \mathcal H_\epsilon
-i\Lambda_{\omega_\epsilon}\overline\partial \mathcal H_\epsilon\check H_\epsilon^{-1}\wedge \partial \mathcal H_\epsilon\\
&-i\Lambda_{\omega_\epsilon} \check H_\epsilon\overline\partial \log B\check
H_\epsilon^{-1}
\wedge \partial \mathcal H_\epsilon-i\Lambda_{\omega_\epsilon}\overline\partial \mathcal H_\epsilon \check
H_\epsilon^{-1}
{(\partial \log \overline B)^t} \check H_\epsilon\\
&-i\Lambda_{\omega_\epsilon} \partial \mathcal H_\epsilon\wedge \overline\partial \log B
-i\Lambda_{\omega_\epsilon}  {(\partial \log \overline B)^t}\wedge\overline\partial\mathcal H_\epsilon\\
&-i\Lambda_{\omega_\epsilon}\mathcal H_\epsilon\partial(\overline\partial \log B)
+i\Lambda_{\omega_\epsilon}\partial(\overline\partial \log B) \mathcal H_\epsilon\\
&-i\Lambda_{\omega_\epsilon} \check H_\epsilon \overline\partial \log B\mathfrak
H_\epsilon
\wedge  {(\partial \log\overline  B)^t}\check H_\epsilon+i\Lambda_{\omega_\epsilon} \check H_\epsilon \mathfrak H_\epsilon \overline\partial\log B
\wedge {(\partial \log\overline B)^t} \check H_\epsilon\\
&-i\Lambda_{\omega_\epsilon}  {(\partial \log\overline B)^t} \mathcal H_\epsilon\wedge\overline\partial \log B+i\Lambda_{\omega_\epsilon}  {(\partial \log\overline  B)^t}\wedge \overline\partial \log B \mathcal H_\epsilon\\
&+i\Lambda_{\omega_\epsilon}\hat\Theta_0\check H_\epsilon,
\end{aligned}
\end{eqnarray}
 where for brevity, we have introduced the notations:
\begin{align*}
&\mathcal H_\epsilon=\check H_\epsilon-I, \ \ \ \ \ \mathfrak
H_\epsilon=\check H^{-1}_\epsilon-I,\\
&\partial\log\overline B=\partial \overline B\,\overline B^{-1},\ \ \overline\partial\log
B=\overline\partial B B^{-1}.
\end{align*}
Hence $\partial\log \overline B=\overline{\overline \partial \log B}$.

We introduce
\begin{equation*}
x_{i,\epsilon}=\epsilon^{- 1/ 2} x_i,\ \
y_{i,\epsilon}=\epsilon^{ 1/ 2}y_i\ \ {\textup{for}}\
\ i=1,2
\end{equation*}
and
\begin{equation}\label{b7}
z_\epsilon=\epsilon^{-1/ 2}z,\ \
w_{\epsilon}=\epsilon^{ 1/ 2}w.
\end{equation}
Then the metric  (\ref{403}) can be rewritten as the Euclidean metric
\begin{equation}\label{b4}
\omega_\epsilon=dy_{1,\epsilon}\wedge
dy_{2,\epsilon}+dx_{1,\epsilon}\wedge dx_{2,\epsilon}.
\end{equation}
We will use
$\bigtriangledown^k_\epsilon$, $\bigtriangleup_\epsilon$ and
$C^{k}_\epsilon$, respectively,  to denote the $k$-th
 covariant derivatives, the Laplace operator and the $C^k$-norm with the new
 coordinates.
Hence, for any $f\in C^\infty(\mathcal U)$,
\begin{equation*}
\bigtriangleup_\epsilon f=\frac{\partial^2 f}{\partial x_{1,\epsilon}^2}+\frac{\partial^2f}{\partial x_{2,\epsilon}^2}+\frac{\partial^2f}{\partial y_{1,\epsilon}^2}+\frac{\partial^2f}{\partial y_{2,\epsilon}^2}
\end{equation*}
is the same as equality (\ref{o1});
\begin{equation*}
\parallel\!\bigtriangledown^j_\epsilon f\!\parallel_{C^0}\leq \epsilon^{-\frac j 2}\parallel\!\bigtriangledown^j f\!\parallel_{C^0};
\end{equation*}
and
\begin{equation}\label{201801}
\parallel\!f\!\parallel_{C^j_\epsilon}=\sum_{i=0}^j\parallel\!\bigtriangledown_\epsilon^i f\!\parallel_{C^0}\leq \epsilon^{-\frac j 2 }\parallel\!f\!\parallel_{C^j}.
\end{equation}
Similarly, we also have
\begin{equation}\label{20180509}
\parallel\!f\!\parallel_{C^j}\leq \epsilon^{-\frac j 2}\parallel\!f\!\parallel_{C^j_\epsilon}.
\end{equation}

In this way, system (\ref{1209065})
 can be rewritten as
\begin{equation}\label{901}
I_2=I_{1^2}+I_1+I_0+I_{-1},
\end{equation}
where
\begin{align*}
I_2&=\frac{\partial^2\mathcal
H_\epsilon}{\partial z_{\epsilon}\partial\overline
z_\epsilon}+\frac{\partial^2\mathcal H_\epsilon}{\partial
w_{\epsilon}\partial\overline w_\epsilon};\\
I_{1^2}&=\frac{\partial \mathcal H_\epsilon}{\partial\overline
z_{\epsilon}}\check H_\epsilon^{-1}\frac{\partial \mathcal
H_\epsilon}{\partial z_{\epsilon}} +\frac{\partial \mathcal
H_\epsilon}{\partial\overline w_{\epsilon}}\check H_\epsilon^{-1}
\frac{\partial \mathcal H_\epsilon}{\partial w_{\epsilon}};\\
I_1&=\check H_\epsilon\frac{\partial \log B}{\partial\overline
z_{\epsilon}}\check H_\epsilon^{-1}\frac{\partial \mathcal
H_\epsilon}{\partial  z_{\epsilon}} +\check H_\epsilon\frac{\partial
\log B}{\partial\overline w_{\epsilon}}\check H_\epsilon^{-1}
\frac{\partial \mathcal H_\epsilon}{\partial  w_{\epsilon}}\\
&+\frac{\partial \mathcal H_\epsilon}{\partial\overline
z_{\epsilon}}\check H_\epsilon^{-1}{\Bigl(\overline{\frac{\partial
\log B}{\partial \overline z_{\epsilon}}}\Bigr)^t}\check H_\epsilon
+\frac{\partial \mathcal H_\epsilon}{\partial\overline
w_{\epsilon}}\check H_\epsilon^{-1}
{\Bigl(\overline{\frac{\partial \log B}{\partial\overline w_{\epsilon}}}\Bigr)^t}\check H_\epsilon\\
&-\frac{\partial \mathcal H_\epsilon}{\partial
z_{\epsilon}}\frac{\partial \log B}{\partial\overline  z_{\epsilon}}
-\frac{\partial\mathcal H_\epsilon}{\partial w_{\epsilon}}
\frac{\partial \log B}{\partial\overline  w_{\epsilon}}\\
&-{\Bigl(\overline{\frac{\partial \log B}{\partial\overline
z_{\epsilon}}}\Bigr)^t}\frac{\partial \mathcal
H_\epsilon}{\partial\overline
z_{\epsilon}}-{\Bigl(\overline{\frac{\partial \log B}{\partial\overline
 w_{\epsilon}}}\Bigr)^t}\frac{\partial \mathcal
H_\epsilon}{\partial\overline w_{\epsilon}};\\
I_0&=-\mathcal H_\epsilon\frac{\partial^2\log B}{\partial
z_{\epsilon}\partial\overline z_\epsilon}
-\mathcal H_\epsilon \frac{\partial^2\log B}{\partial w_{\epsilon}\partial\overline w_\epsilon}\\
&+\frac{\partial^2\log B}{\partial z_{\epsilon}\partial\overline
z_\epsilon} \mathcal H_\epsilon+ \frac{\partial^2\log B}{\partial
w_{\epsilon}\partial\overline w_\epsilon}
\mathcal H_\epsilon\\
&+\check H_\epsilon \frac{\partial \log B}{\partial\overline
z_{\epsilon}}\mathfrak H_\epsilon{\Bigl(\overline{\frac{\partial \log
B}{\partial\overline z_{\epsilon}}}\Bigr)^t}\check H_\epsilon +\check
H_\epsilon\frac{\partial \log B}{\partial\overline
w_{\epsilon}}\mathfrak H_\epsilon
{\Bigl(\overline{\frac{\partial \log B}{\partial\overline w_{\epsilon}}}\Bigr)^t}\check H_\epsilon\\
&-\check H_\epsilon \mathfrak H_\epsilon\frac{\partial \log
B}{\partial\overline z_{\epsilon}}{\Bigl(\overline{\frac{\partial
\log B}{\partial\overline z_{\epsilon}}}\Bigr)^t}\check H_\epsilon
-\check H_\epsilon\mathfrak H_\epsilon\frac{\partial \log B}
{\partial\overline w_{\epsilon}}{\Bigl(\overline{\frac{\partial \log
B}{\partial\overline w_{\epsilon}}}\Bigr)^t}\check H_\epsilon\\
&-{\Bigl(\overline{\frac{\partial \log B}{\partial\overline
z_{\epsilon}}}\Bigr)^t}\mathcal H_\epsilon \frac{\partial \log
B}{\partial\overline z_{\epsilon}} -{\Bigl(\overline{\frac{\partial
\log B}{\partial\overline w_{\epsilon}}}\Bigr)^t}
\mathcal H_\epsilon\frac{\partial \log B}{\partial\overline w_{\epsilon}}\\
&+{\overline{\Bigl(\frac{\partial \log B}{\partial\overline
z_{\epsilon}}}\Bigr)^t} \frac{\partial \log B}{\partial\overline
z_{\epsilon}}\mathcal H_\epsilon+{\Bigl(\overline{\frac{\partial \log
B}{\partial\overline w_{\epsilon}}}\Bigr)^t}
\frac{\partial \log B}{\partial\overline w_{\epsilon}}\mathcal H_\epsilon;\\
I_{-1}&=i\Lambda_{\omega_\epsilon}\hat\Theta_0\check H_\epsilon.
\end{align*}

We observe that: a) All terms in $I_1$ have a factor of the first order
derivatives of $\mathcal H_\epsilon$, while no terms in $I_0$
contain such a factor; b) All terms in $I_0$ have a factor $\mathcal H_\epsilon$ or $\mathfrak
 H_\epsilon$, which, by assumption (\ref{1901003}), satisfies
\begin{equation}\label{01}
\parallel\!\mathcal
H_\epsilon\!\parallel_{C^0}\leq C\epsilon^{l}\qquad\textup{or}\qquad
\parallel\!\mathfrak H_\epsilon\!\parallel_{C^0}\leq
C\epsilon^{l}.
\end{equation}
Hence, for a positive $l$ and sufficiently small $\epsilon>0$,
$\parallel\!\mathcal H_\epsilon\!\parallel_{C^0}$ and
$\parallel\!\mathfrak H_\epsilon\!\parallel_{C^0}$ are indeed very small.
For the term $I_{-1}$, by (\ref{800}) we have
\begin{equation*}\label{a9}
\parallel\!\Lambda_{\omega_\epsilon}
\hat\Theta_0 \!\parallel_{C^j}\leq C\epsilon^{l-j-1}.
\end{equation*}
Then by (\ref{201801}), we get
\begin{eqnarray}\label{20180507}
\parallel\!\Lambda_{\omega_\epsilon} \hat\Theta_0\!\parallel_{C_\epsilon^j}
\leq \epsilon^{-j/2}\parallel\!\Lambda_{\omega_\epsilon}
\hat\Theta_0\!\parallel_{C^j}\leq C \epsilon^{l-\frac{3j}{2}-1}.
\end{eqnarray}
In particular,  we have
\begin{equation}\label{171202}
\parallel\!I_{-1}\!\parallel_{C^0}\leq C\epsilon^{l-1}.
\end{equation}

We will estimate the factors  coming from $\overline\partial \log B$ and
$\partial\overline\partial\log B$ in $I_{1}$ and $I_0$. The most complicated  case is over
$\mathcal U_a$.  (Note that $\mathcal U_0$ has been shrunk in Section
5.) Hence, we will omit the other cases and only do estimates to this
case. By (\ref{802}) and (\ref{801}),
\begin{equation*}
B=e^{\frac 1 2
g_a}\begin{pmatrix}\kappa^{-1}&0\\0&\kappa\end{pmatrix}A.
\end{equation*}
By (\ref{1003}) and (\ref{171018}), $\kappa$ can be written as
\begin{equation*}
\kappa=\left\{\begin{array}{ll}r^{\frac 1 4}\phi^{-\frac 1 4}=r^{\frac 1 4}(1+O(u_\epsilon-\frac 12\ln r))^{-\frac 14}&\ \
\textup{on}\ U_a(2r_0)-U_a(r_0)\\
e^{\frac 12 u_\epsilon}&\ \ \textup{on}\ U_a(r_0)).\end{array}\right.
\end{equation*}
Since $\frac 1 2 \ln r\leq u_\epsilon\leq\frac 1 2 \ln (2r_0)$ by Proposition \ref{prop} and the first inequality in Lemma \ref{171250}, we have
\begin{equation}\label{20180201}
\parallel\!\kappa\!\parallel_{C^0}\leq C,\qquad \parallel\!r\kappa^{-2}\!\parallel_{C^0}\leq  C.
\end{equation}
Since $g_a$ is harmonic and $A$ is holomorphic for the variable $z$, direct calculation gives
\begin{align*}
&\frac{\partial\log B}{\partial\overline z_\epsilon}=\epsilon^{\frac 1 2}\frac 1 2 \frac{\partial
g_a}{\partial\overline
z}I+\epsilon^{\frac 1 2}\frac{\partial\log\kappa}{\partial\overline
z}\begin{pmatrix} -1&0\\
0&1\end{pmatrix},\\
&\frac{\partial \log
B}{\partial\overline w_\epsilon}=\epsilon^{-\frac 1 2}\pi i\begin{pmatrix}0& z\kappa^{-2}\\ \kappa^2&0\end{pmatrix},
 \\
&\frac{\partial^2\log B}{\partial z_\epsilon\partial\overline
z_\epsilon}=\epsilon\frac{\partial^2\log \kappa}{\partial z\partial\overline
z}\begin{pmatrix}-1&0\\0&1\end{pmatrix},\quad \frac{\partial^2\log B}{\partial w_\epsilon\partial\overline
w_\epsilon}=0.
\end{align*}
By Theorem \ref{solution} and (\ref{20180201}), we can easily get
\begin{align}
&\Bigl|\!\Bigl|\frac{\partial \log B}{\partial \overline
z_\epsilon}\Bigr|\!\Bigr|_{C^0}\leq C\epsilon^{-\frac {1} 2},\qquad
\Bigl|\!\Bigl|\frac{\partial\log
B}{\partial\overline w_\epsilon}\Bigl|\!\Bigl|_{C^0}\leq C\epsilon^{-\frac 1 2},\label{20180202}\\
&\Bigl|\!\Bigl|\frac{\partial^2\log B}{\partial z_\epsilon\partial\overline z_\epsilon}\Bigl|\!\Bigl|_{C^0}\leq C\epsilon^{-3},\qquad
\Bigl|\!\Bigl|\frac{\partial^2\log B}{\partial w_\epsilon\partial\overline w_\epsilon}\Bigl|\!\Bigl|_{C^0}=0.
\end{align}
Combined with (\ref{01}), we
obtain
\begin{equation}\label{08}
\parallel\! I_0\!\parallel_{C^0}\leq C\epsilon^{l-3}.
\end{equation}

We need more estimates for preparations. Since
\begin{align*}
&\frac{\partial(z\kappa^{-2})}{\partial z}=\kappa^{-2}-2z\kappa^{-2}\frac{\partial\log\kappa}{\partial z},\ \ \ \frac{\partial\kappa^2}{\partial z}=2\kappa^2\frac{\partial\log\kappa}{\partial z},\\
&\frac{\partial(z\kappa^{-2})}{\partial\overline z}=-2z\kappa^{-2}\frac{\partial\log\kappa}{\partial\overline z},\ \ \ \ \ \frac{\partial \kappa^2}{\partial\overline z}=2\kappa^2\frac{\partial\log\kappa^2}{\partial\overline z},
\end{align*}
and $\kappa$ only depends on the variable $z$,
$\bigtriangledown^j_\epsilon\frac{\partial\log B}{\partial\overline w_\epsilon}$  contains the terms:
\begin{equation*}
\epsilon^{\frac{j-1}{2}}z\kappa^{-2}\bigtriangledown^{j_1}\log\kappa\cdots\bigtriangledown^{j_a}\log\kappa,\ \ \ \epsilon^{\frac{j-1}{2}}
\kappa^2\bigtriangledown^{j_1}\log\kappa\cdots\bigtriangledown^{j_a}\log\kappa,
\end{equation*}
where  $j_1>0,\cdots,j_a>0,\,j_1+\cdots+j_a=j$; and
\begin{equation*}
\epsilon^{\frac{j-1}{2}}\kappa^{-2}\bigtriangledown^{j_1}\log\kappa\cdots\bigtriangledown^{j_a}\log\kappa
\end{equation*}
where  $j_1>0,\cdots, j_a>0,\, j_1+\cdots+j_a=j-1$.
By Theorem \ref{solution},
\begin{equation}\label{20180301}
\parallel\!\bigtriangledown^{i}\log\kappa\!\parallel_{C^0}\leq C\epsilon^{-3i+2}.
\end{equation}
Combined with (\ref{20180201}), we have
\begin{equation*}
\Bigl|\bigtriangledown^j_\epsilon\frac{\partial\log B}{\partial\overline w_\epsilon}\Bigr|\leq C\epsilon^{\frac{-5j+3}{2}}+C\epsilon^{\frac{-5j+9}{2}}\kappa^{-2}.
\end{equation*}
Hence by Lemma \ref{20180901}, we have
\begin{equation}\label{20180403}
\Bigl|\!\Bigl|\bigtriangledown^j_\epsilon\frac{\partial\log B}{\partial \overline w_\epsilon}\Big|\!\Bigl|_{L^p}\leq C\epsilon^{\frac{-5j+3}{2}}.
\end{equation}
By (\ref{20180301}), we also have
\begin{equation}\label{20180404}
\Bigl|\!\Bigl|\bigtriangledown^j_\epsilon\frac{\partial\log B}{\partial\overline z_\epsilon}\Bigr|\!\Bigr|_{C^0}=\epsilon^{\frac{j+1}{2}}\Bigl|\!\Bigl|\bigtriangledown^j\frac{\partial\log B}{\partial\overline z}\Bigr|\!\Bigr|_{C^0}\leq C\epsilon^{\frac{-5j-1}{2}},
\end{equation}
\begin{equation}\label{20180502}
\Bigl|\!\Bigl|\bigtriangledown^j_\epsilon\frac{\partial^2\log B}{\partial z_\epsilon\partial\overline z_\epsilon}\Bigr|\!\Bigr|_{C^0}=\epsilon^{\frac{j+2}{2}}\Bigl|\!\Bigl|\bigtriangledown^j\frac{\partial^2\log B}{\partial z\partial\overline z}\Bigr|\!\Bigr|_{C^0}\leq C\epsilon^{\frac{-5j-6}{2}}.
\end{equation}


Equipped with the preparations, we begin to estimate
$\parallel\!\bigtriangledown_\epsilon^j\mathcal
H_\epsilon\!\parallel_{L^2(\mathcal U)}$ for $j\geq 1$. The approach
is standard.
We must be very careful
when dealing with $\epsilon$. We assume that $\mathcal H_\epsilon$
has a compact support in $\mathcal U_\alpha$, otherwise we can
shrink the open subsets $\mathcal U_\alpha$ to $\mathcal U_\alpha'$
so that they still form an open cover of $X$ and then use cut-off
functions as the proofs of Lemma \ref{20180708} and Proposition \ref{20180902}. Here we note that
we shrink $\mathcal U_\alpha$ to $\mathcal U_\alpha'$ by shrinking
$U_\alpha$ to $U_\alpha'$ in $B$ and hence the cut-off functions
$\chi$ can be taken only dependent on the variable $z$. Thus
$|\frac{\partial\chi}{\partial z_\epsilon}|=\epsilon^{\frac 1
2}|\frac{\partial\chi}{\partial z}|$ and
$\frac{\partial\chi}{\partial w_\epsilon}=0$. This is good enough
for us to do estimates.  We will omit the domain $\mathcal U$ of
integration. We will take $C$ as the generic constant which depends
on $l,k$, and $r_0$, etc.

We first estimate $\parallel\!\bigtriangledown_\epsilon\mathcal H_\epsilon\!\parallel_{L^2}$.
Since $\mathcal H_\epsilon$ is Hermitian symmetric,  by system (\ref{901}) and inequalities (\ref{20180202}) we have
\begin{eqnarray*}
\begin{aligned}
&\int|\!\bigtriangledown_\epsilon \mathcal H_\epsilon|^2=-\int
 \Tr (\mathcal H_\epsilon\bigtriangleup_\epsilon \mathcal
H_\epsilon)\\
=&-4\int\Tr\bigl(\mathcal
H_\epsilon\cdot(I_{1^2}+I_1+I_0+I_{-1})\bigr)\\
\leq &C\parallel\!\mathcal H_\epsilon\!\parallel_{C^0}\Bigl(\int|\!\bigtriangledown_\epsilon \mathcal H_\epsilon|^2
+\epsilon^{-\frac 1 2}\int|\!\bigtriangledown_\epsilon\mathcal H_\epsilon|+\parallel\! I_0\!\parallel_{C^0}+\parallel\!I_{-1}\!\parallel_{C^0}\Bigr).
\end{aligned}
\end{eqnarray*}
Since $l$ is positive, according to (\ref{01}), when $\epsilon$ is small enough,
$\parallel\!\mathcal H_\epsilon\!\parallel_{C^0}$ is very small and
hence the first term of right hand side can be controlled by the term of left hand side. By Cauchy's inequality, the second term is less than
\begin{equation*}
\frac 1 2 \int|\!\bigtriangledown_\epsilon \mathcal H_\epsilon|^2+C\epsilon^{-1}\parallel\!\mathcal H_\epsilon \!\parallel^2_{C^0}.
\end{equation*}
Hence, we have
\begin{equation*}
\int|\!\bigtriangledown_\epsilon\mathcal H_\epsilon|^2\leq C\parallel\!\mathcal H_\epsilon\!\parallel_{C^0}(\epsilon^{-1}\parallel\mathcal\!\mathcal H_\epsilon\!\parallel_{C^0}+\parallel\!I_0\!\parallel_{C^0}+\parallel I_{-1}\parallel_{C^0}).
\end{equation*}
Combined with (\ref{01}), (\ref{08}), and (\ref{171202}), at last we obtain when $l>1$
\begin{equation}\label{09}
\parallel\!\bigtriangledown_\epsilon\mathcal
H_\epsilon\!\parallel_{L^2}\leq C\epsilon^{l-\frac 3 2}.
\end{equation}

Next we  estimate $\parallel\!\bigtriangledown^2_\epsilon\mathcal H_\epsilon\!\parallel_{L^2}$.
When $\mathcal H_\epsilon$ has the
compact support, by the formula in Lemma \ref{20180708} we have
\begin{equation*}
\int|\!\bigtriangledown^2_\epsilon\mathcal H_\epsilon|^2=\int
|\bigtriangleup_\epsilon\mathcal
H_\epsilon|^2.
\end{equation*}
Then by system (\ref{901}), Cauchy's inequality, inequalities (\ref{20180202}), (\ref{08}), and (\ref{171202}),  we have
\begin{eqnarray}\label{b6}
\begin{aligned}
&\int |\!\bigtriangledown^2_\epsilon\mathcal H_\epsilon|^2 \leq C\int\bigl(|I_1|^2+|I_{1^2}|^2+|I_0|^2+|I_{-1}|^2\bigr)\\
& \quad\quad \leq C\int|\!\bigtriangledown_\epsilon\mathcal
H_\epsilon|^4+C\epsilon^{-1}\int|\!\bigtriangledown_\epsilon\mathcal
H_\epsilon|^2+C\epsilon^{2l-6}.
\end{aligned}
\end{eqnarray}

We need the Gagliardo-Nirenberg inequality:

\begin{lemm}\cite{Ni}\label{Ni}
Let $f\in L^p(\mathbb R^n)$, $D^m f\in L^q(\mathbb R^n)$, $1\leq
p,q\leq+\infty$. Then for any $i$ $(0\leq i\leq m)$,
there exists a constant $C$ depending only on $m,\,n,\,p,\,q$ and $i$ such that
\begin{equation*}
\parallel\! D^i f\!\parallel_{L^r(\mathbb R^n)}\leq C\parallel\!
f\!\parallel_{L^p(\mathbb R^n)}^{1-\frac i m}\parallel\!
D^mf\!\parallel_{L^q(\mathbb R^n)}^{\frac i m},
\end{equation*}
where
\begin{equation*}
\frac 1 r=\bigl(1-\frac i m\bigr)\frac 1 p+\frac i m\frac 1 q.
\end{equation*}
\end{lemm}

Using this lemma  for the case where  $n=4$, $i=1$, $r=4$, $m=2$, $q=2$,
and $p=+\infty$, we get
\begin{equation*}\label{11}
\int |\!\bigtriangledown_\epsilon \mathcal H_\epsilon|^4\leq C
\parallel\!\mathcal
H_\epsilon\!\parallel^2_{C^0}\int|\!\bigtriangledown^2_\epsilon\mathcal
H_\epsilon|^2.
\end{equation*}
Hence,  when $\epsilon$ is small enough, the first
term of the right hand side in (\ref{b6}) can be controlled by the term of the left hand side. Thus, by (\ref{09})
we obtain
\begin{equation}\label{12}
\parallel\!\bigtriangledown_\epsilon^2 \mathcal H_\epsilon\!\parallel_{L^2}\leq
C\epsilon^{l-3}.
\end{equation}

Based upon the proofs of inequalities (\ref{09}) and (\ref{12}), we use the inductive method to prove
that for any nonnegative integer $m$ and positive integer $l$ satisfying $l>\frac 5 2 m$ and for any sufficiently small
positive small $\epsilon$,
\begin{equation}\label{z0}
\parallel\!\bigtriangledown_\epsilon^m\mathcal
H_\epsilon\!\parallel_{L^2}<C\epsilon^{l-\frac 5 2 m}.
\end{equation}
It has been proved for $m=0,\,1,\,2$.
Assume that it holds for any $m\leq k-1$ where $k\geq 3$.
We should prove that it also holds for $m=k$.
We give the sketch as
follows. For convenience, we denote
\begin{equation*}
M(k)=\int|\!\bigtriangledown_\epsilon^k\mathcal
H_\epsilon|^2.
\end{equation*}
By the formula in Lemma \ref{20180708} we have
\begin{equation*}
M(k)\leq C\int |\bigtriangledown_\epsilon^{k-2}\bigtriangleup_\epsilon\mathcal H_\epsilon|^2.
\end{equation*}
Then by using the system (\ref{901}) and the basic inequality, we have
\begin{equation}\label{20180508}
\begin{aligned}
&M(k)\leq C\int|\bigtriangledown_\epsilon^{k-2}(I_{1^2}+I_1+I_0+I_{-1})|^2\\
\leq &C\int(|\bigtriangledown^{k-2}_\epsilon I_{1^2}|^2+|\bigtriangledown^{k-2}_\epsilon I_1|^2+|\bigtriangledown^{k-2}_\epsilon I_0|^2+|\bigtriangledown^{k-2}_\epsilon I_{-1}|^2).
\end{aligned}
\end{equation}
For convenience, we denote each term in the right hand side by $M_{1^2}(k)$, $M_1(k)$, $M_0(k)$, and $M_{-1}(k)$ respectively.

We make the following observations:
\begin{equation*}
M_{1^2}(k)\leq \sum C \int |\bigtriangledown_\epsilon^{i_1}\mathcal H_\epsilon|^2|\bigtriangledown_\epsilon^{i_2}\mathcal H_\epsilon|^2|\bigtriangledown^{i_3}_\epsilon\check H_\epsilon|^2,
\end{equation*}
where $i_1\geq i_2>0,\,i_3\geq 0,\,i_1+i_2+i_3=k$;
\begin{equation*}
M_1(k)\leq\sum  C\int |\bigtriangledown_\epsilon^{i_1}\mathcal H_\epsilon|^2|\bigtriangledown_\epsilon^{i_2}\check H_\epsilon|^2|\bigtriangledown^{i_3}_\epsilon\check H_\epsilon|^2|\bigtriangledown^{j}_\epsilon\log B|^2,
\end{equation*}
where $i_1>0,\,i_2\geq i_3\geq 0,\,j\geq 1,\,i_1+i_2+i_3+j=k$;
\begin{equation*}
\begin{aligned}
&M_0(k)\leq \sum C\int|\bigtriangledown^i_\epsilon\mathcal H_\epsilon|^2\Bigl|\bigtriangledown^{j}_\epsilon\frac{\partial^2\log B}{\partial z_\epsilon\partial\overline z_\epsilon}\Bigr|^2\\
&+\sum C \int |\bigtriangledown_\epsilon^{i_1}\mathcal H_\epsilon|^2|\bigtriangledown_\epsilon^{i_2}\check H_\epsilon|^2|\bigtriangledown^{i_3}_\epsilon\check H_\epsilon|^2|\bigtriangledown^{j_1}_\epsilon\log B|^2|\bigtriangledown^{j_2}_\epsilon\log B|^2,
\end{aligned}
\end{equation*}
where $i\geq 0,\, j\geq 0,\, i+j=k-2$ and $i_1\geq 0,\,i_2\geq i_3\geq 0,\,j_1\geq j_2\geq 1,\,i_1+i_2+i_3+j_1+j_2=k$;
\begin{equation*}
M_{-1}(k)=\sum \int |\bigtriangledown^i_\epsilon\check H_\epsilon|^2|\bigtriangledown^j_\epsilon(i\Lambda_{\omega_\epsilon}\hat\Theta_0)|^2,
\end{equation*}
where $i\geq 0,\, j\geq 0$ and $i+j=k-2$. Here all the constants $C$ depend on $k$; For convenience denote
$\bigtriangledown^0_\epsilon \check H_\epsilon=\check H_\epsilon$ and $\bigtriangledown^0_\epsilon\mathcal H_\epsilon=\mathcal H_\epsilon$.
We note that $\bigtriangledown^i_\epsilon\check H_\epsilon=\bigtriangledown^i_\epsilon \mathcal H_\epsilon$
if $i>0$.

Now we proceed to doing estimates.
By  Lemma \ref{Ni} for the case where $p=\infty,\,q=2,\,m=j$ and $1<i<j$, we have
\begin{equation}\label{a10}
\Bigl(\int |\!\bigtriangledown^{i}_\epsilon\mathcal
H_\epsilon|^{\frac{2j}{i}}\Bigr)^{\frac{i}{j}} \leq
C\parallel\!\mathcal
H_\epsilon\!\parallel_{C^0}^{2(1-\frac{i}{j})}\Bigl(\int|\!\bigtriangledown^{j}_\epsilon
\mathcal H_\epsilon|^2\Bigr)^{\frac{i}{j}}.
\end{equation}
We first estimate $M_{1^2}(k)$.  For its summand whose $i_3>0$, we use H\"older's inequality for $\frac{i_1}{k}+\frac{i_2}{k}+\frac{i_3}{k}=1$
and then use the above inequality for $j=k$ to find that  it is less than
\begin{eqnarray*}
\begin{aligned}
 &C\Bigl(\int|\!\bigtriangledown_\epsilon^{i_1}\mathcal
H_\epsilon|^{\frac{2k}{i_1}}\Bigr)^{\frac{i_1}{k}}\Bigl(\int|\!\bigtriangledown_\epsilon^{i_2}\mathcal
H_\epsilon|^{\frac{2k}{i_2}}\Bigr)^{\frac{i_2}{k}}\Bigl(\int|\!\bigtriangledown_\epsilon^{i_3}\mathcal
H_\epsilon|^{\frac{2k}{i_3}}\Bigr)^{\frac{i_3}{k}}
\leq C\parallel\!\mathcal
H_\epsilon\!\parallel_{C^0}^4\int|\!\bigtriangledown^k_\epsilon\mathcal
H_\epsilon|^2.
\end{aligned}
\end{eqnarray*}
For the same reason, a summand of $M_{1^2}(k)$ whose  $i_3=0$ is less
than
\begin{equation*}
C\parallel\!\mathcal
H_\epsilon\!\parallel_{C^0}^2\int|\!\bigtriangledown^k_\epsilon\mathcal
H_\epsilon|^2.
\end{equation*}
Hence, when $\epsilon>0$ is  small enough, $M_{1^2}(k)$ can
be controlled by $M(k)$:
\begin{equation}\label{20180402}
M_{1^2}(k)\leq C\parallel\!\mathcal H_\epsilon\!\parallel_{C^0}^2 M(k)
\end{equation}

Next we focus on  $M_1(k)$. By (\ref{20180202}) its summand whose  $j=1$ is less than
\begin{equation*}
C\epsilon^{-1}S(i_1,i_2,i_3),
\end{equation*}
where for convenience we have denoted
\begin{equation*}
S(i_1,i_2,i_3)=\int |\bigtriangledown_\epsilon^{i_1}\mathcal H_\epsilon|^2|\bigtriangledown_\epsilon^{i_2}\check H_\epsilon|^2|\bigtriangledown^{i_3}_\epsilon\check H_\epsilon|^2
\end{equation*}
where $i_1\geq 1$ and $i_1+i_2+i_3=k-1$.
Clearly $S(i_1,0,0)=M(k-1)$, $S(i_1, i_2,0)$ for $i_2>0$ or $S(i_1,i_2,i_3)$ for $i_3>0$
is a summand of $M_{1^2}(k-1)$ which by (\ref{20180402}) is less than
$C\parallel\!\mathcal H_\epsilon\!\parallel_{C^0}^2 M(k-1)$ or
$C\parallel\!\mathcal H_\epsilon\!\parallel_{C^0}^4M(k-1)$,
respectively, which is less than $CM(k-1)$ if $\epsilon>0$ is small enough.
Hence by inductive assumption (\ref{z0}) a summand of $M_1(k)$ whose $j=1$ is less than
\begin{equation*}
C\epsilon^{-1}\epsilon^{2l-5(k-1)}=C\epsilon^{2l-5k+4}\leq C\epsilon^{2l-5k}.
\end{equation*}
On the other hand, by H\"older's inequality a summand of $M_1(k)$ whose $j>1$ and $i_2=0$ (and also $i_3=0$)  is less than
\begin{equation}\label{20180501}
C\Bigl(\int|\bigtriangledown^{j}_\epsilon\log B|^{\frac{2(k-1)}{j-1}}\Bigr)^{\frac{j-1}{k-1}}
\Bigl(\int |\bigtriangledown_\epsilon^{i_1}\mathcal H_\epsilon|^{\frac{2(k-1)}{i_1}}\Bigr)^{\frac{i_1}{k-1}}
\end{equation}
since $i_1+j=k$. Hence by H\"older's inequality and (\ref{a10}) it is less than
\begin{equation*}
C\Bigl(\int|\bigtriangledown^{j}_\epsilon\log B|^{2(k-1)}\Bigr)^{\frac{1}{k-1}}
\parallel\!\mathcal H_\epsilon\!\parallel_{C^0}^{2(1-\frac{i_1}{k-1})}
\Bigl(\int |\bigtriangledown_\epsilon^{k-1}\mathcal H_\epsilon|^2\Bigr)^{\frac{i_1}{k-1}},
\end{equation*}
or by (\ref{20180403}) or (\ref{20180404}) is less than
\begin{equation}\label{20180905}
C\epsilon^{-5j+4}\parallel\!\mathcal H_\epsilon \!\parallel_{C^0}^{2(1-\frac{k-j}{k-1})}M(k-1)^{\frac{k-j}{k-1}},
\end{equation}
or by (\ref{01}) and inductive assumption (\ref{z0}) it is less than
 \begin{equation*}
C\epsilon^{-5j+4}\epsilon^{2l(1-\frac{k-j}{k-1})}\epsilon^{(2l-5(k-1))\frac{k-j}{k-1}}
=C\epsilon^{2l-5k+4}\leq C\epsilon^{2l-5k}.
\end{equation*}
By a similar method, we find that a summand of $M_1(k)$ whose $j>1$, $i_2>0$ and $i_3=0$ or $i_3>0$ can be controlled by (\ref{20180905}). In fact
it is less than
\begin{equation*}
C\epsilon^{-5j+4}\parallel\!\mathcal H_\epsilon \!\parallel_{C^0}^{2(a-\frac{k-j}{k-1})}M(k-1)^{\frac{k-j}{k-1}},
\end{equation*}
where $a=2$ if $i_3=0$ and $a=3$ if $i_3>0$.
From the above discussions, we see that
\begin{equation*}
M_1(k)\leq C\epsilon^{2l-5k}.
\end{equation*}

Thirdly we consider the terms in $M_0(k)$. For its term whose $i\geq 0,\,j\geq 0$ and $i+j=k-2$,  by inductive assumption (\ref{z0}) for $m=i$ and inequality (\ref{20180502}) we find that it is less than
\begin{equation*}
C\epsilon^{2l-5i}\epsilon^{-5j-6}=C\epsilon^{2l-5k+4}<C\epsilon^{2l-5k}.
\end{equation*}
By a similar method of estimate to a term in $M_1(k)$, we can prove that a term in $M_0(k)$ whose $i_1\geq 0,\,i_2\geq i_2\geq 0,\,j_1\geq j_2\geq 1$
is less than $C\epsilon^{2l-5k}$. We omit its proof here.
Therefore, we get the conclusion that
\begin{equation*}
M_0(k)\leq C\epsilon^{2l-5k}.
\end{equation*}

Now we arrive at the estimate to $M_{-1}(k)$. By (\ref{20180507}) the term in $M_{-1}(k)$ whose $i=0$ and $j=k-2$ is less than
$C\epsilon^{2l-3k+4}$. The term in $M_{-1}(k)$ whose $i>0$ and $j=k-i-2$ is less than
\begin{equation*}
C\epsilon^{2l-5i}\epsilon^{2l-3j-2}<C\epsilon^{4l-5k+2j+8}
\end{equation*}
by inductive assumption (\ref{z0}) for $m=i$ and (\ref{20180507}).
Hence we have
\begin{equation*}
M_{-1}(k)\leq C\epsilon^{2l-5k}.
\end{equation*}


In summary, we have proved that $M_{1^2}(k)$ can be controlled by $M(k)$, and  $M_1(k)$, $M_0(k)$ and $M_{-1}(k)$ are less than $C\epsilon^{2l-5k}$.
Hence from (\ref{20180508}) we obtain
\begin{equation*}
M(k)=\int|\!\bigtriangledown_\epsilon^k\mathcal H_\epsilon|^2\leq C\epsilon^{2l-5k}.
\end{equation*}
Thus, by the inductive method we have proved  inequality (\ref{z0}).

Now for a given positive integer $k$ and integer $l$ satisfying $l>\frac 5 2 (k+3)$,
when $\epsilon>0$ is small enough,
the Sobolev inequality and inequalities (\ref{z0}) for $0\leq m\leq k+3$
produce
\begin{equation*}
\parallel\!\mathcal H_\epsilon\!\parallel_{C_\epsilon^{k}(\mathcal U)}\leq C\sum_{m=0}^{k+3}
\parallel\!\mathcal \bigtriangledown^m_\epsilon\mathcal H_\epsilon\!\parallel_{L^2(\mathcal U)}\leq C\epsilon^{l-\frac  5 2(k+3)},
\end{equation*}
which by (\ref{20180509}) results in
\begin{equation*}
\parallel\!\mathcal H_\epsilon\!\parallel_{C^{k}(\mathcal U)}\leq C\epsilon^{l-3k-\frac {15}2 }.
\end{equation*}
Since $\mathcal H_\epsilon=\check H_{\epsilon}-I$, we have finished the proof of the theorem.
\end{proof}


\begin{thebibliography}{99}
\bibitem{Cr} Croke C. Some isoperimetric inequalities and
eigenvalue estimates. Ann Sci \'{E}cole Norm Sup, 1980, 13: 419--435

\bibitem{Do}Donaldson S K. Anti  self-dual Yang-Mills
connections over complex algebraic surfaces and stable vector
bundles. Proc London Math Soc, 1985, 50: 1--26

\bibitem{Do2} Donaldson S K. Infinite determinants, stable bundles and curvature. Duke Math J, 1987, 54: 231--247

\bibitem{Fr} Friedman R. Rank two vector bundles over regular elliptic
surfaces. Invent Math, 1989, 96: 283--332

\bibitem{FMW}Friedman R, Morgan J,  Witten E. Vector
bundles and $F$ theory.  Comm Math Phys,  1997, 187: 679--743

\bibitem{Fu0}Fukaya K. Mirror symmetry of abelian varieties and multi-theta functions.
J Algebraic Geom, 2002, 11: 393--512

\bibitem{Fu} Fukaya K. Multivalued Morse theory, asymptotic
analysis and mirror symmetry. In:  Graphs and Patterns in Mathematics
and Theoretical Physics. Proc Sympos Pure Math, vol. 73. Amer Math Soc, Providence, RI, 2005, 205--278

\bibitem{Gi} Gidas B, Ni W-M, Nirenberg L. Symmetry
and related properties via the maximum principle. Comm Math Phys,
1979, 68: 209--243

\bibitem{GT}Gilbarg D, Trudinger N S. Elliptic Partial Differential Equations of Second Order.
Berlin: Springer-Verlag,  2001

\bibitem{GH} Griffiths P, Harris J. Principles of Algebraic Geometry. New York: John Wiley and  Sons, Inc, 1994

\bibitem{GTZ}Gross M, Tosatti V, Zhang Y. Collapsing of abelian fibred Calabi-Yau manifolds.
Duke Math J, 2013: 162, 517--551

\bibitem{Gr}Gross M, Wilson P M H. Large complex
structure limites of $K3$ surfaces.  J. Differential Geom,  2000, 55: 475--546

\bibitem{Hit} Hitchin N J. The self-duality equations on a Riemann surface.
Proc London Math Soc, 1987, 55: 59--126

\bibitem{jost} Jost J. Partial differential equations.
Graduate Texts in Mathematics, vol. 214. New York: Springer, 2007

\bibitem{Ko}Kobayashi S. Differential Geometry of Complex
Vector Bundles. Princeton University Press and Iwanami Shoten,
1987

\bibitem{Kon}Kontsevich M. Homological algebra of mirror
symmetry. In: Proceedings of the International Congress of Mathematicians, vol. 1, 2 (Z\"urich, 1994). Basel: Birkh\"{a}user, 1995, 120--139

\bibitem{KS} Kontsevich M, Soibelman Y. Homological mirror symmetry
and torus fibrations. In: Symplectic Geometry and Mirror Symmetry (Seoul, 2000). River Edge: World Sci Publ, 2001, 203--263



\bibitem{La}Lang S. Fundamentals of Diophantine Geometry.
New York: Springer-Verlag, 1983

\bibitem{Le}Leung N C. Geometric aspects of mirror
symmetry (with SYZ for rigid CY manifolds). In: Second International
Congress of Chinese Mathematicians. New Stud Adv Math, vol.
4.  Somerville: Int Press, 2004, 305--342

\bibitem{LYZ}Leung N C, Yau S-T, Zaslow E. From special
Lagrangian to Hermitian-Yang-Mills via Fourier-Mukai transform.
Adv Theor Math Phys, 2000, 4: 1319--1341

\bibitem{Pli}Li P. On the Sobolev constant and the
p-spectrum of a compact Riemannian manifold. Ann Sci \'Ecole
Norm Sup, 1980, 13: 451--468

\bibitem{LYZ0}Loftin J, Yau S-T, Zaslow E. Affine manifolds, SYZ geometry and the ``Y'' vertex. J
Differential Geom, 2005, 71: 129--158

\bibitem{MSWW}Mazzeo R, Swoboda J, Weiss H, Witt F. Ends of the moduli space of Higgs bundles. Duke Math J, 2016, 165: 2227--2271

\bibitem{Ni}Nirenberg L. On elliptic partial differential equations. Ann
Scuola Norm Sup Pisa, 1959, 13: 115--162

\bibitem{RZ}Ruan W-D. Zhang Y. Convergence of Calabi-Yau manifolds. Adv Math, 2011, 228: 1543--1589

\bibitem{Sc}
Schryer N L. Solution of monotone nonlinear elliptic boundary
value problems. Numer Math, 1971/1972, 18: 336--344

\bibitem{Si}Simpson C T. Constructing variations of Hodge
structure using Yang-Mills theory and applications to
uniformization. J Amer Math Soc, 1988, 1: 867--918

\bibitem{Siu} Siu Y T. Lectures on Hermitian-Einstein Metrics for Stable Bundles and K\"ahler-Einstein Metrics. DMV Seminar, 8.
Basel: Birkh\"auser Verlag, 1987

\bibitem{SYZ} Strominger A, Yau S-T, Zaslow E. Mirror symmetry
is $T$-duality. Nuclear Phys B, 1996, 479: 243--259

\bibitem{Th} Thomas R P. Moment maps, monodromy and mirror manifolds. In: Symplectic Geometry and
Mirror Symmetry (Seoul, 2000). River Edge: World Sci Publ, 2001, 467--498

\bibitem{TY}Thomas R P, Yau S-T. Special Lagrangians, stable bundles and
mean curvature flow.  Comm Anal Geom,  2002, 10: 1075--1113

\bibitem{To}Tosatti V. Adiabatic limits of Ricci-flat K\"ahler metrics.
J Differential Geom, 2010, 84: 427--453

\bibitem{UY}Uhlenbeck K K, Yau S-T. On the existence
of Hermitian-Yang-Mills connections in stable vector bundles. Comm
Pure  Appl Math, 1986, 39-S: S257--S293

\bibitem{wells} Wells R O Jr. Differential Analysis on Complex
Manifolds.  With a new appendix by O. Garcia-Prada. Graduate Texts in Mathematics, vol. 65. New York: Springer, 2008

\bibitem{Wi} Wilson P M H. Metric limits of Calabi-Yau manifolds. In: The Fano Conference.
Turin: Univ. Torino, 2004, 793--804

\bibitem{Wit} Witten E. Mirror symmetry, Hitchin's equations, and Langlands
duality.  In: The many facets of geometry. Oxford:  Oxford Univiversity Press, 2010, 113--128

\bibitem{Yau78} Yau S-T. On the Ricci curvature of a compact
K\"ahler manifold and the complex Monge-Amp\`ere equation,
I.  Comm Pure Appl Math, 1978, 31: 339--411

\bibitem{Zh}Zharkov I. Limiting behavior of local Calabi-Yau metrics. Adv Theor Math Phys, 2004, 8: 395--420

\end{thebibliography}
\end{document}